\documentclass[11pt, reqno]{amsart}
\usepackage{amsmath}
\usepackage[margin=1.00in]{geometry}
\usepackage[foot]{amsaddr}
\usepackage{amssymb}
\usepackage{mathrsfs}
\usepackage{enumerate}         %
\usepackage{color}
\usepackage{cases}
\usepackage{url} 
\usepackage{amsthm}
\usepackage{hyperref}
\usepackage{bm}
\usepackage{xy}
\usepackage{enumitem}
\usepackage[gen]{eurosym}
\usepackage{graphicx}
\usepackage{comment}

\DeclareMathOperator*{\argmin}{arg\,min}

\usepackage[colorinlistoftodos]{todonotes}

\newcommand{\ind}[1]{\mathbf{1}_{\{#1\}}}
\newcommand{\llip}[1]{\mathrm{Lip}_{\mathrm{loc}}(#1)}

\theoremstyle{plain}
\newtheorem{theorem}{Theorem}[section]

\newtheorem{corollary}[theorem]{Corollary}

\newtheorem{lemma}[theorem]{Lemma}

\newtheorem{proposition}[theorem]{Proposition}

\newtheorem{assumptionalt}{Assumption}[theorem]
\newenvironment{assumptionp}[1]{
  \renewcommand\theassumptionalt{#1}
  \assumptionalt
}{\endassumptionalt}
\theoremstyle{remark}
\newtheorem{remark}[theorem]{Remark}
\newtheorem{example}[theorem]{Example}

\newlist{assumptionenum}{enumerate}{1}
\setlist[assumptionenum]{
    label=(\roman*),
    ref=\theassumptionalt.(\roman*)
}
\newlist{assumptioncases}{enumerate}{1}
\setlist[assumptioncases]{
    label=(\alph*),
    ref=\theassumptionalt.(\roman{assumptionenumi}).(\alph*)
}

\numberwithin{equation}{section}

\title[When should one stop the most exciting game?]{When should one stop the most exciting game?\\
{\tiny Sequential Inference for win-martingales}} 
\author{Steven Campbell$^a$ and Karl Kristian Engelund$^b$} %
\thanks{SC gratefully acknowledges support from an NSERC Postdoctoral Fellowship (PDF‑599675-2025) and a Columbia University CDFT Research Grant.}
\address{$^a$Dept.\ of Statistics, Columbia University, New York, NY, USA.}
\address{$^b$Dept.\ of Mathematical Sciences, University of Copenhagen, Copenhagen, Denmark.}
\email{kken@math.ku.dk}
\date{\today}

\begin{document}

\begin{abstract}

Prediction markets have become a prominent way of aggregating beliefs about binary future events, and their price processes are often interpreted as evolving win probabilities, or ``win-martingales.'' Motivated by this perspective and recent work on Aldous' ``most exciting game,'' we study when a decision maker should stop observing a win-martingale and make a decision about the outcome. In particular, we allow the true outcome to be revealed at a fixed finite horizon, as in a sports game or election. Under a general terminal loss and running cost, we reduce the Bayes risk to an optimal stopping problem for the win probability process. When the win-martingale is a diffusion and its volatility separates into a deterministic time factor and a state-dependent factor, a deterministic time change transforms the problem into one for a time-homogeneous diffusion with a generally time-inhomogeneous running cost.

Under explicit structural assumptions, we obtain a complete free-boundary characterization of the solution to the stopping problem in both finite and infinite horizons without discounting. We prove smooth-fit and $C^1$ regularity of the value function, $C^1$ regularity of the optimal stopping boundaries before the horizon, and derive a nonlinear integral equation that characterizes the boundaries uniquely. Taken together, these results yield a common decision-theoretic framework and solution theory for a broad class of posterior dynamics that includes the Aldous, Bass, and binary sequential-inference martingales as special cases. Our analysis requires no temporal monotonicity of the running cost and therefore accommodates highly nonmonotone stopping boundaries. In particular, we exhibit an example in which the optimal boundary has no limit as the calendar-time horizon is approached.
\end{abstract}
    
	\maketitle
	\vspace{-0.5cm}

	

\section{Introduction}
Many high-stakes decisions reduce to a single question: \emph{when should one stop observing and commit to an action?} Examples range from calling a default, distinguishing noise from informed trading, halting an A/B test, and closing a position, to any setting in which information arrives sequentially and delaying a decision is costly. In each case the decision maker must balance the value of additional information against the cost of waiting.

This paper studies this question in a binary setting. There is an unknown outcome, denoted by $\theta\in\{0,1\}$, and the decision maker observes information over time. The natural state variable is the posterior probability that the outcome equals one. Since this posterior is a conditional expectation, it evolves as a bounded martingale taking values in the unit interval. We refer to such processes as posterior win-martingales.

A useful motivation comes from prediction markets. In a prediction market for a binary event, the quoted price changes as information arrives. Abstracting from market frictions, strategic effects, and risk premia, one may, as argued by Wolfers and Zitzewitz \cite{wolfers2004prediction}, view this price process as a dynamically updated belief about the terminal outcome. This perspective separates two questions. 
First, how should one model the dynamics of an evolving win probability in a way that captures persistent uncertainty? Second, once such a martingale has been chosen, when should a decision maker stop observing it and act?

The first question is closely connected to Aldous' problem of finding the ``\emph{most exciting game}'' \cite{aldous}. In Aldous' formulation, one observes a game with a binary terminal outcome, and its excitement is encoded by the evolution of the win probability. A game is exciting when this probability continues to fluctuate and does not quickly settle near zero or one. Recent work on \emph{specific} distances between continuous-time martingales has given this idea a precise variational formulation. Specific relative entropy, introduced by Gantert in \cite{gantert1991einige}, and more recent Wasserstein-type variants provide ways of comparing continuous-time martingales even when their path laws are mutually singular; see, for example, \cite{backhoff2024exciting,backhoff2024most,backhoff2024specific,guo2025randomness}. Within this framework, variational problems over martingale laws select canonical win-martingales, including the Aldous martingale, Bass-type martingales, and the classical binary sequential-inference posterior.

This paper addresses the second question.  We take posterior win-martingales as models for evolving beliefs and study the decision-theoretic problem of when to declare a winner, subject to a running cost and a terminal loss associated with the declaration. 
For fixed running cost, terminal loss, and initial belief, the value function is the minimal Bayes risk and gives a decision-theoretic measure of how costly the game is to call optimally. 
The difference between the loss from stopping immediately and the optimal value measures the benefit of continued observation. In this sense, it provides a decision-dependent way to compare how difficult different posterior win-martingales are to call, and can be viewed as a complementary decision-theoretic notion of ``excitement.''

The classical version of this problem appears in Bayesian sequential testing \cite{shiryaev1967two}. There, a decision maker observes a signal about an unknown binary state and chooses when to stop and make a terminal decision; see also \cite{campbell2025softclass,ekstrom2022multi,gapeev2004wiener,gapeev2013bayesian,gapeev2017sequential}. In the Gaussian case, the posterior probability process is a sufficient statistic for the optimal decision rule, and the associated stopping problem has been studied extensively. 
Our motivation leads us to take a different approach: we take the win-martingale as given and allow for general running costs and terminal penalties arising from different terminal decision rules.   The problem of embedding a given win-martingale as the Bayesian posterior of an explicit sequential experiment is studied in our companion work \cite{campbell2026embedding}. A related Bernoulli--Doob representation of bounded martingale diffusions is developed in \cite{brigovrin2026}. The class of win-martingales we study naturally accommodates sequential decision problems in which the posterior reveals the true binary state at a prescribed finite horizon, as in a sports game or election.  A specific model with this feature was studied by
Lisovskii \cite{lisovskii2019bayesian}, who considered Bayesian sequential
testing between two possible pinning points of a Brownian bridge. To the best of our knowledge, our analysis provides the first general free-boundary and regularity theory for Bayesian sequential decision problems in which the true state is revealed at a fixed finite time.

We focus on developing the theory for posterior win-martingales whose volatility separates into a product of deterministic time and state-dependent factors.  The deterministic time factor describes the rate at which information arrives, and the state-dependent factor describes how the posterior fluctuates at a given belief level. This covers the canonical examples arising from the ``most exciting game'' and martingale-distance literature, while remaining broad enough to cover more general posterior dynamics. 

The decision problem is solved by first reducing it to an associated optimal stopping problem for the win-martingale. A deterministic time change then removes the time-dependent factor from the posterior dynamics and transfers it to the running cost. The original calendar-time cost is thereby converted into an ``operational time'' cost for observing a time-homogeneous martingale. Consequently, even monotone costs in the original time scale can lead to nonmonotone transformed costs and highly nonmonotone optimal stopping boundaries.

The optimal rule is characterized by a free boundary. While the posterior remains sufficiently uncertain, the decision maker continues observing. Once the posterior becomes sufficiently close to certainty, stopping and making the terminal declaration becomes optimal. We study three settings. In the homogeneous infinite-horizon problem, the transformed running cost is constant and the continuation region is an interval with two constant boundaries. We then analyze the time-inhomogeneous problem on finite and infinite horizons without introducing discounting. Under smoothness, symmetry, and shape assumptions on the problem inputs, the time slices of the continuation region are intervals of the form $(1-b_T(t),b_T(t))$ for a free boundary $b_T$. We prove that the value function is $C^1$ globally and $C^{1,2}$ away from the boundary. We also prove that $b_T$ is $C^1$ before the terminal horizon and derive a nonlinear integral equation that characterizes it uniquely.

The time-inhomogeneous analysis does not assume that the running cost is monotone. This allows for shrinking, expanding, and nonmonotone continuation regions. It also permits the treatment of an oscillatory example in which the calendar-time boundary has no limit as the terminal time is reached. The finite-horizon regularity analysis builds on the probabilistic methods of De Angelis and Stabile \cite{de2019lipschitz} and De Angelis and Peskir \cite{de2020global}. After establishing local Lipschitz regularity and smooth fit in the present two-sided boundary setting, we verify the conditions of De Angelis and Lamberton \cite{de2024probabilistic} to lift the boundary regularity to $C^1$. Estimates formed uniformly in the horizon then allow us to pass to the infinite-horizon problem. We also obtain probabilistic representations of the derivatives. For the temporal derivative on finite horizons, we use a technique similar to that of Jaillet, Lamberton, and Lapeyre \cite{jaillet1990variational}.

The contributions of this paper are threefold. First, it places a large class of win-martingales, including those selected by recent martingale-distance problems, into a common sequential decision framework. For fixed running and terminal costs, this provides a decision-theoretic comparison of how difficult the corresponding games are to call. Second, it builds on modern regularity theory to develop a free-boundary theory for two-sided optimal stopping problems with nonmonotone time-inhomogeneous running costs. The analysis covers finite and infinite horizons without discounting and yields both boundary regularity and a unique nonlinear integral-equation characterization. Third, we extend related results in Bayesian sequential testing. Relative to Ekstr\"om and Vaicenavicius \cite{ekstrom2015bayesian}, we allow a broader class of smooth terminal decision costs, while relative to Campbell and Zhang \cite{campbell2025softclass}, we allow more general posterior martingales. In both cases, our framework permits time-inhomogeneous running costs without monotonicity assumptions.

The remainder of the paper is organized as follows. Section \ref{sec:motivation} introduces the motivating games, reduces the decision-theoretic problem to optimal stopping of the posterior, and performs the deterministic time change. Section \ref{sec:setup} states the transformed problem and its assumptions, introduces useful transformations, and provides a preliminary analysis. Section \ref{sec:main_results} presents the main structural and regularity results whose proofs are developed in the subsequent sections. Sections \ref{sec:homogeneous} and \ref{sec:inhomogeneous} treat the homogeneous and inhomogeneous problems. Section \ref{sec:integral_eqs} derives the integral-equation characterization. Section \ref{sec:examples} presents examples and extensions, and Section \ref{sec:conclusion} concludes.

\section{A class of motivating games}\label{sec:motivation}

We begin with a broad class of sequential prediction problems on a time horizon $S\in(0,\infty)$. Consider a filtered probability space $(\Omega,\mathcal{F},(\mathcal{G}_s)_{s\in[0,S]},\mathbb{P})$ satisfying the usual conditions.
A game produces a $\mathcal{G}_S$--measurable binary outcome $\theta\in\{0,1\}$, where $\theta=1$ indicates a win and $\theta=0$ a loss. 
As the game unfolds, the decision maker observes information over time and updates their belief about the eventual outcome. 
The resulting posterior process
\[
    \Pi_s:=\mathbb{P}(\theta=1\mid\mathcal{G}_s), \qquad s\in[0,S],
\]
will be referred to as the \emph{win-martingale}. 
Thus $\Pi_s$ represents the conditional win probability at time $s$, given the information available up to that time, as encoded in the filtration $(\mathcal{G}_s)_{s\in[0,S]}$. 

Our objective is to decide \emph{when} to stop observing the game and \emph{what} prediction to make at that time. 
This leads naturally to a sequential decision problem with two components: a stopping rule and a terminal declaration.
Formally, a decision rule is a pair $(\nu,d)$, where $\nu\leq S$ is a $(\mathcal{G}_s)_{s\in[0,S]}$--stopping time and $d$ is a $D$-valued $\mathcal{G}_\nu$--measurable declaration, 
where the declaration set is a compact set $D\subseteq[0,1]$. 
This includes both hard classification, $D=\{0,1\}$, and possibly soft classification, where one reports a probability level, e.g. $D=[0,1]$.

To quantify performance, let $f:\mathbb{R}_+\to(0,\infty)$ denote the measurable running cost of observation, and let
\[
    \ell:\{0,1\}\times D\to[0,\infty]
\]
be a measurable loss function penalizing the final declaration. 
The associated Bayes risk of a decision rule $(\nu,d)$ is
\begin{equation}\label{eq:motivation_risk}
    \mathbb{E}\!\left[\int_0^\nu f(s)\,ds+\ell(\theta,d)\right].
\end{equation}
The aim is to minimize \eqref{eq:motivation_risk} over all admissible decision rules for which the expectation is finite to obtain the Bayes rule.

The posterior structure allows the terminal decision to be reduced to a function of the current belief alone.
Indeed, if $\nu$ is fixed and $d$ is $\mathcal{G}_\nu$--measurable, then the tower property yields
\[
    \mathbb{E}\!\left[\ell(\theta,d)\right]
    =
    \mathbb{E}\!\left[h(\Pi_\nu,d)\right],
\]
where
\begin{equation}\label{eq:h_def}
    h(\pi,d):=\pi\,\ell(1,d)+(1-\pi)\,\ell(0,d),
\end{equation}
using the convention $0\cdot\infty=0$.
For each posterior level $\pi\in[0,1]$, the quantity $h(\pi,d)$ is the conditional expected loss of declaring $d$. 
If $d\mapsto h(\pi,d)$ is lower semicontinuous, then, since $\pi\mapsto h(\pi,d)$ is lower semicontinuous by \eqref{eq:h_def}, and hence, in particular, measurable, one may select a measurable (see e.g. \cite[Theorem 14.37]{rockafellar1998variational}) optimal declaration
\[
    j(\pi)\in\argmin_{d} (h(\pi,d)+\delta_{D}(d)),
\]
and a measurable optimal terminal cost
\begin{equation*}
    g(\pi):=\inf_{d} (h(\pi,d)+\delta_{D}(d))=h(\pi,j(\pi)),
\end{equation*}
where $\delta_{D}(d)=0$ if $d\in D$ and $\infty$ otherwise.
Consequently, the original sequential decision problem reduces to the optimal stopping problem
\begin{equation}\label{eq:motivation_OSP_pi}
    \inf_{(\nu,d)}\mathbb E\left[\int_0^\nu f(s)\,ds+\ell(\theta,d)\right]
=\inf_{\nu\leq S}\mathbb{E}\left[\int_0^\nu f(s)\,ds+g(\Pi_\nu)\right],
\end{equation}
where we seek an optimal stopping time $\nu^*$ realizing the minimum such that we obtain the Bayes rule $(\nu^*,j(\Pi_{\nu^*}))$ minimizing the Bayes risk in \eqref{eq:motivation_risk}. 

This formulation covers a range of standard choices. However, our attention will primarily be on the following examples.
Let $D=[0,1]$ and consider the squared loss,
\[
    \ell_{L^2}(\theta,d)=(\theta-d)^2.
\]
Then the optimal declaration is $j(\pi)=\pi$ and the optimal terminal cost becomes
\[
    g_{L^2}(\pi)=\pi(1-\pi).
\]
Similarly, if $D=[0,1]$ with cross-entropy loss,
\begin{equation*}
    \ell_{\text{CE}}(\theta,d)=-\theta\log (d)-(1-\theta)\log (1-d),
\end{equation*}
then the optimal declaration is $j(\pi)=\pi$ and
the optimal terminal cost becomes
\begin{equation*}
    g_{\text{CE}}(\pi)=-\pi\log(\pi)-(1-\pi)\log(1-\pi).
\end{equation*}
More generally, different terminal decision rules simply lead to different choices of the terminal cost $g$. One may, for example, restrict the declarations to soft classification rules as in \cite{campbell2025softclass}. We lay out the assumptions on $g$ in Section \ref{subsec:standing_assumptions}. 

Concrete examples of win-martingales include a family of canonical posterior martingales driven by Brownian noise and pinned to the binary outcome $\theta$ at time $S$, stemming from the literature on specific distances between continuous-time martingales. For instance, specific relative entropy selects the \emph{Aldous} martingale $\Pi^A$, whose dynamics are
\begin{equation*}
    d\Pi^A_s = \frac{\sin(\pi \Pi^A_s)}{\pi\sqrt{S-s}}\,dW_s,\qquad s\in[0,S),
\end{equation*}
while suitable choices of specific $p$--Wasserstein distances produce the \emph{Bass} martingale $\Pi^B$ \cite{backhoff2024gradient,backhoff2023bass} and the classical \emph{binary sequential-inference} martingale $\Pi^I$:
\begin{equation*}
    d\Pi_s^B=\frac{\varphi(\Phi^{-1}(\Pi_s^B))}{\sqrt{S-s}}\,dW_s,\qquad
    d\Pi_s^I=\frac{\Pi_s^I(1-\Pi_s^I)}{\sqrt{S-s}}\,dW_s.    
\end{equation*}
Here $W$ is a Brownian motion and $\varphi,\Phi$ are the standard normal pdf/cdf. The Aldous martingale is often presented as the win probability in the ``\emph{most exciting game},'' in the sense that it maximizes ``excitement'' (equivalently, uncertainty) about the binary outcome under an entropy metric. 

To connect this class of games with the analysis in the sequel, we now specialize to win-martingales with separable volatility of the form
\begin{equation}\label{eq:motivation_pi_sde}
    d\Pi_s=\rho(s)\sigma(\Pi_s)\,dW_s, \qquad \Pi_0=x\in(0,1),
\end{equation}
for a Brownian motion $W$, a continuous deterministic time factor $\rho:[0,S)\to(0,\infty)$, and a state-dependent volatility $\sigma$ such that \eqref{eq:motivation_pi_sde} admits a strong $[0,1]$--valued solution. 
At this stage we do not impose detailed assumptions on $\sigma$; these will be stated in the next section. 
The key point is that \eqref{eq:motivation_pi_sde} separates the time-inhomogeneous and state-dependent features of the posterior dynamics. The win-martingale $\Pi$ is the only state process and so we will, for simplicity, assume $(\mathcal{G}_s)_{s\in[0,S]}$ is the filtration generated by $\Pi$ augmented to satisfy the usual conditions.

This structure makes a time change natural.
Define
\begin{equation}\label{eq:A_motivation}
    A(s):=\int_0^s \rho^2(v)\,dv,
\end{equation}
which we assume to be finite for all $s\in[0,S)$, and let
\[
    T:=\lim_{s\nearrow S}A(s)\in(0,\infty].
\]
Since $A$ is strictly increasing, we denote its inverse by $\Gamma:[0,T)\to[0,S)$ and define the time-changed process
\begin{equation*}
    X_t:=\Pi_{\Gamma(t)}, \qquad t<T.
\end{equation*}
Under the standing assumptions introduced later, the process $X$ solves the
time-homogeneous SDE
\begin{equation}\label{eq:X_motivation}
    dX_t=\sigma(X_t)\,dB_t
\end{equation}
for a Brownian motion $B$ with respect to the operational filtration
$(\mathcal{F}_t)_{t<T}$, where
$\mathcal{F}_t=\mathcal{G}_{\Gamma(t)}$ with the usual augmentation. We use
the convention $\Gamma(T)=S$. Thus, the deterministic time factor has
been absorbed into the clock, while the state dependence remains in the
diffusion coefficient.

If $T<\infty$, the preterminal process $X$ has the continuous extension
\[
    X_T:=\lim_{t\uparrow T}X_t
    =\lim_{s\uparrow S}\Pi_s
    =\Pi_{S-}.
\]
In general, this left limit need not equal $\Pi_S=\theta$.

The objective in \eqref{eq:motivation_OSP_pi} transforms accordingly for
stopping rules that occur before the terminal time. The running cost
becomes
\begin{equation}\label{eq:c_motivation}
    c(t):=f(\Gamma(t))\Gamma'(t).
\end{equation}
If $\nu<S$ almost surely and $\tau=A(\nu)$, then $\tau<T$ almost surely.
Conversely, a stopping time $\tau<T$ corresponds to
$\nu=\Gamma(\tau)<S$. The time change therefore gives a one-to-one
correspondence between preterminal stopping rules and preserves their
costs.

For the canonical examples above, $T=\infty$ and
$X_\infty=\theta$ almost surely. With the conventions $A(S)=\infty$ and
$\Gamma(\infty)=S$, the correspondence then also includes terminal
stopping. The time-changed problem is
\begin{equation}\label{eqn:motivating.obj.X}
    \inf_{\tau\geq0}
    \mathbb{E}\left[
        \int_0^\tau c(u)\,du+g(X_\tau)
    \right],
\end{equation}
where the infimum is taken over extended
$(\mathcal{F}_t)_{t\geq0}$--stopping times for which the expectation is
finite.

We also study finite-horizon problems of the form \eqref{eqn:motivating.obj.X}.  Such a horizon may arise because $T<\infty$ or because an exogenous decision deadline is imposed before $T$. In either case, we formulate the problem using only the information contained in $X$. In particular, when the deadline coincides with the original terminal time $S$ and $T=A(S)<\infty$, we tacitly assume that no additional information is revealed at $S$, so that $\Pi_{S-}=\Pi_S$. We set aside the exceptional case in which additional information is revealed discontinuously at the terminal time, so that $\Pi_{S-}\neq\Pi_S$. Such a terminal information jump leads to a different decision problem that is not considered here.

More generally, by restarting the problem from time $t$ and utilizing the strong Markov property, one is led to the family
\begin{equation}\label{eq:motivation_OSP_X_shifted}
    V_T(t,x):=
    \inf_{\tau\leq T-t}\mathbb{E}\!\left[\int_0^\tau c(t+u)\,du+g(X_\tau^x)\right],
\end{equation}
where $X^x$ solves \eqref{eq:X_motivation} with $X^x_0=x$, and the infimum is taken over stopping times with respect to the usual augmentation of the filtration generated by $X^x$. The constraint $\tau \leq T-t$ is interpreted as $\tau\leq\infty$ when $T=\infty$, and we minimize over stopping times for which the expectation is finite.

Thus, sequential prediction in a game with binary outcome naturally leads to an optimal stopping problem for a posterior martingale, and for the class \eqref{eq:motivation_pi_sde} this problem can be rewritten, after a deterministic time change, as an optimal stopping problem for a time-homogeneous diffusion. 
The next section takes the transformed problem variables $(\sigma,c,g)$ as primitives and formalizes the assumptions under which the transformed problem is well posed.

\section{Setup and Preliminary Analysis}\label{sec:setup}

In this section we formalize the class of optimal stopping problems studied throughout the paper, either stemming from motivating problems such as those in Section \ref{sec:motivation} or in their own right. We first introduce the underlying state process and the associated infinite-horizon stopping problem. We then collect the assumptions and present different formulations of associated problems used in the subsequent analysis. Finally, we derive a number of preliminary analytical results. 

\subsection{Problem formulation}\label{subsec:problem_formulation}

Let $(\Omega,\mathcal{F},\mathbb{P})$ be a probability space supporting a standard Brownian motion $B=(B_t)_{t\ge0}$, and let $\mathbb{F}=(\mathcal{F}_t)_{t\ge0}$ denote the usual augmentation of the natural filtration of $B$.

For each $x\in(0,1)$, we consider a process $X^x=(X_t^x)_{t\ge0}$ solving the stochastic differential equation
\begin{equation}\label{eq:general_X_SDE}
    dX_t^x=\sigma(X_t^x)\,dB_t,\qquad X_0^x=x.
\end{equation}
We will impose assumptions below ensuring that \eqref{eq:general_X_SDE} admits a unique strong solution and that the boundary points $\{0,1\}$ are inaccessible in finite time. For $x\in\{0,1\}$ we extend the definition by letting the process remain constant, that is, $X_t^x=x$ for all $t\ge0$.

For $x\in(0,1)$, under these assumptions $X^x$ is a time-homogeneous martingale diffusion taking values in $[0,1]$ and absorbed at the endpoints. Since $\sigma$ depends on the state alone, in line with the terminology of \cite{campbell2026embedding}, we call $X^x$ an \emph{autonomous win-martingale}. However, since autonomy is assumed throughout the remainder of the paper, we suppress the qualifier ``autonomous'' in what follows.

We let $\mathcal{L}$ denote the infinitesimal generator associated with $X^x$, given by
\begin{equation*}
    \mathcal{L}=\frac{1}{2}\sigma^2(x)\partial_{xx}.
\end{equation*}

Let $c:\mathbb{R}_+\to(0,\infty)$ be a measurable running cost and let $g:[0,1]\to\mathbb{R}_+$ be a bounded measurable terminal cost. The infinite-horizon optimal stopping problem is
\begin{equation}\label{eq:inf_horizon_osp}
    V(t,x):=\inf_{\tau}\mathbb{E}\left[\int_0^\tau c(t+u)\,du+g(X_\tau^x)\right],
\end{equation}
where the infimum is taken over all $\mathbb{F}$--stopping times $\tau$ such that the expectation is finite. We denote this set by $\mathcal{T}^{(t)}$. In particular, since $g$ is bounded, every admissible stopping time $\tau\in\mathcal{T}^{(t)}$ satisfies
\begin{equation*}
    \mathbb{E}\left[\int_0^\tau c(t+u)\,du\right]<\infty.
\end{equation*}
Similarly, we also define $\mathcal{T}$ as the set of stopping times such that
\begin{equation*}
        \mathbb{E}\left[\int_0^\tau c(u)\,du\right]<\infty.    
\end{equation*}
It clearly holds that $\mathcal{T}^{(t)}\subseteq \mathcal{T}$ when $c$ is integrable on $[0,T_1]$ for any $T_1\geq0$.

We denote the state space by
\[
    \mathcal{S}:=\mathbb{R}_+\times(0,1).
\]
Furthermore, we define the continuation and stopping regions by
\begin{equation}\label{eq:stop_and_cont_sets}
    \mathcal{C}:=\{(t,x)\in\mathcal{S}:V(t,x)<g(x)\},\qquad
    \mathcal{D}:=\{(t,x)\in\mathcal{S}:V(t,x)=g(x)\}.
\end{equation}
For convenience we also define the extended state space by
\begin{equation*}
    \overline{\mathcal{S}}=\mathbb{R}_+\times[0,1].
\end{equation*}
However, we consider points $(t,x)$ to be in $\mathcal{S}$ unless explicitly stated.

\subsection{Assumptions}\label{subsec:standing_assumptions}
We now collect the assumptions used throughout the paper. 
We separate the baseline assumptions needed for the basic formulation and qualitative structure from the extended regularity assumptions needed for the finer analysis of the time-inhomogeneous problem.

\begin{assumptionp}{A}[Baseline assumptions]\label{asmp:baseline}
Assume the following.

\begin{assumptionenum}
    \item\label{asmp:state_dynamics}
    The function $\sigma:[0,1]\to\mathbb{R}_+$ is $C([0,1])$, $\llip{(0,1)}$, and satisfies $\sigma(x)=0$ if and only if $x\in\{0,1\}$.
    Moreover, $\sigma^{-2}\in L^1_{\mathrm{loc}}((0,1))$ and, for any $a\in(0,1)$,
    \begin{equation}\label{eq:feller_test}
        \int_a^x\int_a^y \frac{1}{\sigma^2(z)}\,dz\,dy\to\infty
        \qquad\text{as }x\to0\text{ and as }x\to1.
    \end{equation}

    \item\label{asmp:base_penalty:c}
    The running cost $c:\mathbb{R}_+\to(0,\infty)$ is $C^1([0,\infty))$ and there exists a constant $K>0$ such that
    \begin{equation}\label{eq:c_log_derivative_bound}
        \left|\frac{c'(t)}{c(t)}\right|\le K,\qquad t\ge0.
    \end{equation}

    \item\label{asmp:base_penalty:g}
    The terminal cost $g:[0,1]\to\mathbb{R}_+$ is $C([0,1])$, $C^2((0,1))$, concave, and $g(x)=0$ if and only if $x\in\{0,1\}$.

    \item\label{asmp:Lg:unimodal}
    There exists $x_0\in(0,1)$ such that $\mathcal{L}g$ is strictly decreasing on $(0,x_0)$ and strictly increasing on $(x_0,1)$. 

    \item\label{asmp:Lg:goodset_nonempty} For all $t\geq0$, it holds that
    \begin{equation*}
        \mathcal{L}g(x_0)+c(t)<0.
    \end{equation*}
    
\end{assumptionenum}
\end{assumptionp}

Assumption \ref{asmp:state_dynamics} fixes the class of posterior diffusions considered in the paper, namely, posterior diffusions with unique strong solutions such that the process stays in $(0,1)$ in finite time, see Proposition \ref{prop:state_process}. Assumptions \ref{asmp:base_penalty:c} and \ref{asmp:base_penalty:g} state the basic assumptions on the running cost and the terminal cost. The concavity of $g$ and its vanishing on the boundary in Assumption \ref{asmp:base_penalty:g} are natural in the processes arising from Section \ref{sec:motivation}. Indeed, since $\pi\mapsto h(\pi,d)$ is affine for every fixed $d$, the induced terminal cost $g$ is the pointwise minimum of affine functions, and hence concave.

The second part of Assumption \ref{asmp:base_penalty:c}
gives an exponential bound for the growth of the cost. Namely, we have $\lvert\partial_t \log(c(v))\rvert\leq K$ for $v\geq0$ and integration gives
\begin{equation*}
    \left\lvert \log\frac{c(v+u)}{c(v)}\right\rvert\leq \lvert u\rvert K
\end{equation*}
where $v+u\geq 0$. By exponentiation, we obtain the exponential growth condition 
\begin{equation}\label{eq:expGrowth}
    c(v)e^{-K\lvert u\rvert }\leq c(v+u)\leq c(v)e^{K\lvert u\rvert }.
\end{equation}
Thus, we get the local (in $u$) estimate
\begin{equation}\label{eq:cprime_helper}
    |c'(v+u)|\leq K c(v+u)\leq K'c(v),
\end{equation}
where for, say, $|u|\leq1$, $K':=Ke^K$.
Equation \eqref{eq:c_log_derivative_bound} enables control of the running cost useful for basic continuity results such as Proposition \ref{prop:continuity_fixed_stopping}.

Assumption \ref{asmp:Lg:unimodal} ensures a unimodal structure of the stopping problem. Namely, the closer the process is to the boundary $\{0,1\}$, the more likely it should be that we stop. This can be realized through the Lagrange formulation of the problem given in Proposition \ref{prop:lagrange}. Finally, Assumption \ref{asmp:Lg:goodset_nonempty} ensures a nonempty continuation set, while also ensuring a bounded running cost.

For the time-inhomogeneous problem we impose the following additional assumptions.

\begin{assumptionp}{B}[Extended assumptions]\label{asmp:extended}
In addition to Assumption \ref{asmp:baseline}, assume the following:

\begin{assumptionenum}
    \item\label{asmp:sigma_g}
    The coefficients satisfy $\sigma\in C^2((0,1))$ and $g\in C^3((0,1))$.

    \item\label{asmp:dxLg}
    For the point $x_0$ appearing in Assumption \ref{asmp:Lg:unimodal},
    \[
        \partial_x\mathcal{L}g(x)<0 \quad\text{on }(0,x_0),
        \qquad
        \partial_x\mathcal{L}g(x)>0 \quad\text{on }(x_0,1).
    \]

    \item\label{asmp:bounded}
    Either
    \begin{assumptioncases}
        \item\label{asmp:bounded:g'} $\|g'\|_\infty<\infty$, or
        \item\label{asmp:bounded:c} there exists $m>0$ such that $c(t)\ge m$ for all $t\ge0$.
    \end{assumptioncases}
    
    \item\label{asmp:sym}
    The functions $\sigma$, $g$, and $\mathcal{L}g$ are symmetric around $1/2$.

    \item\label{asmp:martingale}
    The stochastic flow derivative $(\partial_xX_t^x)_{t\in[0,T_1]}$, defined below in \eqref{eq:sde_partialx}, is a true martingale for every $x\in(0,1)$ and $T_1<\infty$.
\end{assumptionenum}
\end{assumptionp}

Assumption \ref{asmp:sigma_g} ensures enough differentiability to apply modern regularity techniques in optimal stopping. Assumption \ref{asmp:dxLg} strengthens Assumption \ref{asmp:Lg:unimodal}, removing stationary points. 
The assumptions in Assumption \ref{asmp:bounded} are used to make derivative estimates of the value function finite.

Our main interest lies in symmetric games, and Assumption \ref{asmp:sym} will simplify several later arguments. Most of the analysis can be directly extended to asymmetric problems, but doing so requires additional bookkeeping.
Finally, Assumption \ref{asmp:martingale} can be verified in several standard ways. For many examples, such as the classical sequential testing problem with $\sigma(x)=x(1-x)$, the condition $\|\sigma'\|_\infty<\infty$ is enough to verify Novikov's condition. For other examples, such as the Bass win-martingale, Novikov's condition may be less tractable. A useful alternative sufficient condition is linear growth of the Lamperti drift given in \eqref{eq:lamperti_drift},
\begin{equation*}
    |\mu(l)|\le k(1+|l|),
\end{equation*}
for some constant $k>0$, which implies the martingale property by a Bene\v{s}-type argument; see e.g. \cite[Theorem 4.1]{Klebaner2014}. This indeed holds for the Bass win-martingale.

\subsection{Basic consequences of the setup}\label{subsec:basic_consequences}

We first record the basic consequences of the setup and basic assumptions.

\begin{proposition}\label{prop:state_process}
Under Assumption \ref{asmp:state_dynamics}, for every $x\in(0,1)$ the SDE \eqref{eq:general_X_SDE} admits a unique strong solution. Moreover, the boundary points $\{0,1\}$ are inaccessible in finite time and 
\[
    X_\infty^x:=\lim_{t\to\infty}X_t^x\in\{0,1\}
\]
almost surely.
\end{proposition}

\begin{proof}
Since $\sigma\in\llip{(0,1)}$, a standard localization argument yields existence and pathwise uniqueness of $X^x$ up to the first exit time from $(0,1)$. The Feller test \eqref{eq:feller_test} implies that the boundary points are inaccessible in finite time. Since $X^x$ is a bounded martingale in $[0,1]$, it converges almost surely to a random variable in $[0,1]$. Since the process is in natural scale, the scale function is finite at $0$ and $1$, so the limit must lie in $\{0,1\}$. See e.g. \cite[Chapter 5]{karatzas2012brownian}.
\end{proof}

As with the extension of $X^x$ to $x\in\{0,1\}$, we also extend the definition of the value function to the extended state space $\overline{\mathcal{S}}$ by setting
\begin{equation*}
    V(t,x)\equiv 0
\end{equation*}
on $(t,x)\in\overline{\mathcal{S}}\setminus \mathcal{S}$ under Assumption \ref{asmp:base_penalty:g} since $g(x)=0$ if $x\in\{0,1\}$.
Let $\mathbb{F}^x:=(\mathcal{F}_t^x)_{t\ge0}$ denote the usual augmentation of the natural filtration of $X^x$.

\begin{proposition}\label{prop:filtration}
Under Assumption \ref{asmp:state_dynamics}, for every $x\in(0,1)$, the filtrations $\mathbb{F}$ and $\mathbb{F}^x$ coincide.
\end{proposition}

\begin{proof}
Fix a $t<\infty$ and $x\in(0,1)$. Since $X_s^x\notin\{0,1\}$ for all $s\leq t$ almost surely, it follows that $\sigma(X_s^x)\neq0$, and thus
\[
    B_t=\int_0^t \frac{1}{\sigma(X_s^x)}\,dX_s^x
\]
by \eqref{eq:general_X_SDE}. Hence $B_t$ is $\mathcal{F}_t^x$--measurable. Conversely, $X^x$ is adapted to $\mathbb{F}$ by \eqref{eq:general_X_SDE}. Hence, the raw filtration generated by $B$ is contained in the raw filtration generated by $X^x$,
and therefore their usual augmentations coincide.
\end{proof}

\begin{proposition}\label{prop:joint_continuity}
    Under Assumption \ref{asmp:state_dynamics}, there exists a jointly measurable version of $(t,x,\omega)\mapsto X^x_t(\omega)$ such that on compacts of $\mathcal{S}$,
    \begin{equation*}
        (t,x)\mapsto X^x_t
    \end{equation*}
    is continuous almost surely.
\end{proposition}
\begin{proof}
    Since $\sigma\in\llip{(0,1)}$, the SDE \eqref{eq:general_X_SDE} admits a unique strong solution up to the first exit time from $(0,1)$. Furthermore, the solution has continuous paths. By Proposition \ref{prop:state_process}, the exit time is infinite due to the Feller condition. Using \cite[Theorem V.37]{protter2005stochastic}, the flow $x\mapsto X^x_t$ is continuous in the topology of uniform convergence on compacts, almost surely. Thus, joint continuity holds almost surely. Since $(t,x,\omega)\mapsto X^x_t(\omega)$ is jointly almost surely continuous in the first two arguments and measurable in the last argument, a standard modification argument yields a Carathéodory function. Thus, it is jointly measurable in $(t,x,\omega)$.
\end{proof}

\begin{proposition}\label{prop:Stopping_set}
    Under Assumption \ref{asmp:base_penalty:c}, the set of admissible strategies obeys $\mathcal{T}^{(t)}=\mathcal{T}$.
\end{proposition}
\begin{proof}
    We already have $\mathcal{T}^{(t)}\subseteq\mathcal{T}$. For the converse, take $\tau\in\mathcal{T}$ and any $t\geq0$. Then, by \eqref{eq:expGrowth}, we have $c(t+u)\leq e^{Kt}c(u)$ for any $u\geq0$. Hence,
    \begin{equation*}
        \mathbb{E}\left[\int_0^\tau c(t+u)\,du\right]\leq e^{Kt}\mathbb{E}\left[\int_0^\tau c(u)\,du\right]<\infty
    \end{equation*}
    by assumption.
\end{proof}

\begin{proposition}\label{prop:continuity_fixed_stopping}
    Under Assumptions \ref{asmp:state_dynamics}, \ref{asmp:base_penalty:c}, and \ref{asmp:base_penalty:g}, for every fixed admissible stopping time $\tau\in\mathcal{T}$, the map 
    \begin{equation*}
        (t,x)\mapsto \mathbb{E}\left[\int_0^\tau c(t+u)\,du+g(X_\tau^x)\right]
    \end{equation*}
    is continuous on $\overline{\mathcal{S}}$.
\end{proposition}
\begin{proof}
    Since the expression is additively separable, it suffices to show continuity of 
    \begin{equation*}
        t\mapsto \mathbb{E}\left[\int_0^\tau c(t+u)\,du\right], \qquad\text{and}\qquad x\mapsto \mathbb{E}[g(X^x_\tau)].
    \end{equation*}
    For temporal continuity, if $t\in[0,T_1]$, then $c(t+u)\leq e^{KT_1}c(u)$ by \eqref{eq:expGrowth}. Thus, since $\tau\in\mathcal{T}$, we can apply dominated convergence to conclude temporal continuity. 
    
    For spatial continuity we first take $(x_n)\subset(0,1)$ such that $x_n\to x\in(0,1)$. Then, pathwise on $\{\tau<\infty\}$, we have
    \begin{equation*}
        X^{x_n}_{\tau(\omega)}(\omega)\to X^{x}_{\tau(\omega)}(\omega)
    \end{equation*}
    for almost all $\omega\in\Omega$ by Proposition \ref{prop:joint_continuity}. On $\{\tau=\infty\}$ we have $X^y_\tau=X^y_\infty\in\{0,1\}$ for all $y\in(0,1)$, hence
    \begin{equation*}
        g(X^{x_n}_\tau)=0=g(X^x_\tau)
    \end{equation*}
    for all $n$. Combining, we get for any stopping time that $g(X^{x_n}_\tau)\to g(X^x_\tau)$ almost surely. Since $g$ is bounded, dominated convergence yields the claim for $x\in(0,1)$. 
    
    To extend to the endpoints we note that by continuity of $g$ we can for every $\varepsilon>0$ find a $\delta>0$ such that $g(x)<\varepsilon$ for all $x\in[0,\delta]$. Since $X^x$ is a bounded martingale, optional sampling yields $\mathbb{E}[X^x_{\tau}]=x$ for possibly infinite stopping times $\tau$. Markov's inequality yields
    \begin{equation*}
        \mathbb{P}(X^x_\tau>\delta)\leq \frac{x}{\delta},
    \end{equation*}
    and hence 
    \begin{equation*}
        \mathbb{E}[g(X^x_\tau)]\leq \varepsilon \mathbb{P}(X^x_\tau\leq \delta) +\|g\|_\infty \mathbb{P}(X^x_\tau>\delta)\leq \varepsilon+\|g\|_\infty \frac{x}{\delta}.
    \end{equation*}
    Taking the limit we get
    \begin{equation*}
        0\leq \limsup_{x\searrow0} \mathbb{E}[g(X^x_\tau)]\leq \varepsilon.
    \end{equation*}
    Since $\varepsilon$ was arbitrary, we get spatial continuity at $0$. A symmetric argument at $1$ yields the claim. 
\end{proof}

\subsection{Transformations and subproblems}

Different transformations of the problem yield insights into the structure of the solutions. A standard transformation is the Lagrange formulation of the problem defined by 
\begin{equation*}
    W(t,x):=V(t,x)-g(x).
\end{equation*}
Then Dynkin's formula yields
\begin{equation}\label{eq:lagrange_formulation}
    W(t,x)=\inf_\tau \mathbb{E}\left[\int_0^\tau \big(c(t+u)+\mathcal{L}g(X_u^x)\big)\,du\right],
\end{equation}
by the following proposition.

\begin{proposition}\label{prop:lagrange}
Under Assumption \ref{asmp:baseline}, it holds that
\begin{equation}\label{eq:dynkin}
    \mathbb{E}[g(X^x_\tau)]-g(x)=\mathbb{E}\left[\int_0^\tau \mathcal{L}g(X^x_u)\,du\right],
\end{equation}
for any stopping time $\tau$ adapted to $\mathbb{F}$.
Specifically, the representation \eqref{eq:lagrange_formulation} holds.
\end{proposition}

\begin{proof}
Fix $x\in(0,1)$ and an $\mathbb{F}$--stopping time $\tau$. Since $g'$ may be unbounded near
the endpoints and $\tau$ may be unbounded, we need to localize in space and time, respectively. For $m$ large enough that
$x\in(1/m,1-1/m)$, set
\[
    \eta_m:=\inf\{u\ge0:X_u^x\notin(1/m,1-1/m)\}.
\]
By Proposition \ref{prop:state_process}, $\eta_m\nearrow\infty$ almost surely. Since
$g\in C^2(1/m,1-1/m)$ with bounded derivatives on compact subintervals,
Dynkin's formula applied to $\tau\wedge n\wedge\eta_m$ gives
\[
    \mathbb E[g(X^x_{\tau\wedge n\wedge\eta_m})-g(x)]=\mathbb E\left[\int_0^{\tau\wedge n\wedge\eta_m}\mathcal L g(X_u^x)\,du\right].
\]
Letting $m\to\infty$, the left-hand side converges by dominated convergence,
because $g$ is bounded and continuous. On the right-hand side, concavity of
$g$ implies $\mathcal L g\le0$, so monotone convergence yields
\[
    \mathbb E[g(X^x_{\tau\wedge n})-g(x)]=\mathbb E\left[\int_0^{\tau\wedge n}\mathcal L g(X_u^x)\,du\right].
\]
Finally let $n\to\infty$. Again the left-hand side converges by dominated
convergence, using the convention $X^x_\tau=X^x_\infty$ on $\{\tau=\infty\}$,
while the right-hand side converges by monotone convergence since
$\mathcal Lg\le0$. Hence, we arrive at \eqref{eq:dynkin}.
Adding the time cost and taking the infimum over stopping times $\tau\in\mathcal{T}$, we arrive at \eqref{eq:lagrange_formulation}.
\end{proof}

Although the formulations in $(t,x)$--coordinates are natural, the state dependence of the diffusion coefficient complicates the analysis. To work with a unit diffusion coefficient, we introduce the Lamperti transformation. Fix a reference point $z\in(0,1)$, say $z=1/2$ under the symmetry assumption \ref{asmp:sym}, and define
\begin{equation*}
    \Psi(x):=\int_z^x \frac{1}{\sigma(y)}\,dy,
\end{equation*}
for $x\in(0,1)$.
Then $\Psi$ is strictly increasing, and hence invertible on its image. Writing $l=\Psi(x)$, define the transformed process
\[
    L_t^l:=\Psi(X_t^x).
\]
As in the original coordinates we formally extend the process and the mapping at the endpoints by $L^l_t\equiv l$ for all $t$ when $l=\Psi(x):=\lim_{y\to x}\Psi(y)$ for $x\in\{0,1\}$, where the limit may be infinite.
Under Assumption \ref{asmp:sigma_g} on $\sigma$, It\^o's formula shows that $L^l$ solves
\begin{equation}\label{eq:dynamics_L}
    dL_t^l=\mu(L_t^l)\,dt+dB_t,\qquad L_0^l=l,
\end{equation}
where
\begin{equation}\label{eq:lamperti_drift}
    \mu(l)=-\frac{\sigma'(\Psi^{-1}(l))}{2}.  
\end{equation}
The transformed state space is
\[
    \tilde{\mathcal{S}}:=\mathbb{R}_+\times \Psi((0,1)),
\]
and we note that the assumptions in Section \ref{subsec:standing_assumptions} do not impose any assumptions on the size of $\Psi((0,1))$, other than it being a nonempty interval.

Define the transformed value function
\begin{equation*}
    \tilde{V}(t,l):=\inf_{\tau}\mathbb{E}\left[\int_0^\tau c(t+u)\,du+\tilde{g}(L_\tau^l)\right]
    =V(t,\Psi^{-1}(l)),
\end{equation*}
where the infimum is taken over admissible $\mathbb{F}$ stopping times and $\tilde{g}(l):=g(\Psi^{-1}(l))$. Similarly, define
\begin{equation*}
    \tilde{W}(t,l):=\inf_{\tau}\mathbb{E}\left[\int_0^\tau \big(c(t+u)+\tilde{\mathcal{L}}\tilde{g}(L_u^l)\big)\,du\right]
    =W(t,\Psi^{-1}(l)),
\end{equation*}
where the infinitesimal generator of $L^l$ is
\begin{equation*}
    \tilde{\mathcal{L}}=\mu(l)\partial_l+\frac12\partial_{ll}.
\end{equation*}
Similarly, we also denote the Lamperti-transformed sets by $\tilde{\mathcal{C}}$, $\tilde{\mathcal{D}}$ and so on, meaning, e.g., $(t,l)\in\tilde{\mathcal{C}}$ if and only if $(t,\Psi^{-1}(l))\in\mathcal{C}$.

Through the Lamperti transformation, the derivative of the stochastic flow can easily be derived under Assumptions \ref{asmp:state_dynamics} and \ref{asmp:sigma_g} on $\sigma$. Differentiating the integrated version of \eqref{eq:dynamics_L} gives
\begin{equation*}
    \partial_lL_t^l=1+\int_0^t \mu'(L_u^l)\partial_lL_u^l\,du,
\end{equation*}
for $l\in\Psi((0,1))$, and hence
\begin{equation}\label{eq:partiall_explicit}
    \partial_lL_t^l=\exp\left(\int_0^t \mu'(L_u^l)\,du\right),
\end{equation}
up to the first time $L^l_t\in\Psi(\{0,1\})$, which does not happen in finite time.

By the chain rule this yields
\begin{equation}\label{eq:partialx_explicit}
    \partial_xX_t^x
    =
    \frac{\sigma(X_t^x)}{\sigma(x)}
    \exp\left(-\int_0^t \frac{\sigma''(X_u^x)\sigma(X_u^x)}{2}\,du\right),
\end{equation}
and equivalently
\begin{equation}\label{eq:sde_partialx}
    d(\partial_xX_t^x)=\sigma'(X_t^x)\partial_xX_t^x\,dB_t,
    \qquad \partial_xX_0^x=1.
\end{equation}
It should be clear from the explicit expression \eqref{eq:partialx_explicit} that joint continuity and measurability statements like those in Proposition \ref{prop:joint_continuity} can also be made for $(t,x,\omega)\mapsto \partial_xX_t^x(\omega)$.

Although the main motivation comes from exciting games that naturally lead to infinite-horizon problems, the regularity analysis is more conveniently carried out for finite-horizon problems. We therefore consider, for each $0<T<\infty$, the family
\begin{equation}\label{eq:OSP_T}
    V_T(t,x):=\inf_{\tau\in\mathcal{T}_{T-t}}\mathbb{E}\left[\int_0^\tau c(t+u)\,du+g(X_\tau^x)\right],
\end{equation}
where
\[
    \mathcal{T}_{T-t}:=\{\tau\in\mathcal{T}:\tau\le T-t\}.
\]
Likewise, we denote the finite-horizon versions of the infinite-horizon quantities by the $T$ subscript, e.g., $\mathcal{C}_T$ and $\tilde{W}_T$ for the finite-horizon continuation region and the finite-horizon Lamperti-transformed Lagrange formulation, respectively. Formally, we also write $V_\infty=V$, etc. for the infinite-horizon problem. Specifically, we denote the state space by 
\[
\mathcal{S}_T:=[0,T)\times(0,1),
\]
and the extended state space by
\begin{equation*}
    \overline{\mathcal{S}}_T:=[0,T]\times[0,1].
\end{equation*}
As for the infinite-horizon state space, we will consider points $(t,x)$ to be in the state space $\mathcal{S}_T$ unless explicitly stated. We formally set
\begin{equation*}
    V_T(t,x):=g(x)
\end{equation*}
for $(t,x)\in\overline{\mathcal{S}}_T\setminus\mathcal{S}_T$.

When no ambiguity arises, we suppress the dependence of the admissible class of stopping times on the horizon. We emphasize that the running cost $c$ remains defined on all of $\mathbb{R}_+$.

\subsection{Preliminary Analysis}

We next address initial properties of the value function and the optimal stopping rule. The following results hold for $T\leq\infty$ unless otherwise stated.

\begin{proposition}\label{prop:semicontinuity}
If Assumption \ref{asmp:baseline} and $T\leq \infty$ hold, then $V_T$ is upper semicontinuous, $\mathcal{C}_T$ is open, and $\mathcal{D}_T$ is closed. Furthermore, the stopping time
\begin{equation}\label{eq:definition_optimal_stopping_time}
    \tau^*_T(t,x):=\inf\{0\le s<T-t:(t+s,X_s^x)\in\mathcal{D}_T\}\wedge(T-t)
\end{equation}
is the smallest optimal stopping time for $T<\infty$.
\end{proposition}

\begin{proof}
If $T=\infty$, taking the infimum over continuous functions indexed by a family $\tau\in\mathcal{T}$ induces an upper semicontinuous function $V_T$. For $T<\infty$, fix $(t,x)\in\mathcal S_T$ and $\varepsilon>0$. Choose an $\varepsilon$--optimal time $\tau\in\mathcal T_{T-t}$ such that
\begin{equation*}
    V_T(t,x)+\varepsilon >
    \mathbb E\!\left[\int_0^\tau c(t+u)\,du + g(X_\tau^x)\right].
\end{equation*}
Let $(t_n,x_n)\to(t,x)$ and define $\tau_n:=\tau\wedge (T-t_n)$, for which $\tau_n\to\tau$ almost surely. 
Reverse Fatou's lemma and continuity yield
\begin{equation*}
    \limsup_{n\to\infty}\mathbb E\!\left[\int_0^{\tau_n} c(t_n+u)\,du + g(X_{\tau_n}^{x_n})\right]
    \leq
    \mathbb E\!\left[\int_0^\tau c(t+u)\,du + g(X_\tau^x)\right],
\end{equation*}
where we use that on a finite horizon the integrand can easily be dominated.
Since $V_T(t_n,x_n)$ is the optimal value, we have
\begin{equation*}
    \limsup_{n\to\infty} V_T(t_n,x_n)
    \le
    V_T(t,x)+\varepsilon.
\end{equation*}
Since $\varepsilon>0$ was arbitrary, $V_T$ is upper semicontinuous. Since $g$ is continuous, it holds that $V_T-g$ is upper semicontinuous and thus $\mathcal{C}_T$ is open and $\mathcal{D}_T$ is closed, with respect to $\mathcal{S}_T$.

The second claim follows from standard optimal stopping arguments, see \cite[Corollary 2.9]{peskir2006optimal}, if $\tau^*_T(t,x)<\infty$ almost surely, which holds in the finite-horizon problem. 
\end{proof}

If $\int_0^\infty c(t)\,dt<\infty$, then stopping at infinity may also be optimal for the infinite-horizon problem. We later show that $\tau^*_T(t,x)$ remains optimal in the infinite-horizon problem.

Next, we analyze the structure of the state space.
Introduce the set
\begin{equation*}
    \mathcal{U}:=\{(t,x)\in\mathcal{S}:\mathcal{L}g(x)<-c(t)\}.
\end{equation*}

\begin{proposition}\label{prop:U_subset_C}
Under Assumption \ref{asmp:baseline}, it holds that $\mathcal{U}\cap\mathcal{S}_T\subseteq\mathcal{C}_T$.
\end{proposition}

\begin{proof}
Take any $(t,x)\in\mathcal{U}\cap\mathcal{S}_T$. Since $\mathcal{L}g$ and $c$ are continuous, $\mathcal{U}$ is open and thus $\mathcal{U}\cap\mathcal{S}_T$ is open. Thus, there exists $\varepsilon>0$ such that
\begin{equation*}
    R_\varepsilon(t,x):=[t,t+\varepsilon)\times(x-\varepsilon,x+\varepsilon)\subset\mathcal{U}\cap\mathcal{S}_T.    
\end{equation*}
Let
\begin{equation*}
    \tau_\varepsilon:=\inf\{u\ge0:(t+u,X_u^x)\notin R_\varepsilon(t,x)\}
\end{equation*}
be the first exit time from $R_\varepsilon(t,x)$, which we note is admissible even if $T<\infty$. Furthermore, $\tau_\varepsilon>0$ almost surely. Then
\begin{equation*}
    W_T(t,x)\le \mathbb{E}\left[\int_0^{\tau_\varepsilon}\big(c(t+u)+\mathcal{L}g(X_u^x)\big)\,du\right]<0.
\end{equation*}
Therefore $V_T(t,x)<g(x)$, and hence $(t,x)\in\mathcal{C}_T$.
\end{proof}

Assumption \ref{asmp:Lg:goodset_nonempty} guarantees a nonempty continuation region ex ante, since it implies that the $t$--slice of $\mathcal{U}$ is nonempty for all $t\ge0$. Assumption \ref{asmp:Lg:unimodal} then gives the basic interval structure of the $t$--slice.

\begin{lemma}\label{lem:Lg0}
Under Assumption \ref{asmp:baseline} it holds that
\begin{equation*}
    \mathcal{L}g(0+)=\mathcal{L}g(1-)=0.
\end{equation*}
Consequently, $\mathcal{U}(t)\subset(0,1)$ and $\mathcal{U}(t)$ is an interval for every $t\ge0$, where $\mathcal{U}(t):=\{x:(t,x)\in\mathcal{U}\}$.
\end{lemma}

\begin{proof}
We only treat the case $x\to1$ since the other endpoint is analogous. By Assumption \ref{asmp:Lg:unimodal}, the limit $\mathcal{L}g(1-)$ exists. Since $g$ is concave, $\mathcal{L}g\le0$. Suppose, for contradiction, that $\mathcal{L}g(1-)\le -\varepsilon$ for some $\varepsilon>0$. By monotonicity we then have $\mathcal{L}g(x)\le-\varepsilon$ for all $x\in(x_0,1)$, so that
\[
    g''(x)\le -\frac{2\varepsilon}{\sigma^2(x)},\qquad x\in(x_0,1).
\]
Integrating twice yields
\begin{align*}
    g(x)
    &= g(x_0)+\int_{x_0}^x \left(g'(x_0)+\int_{x_0}^y g''(z)\,dz\right)\,dy \\
    &\le g(x_0)+(x-x_0)g'(x_0)-2\varepsilon\int_{x_0}^x \int_{x_0}^y \frac{1}{\sigma^2(z)}\,dz\,dy.
\end{align*}
Letting $x\to1$ and using the Feller condition \eqref{eq:feller_test} gives $0=g(1)\le-\infty$, a contradiction.

For the second part, Assumption \ref{asmp:Lg:unimodal} implies that for each $t\ge0$, the equation
\[
    \mathcal{L}g(x)=-c(t)
\]
has at most two solutions, which lie inside $(0,1)$. Hence $\mathcal{U}(t)$ is either empty or an open interval strictly contained in $(0,1)$. By Assumption \ref{asmp:Lg:goodset_nonempty}, $\mathcal{U}(t)$ is nonempty and so we are done.
\end{proof}

We denote the upper boundary of $\mathcal{U}$ by $\gamma$, that is,
\begin{equation*}
    \gamma(t):=\sup\{x\in(x_0,1):\mathcal{L}g(x)+c(t)=0\}.
\end{equation*}
By Lemma \ref{lem:Lg0}, we have $\gamma(t)\in(x_0,1)$ for all $t\ge0$. Furthermore, under the extended Assumption \ref{asmp:extended}, it holds that $(t,x)\mapsto \mathcal{L}g(x)+c(t)$ is $C^1([0,\infty)\times(0,1))$, and so the implicit function theorem yields that $t\mapsto \gamma(t)$ is $C^1([0,\infty))$ since Assumption \ref{asmp:dxLg} ensures 
\begin{equation*}
    \partial_x(\mathcal{L}g(x)+c(t))\Big\vert_{x=\gamma(t)}>0,
\end{equation*}
for $t\geq0$.

\section{Main Results}\label{sec:main_results}

We now state the main results of the paper. 
The proofs are postponed to the following sections and will generally be proved first for the homogeneous infinite-horizon problem, then for the inhomogeneous finite-horizon, and finally for the inhomogeneous infinite-horizon problem. In all statements of the current section we assume Assumption \ref{asmp:extended} and $T\leq \infty$.
The first main structural result shows that the continuation set inherits the same one-interval geometry as $\mathcal{U}$, namely, by Assumption \ref{asmp:sym},
\begin{equation*}
    \mathcal{U}=\{(t,x)\in\mathcal{S}:1-\gamma(t)<x<\gamma(t)\}.
\end{equation*}

\begin{theorem}[Structure of the continuation region]\label{thm:main_structure}
There exists a lower semicontinuous boundary $b_T:[0,T)\to(1/2,1)$
such that
\begin{equation}\label{eq:main_continuation_region}
    \mathcal{C}_T
    =
    \{(t,x)\in\mathcal{S}_T:1-b_T(t)<x<b_T(t)\}.
\end{equation}
The first hitting time in \eqref{eq:definition_optimal_stopping_time} is given by
\begin{equation*}
    \tau^*_T(t,x)=\inf\{0\le u<T-t:X_u^x\not\in(1-b_T(t+u), b_T(t+u))\}\wedge(T-t)
\end{equation*}
and is an optimal stopping time for \eqref{eq:OSP_T} if $T<\infty$ and \eqref{eq:inf_horizon_osp} if $T=\infty$. 
Furthermore,
\begin{equation*}
    \gamma(t)<b_T(t)<1,\qquad t<T.
\end{equation*}
\end{theorem}

For the inhomogeneous case, we establish Theorem \ref{thm:main_structure} in four steps. First, we extend the conclusion of Proposition \ref{prop:semicontinuity} to the infinite-horizon problem by a limiting argument.
Second, we prove existence of a boundary function satisfying
\begin{equation*}
    \gamma(t)\le b_T(t)\le 1,
    \qquad t<T.
\end{equation*}
Third, we prove $b_T(t)<1$. Finally, using the smooth fit obtained in Theorem \ref{thm:main_regular_smooth} below, we prove $b_T(t)>\gamma(t)$.
This result is the main qualitative description of the optimal rule. 
The decision maker continues only while the posterior remains sufficiently uncertain. 
The boundary $b_T(t)$ determines how close the posterior must be to certainty before stopping becomes optimal. In particular, the strict upper boundary $b_T(t)<1$ for $t<T$ implies that there always is a level of uncertainty at which it is optimal to make the decision.

The next result records the regularity of the value function and establishes the well-known smooth-fit principle.
\begin{theorem}[Regularity and smooth fit]\label{thm:main_regular_smooth}
The value function is $V_T\in C^{1,2}(\mathcal{S}_T\setminus\partial \mathcal{D}_T)\cap C^1(\mathcal{S}_T)$. 
\end{theorem}

The proof of this theorem is one of the main technical parts of the paper. 
For finite horizons, one can show probabilistic regularity of the optimal stopping boundary by showing that the lower semicontinuity of $b_T$ can be extended to Lipschitz continuity. This yields the spatial and temporal smooth fit. However, the arguments do not extend directly to the infinite-horizon case, due to integrability issues on an infinite horizon. Instead, we show the spatial smooth fit by a simple argument based on uniform convergence.

We can also show that the regularity found in Theorem \ref{thm:main_regular_smooth} is in some sense maximal.
\begin{corollary}\label{cor:main_discont_dxx}
    The second spatial derivative of the value function admits the representation
    \begin{equation*}
        \partial_{xx} V_T(t,x)=\begin{cases}
            -2\frac{c(t)+\partial_tV_T(t,x)}{\sigma^2(x)} & \text{for }(t,x)\in\mathcal{C}_T\\
            g''(x)&\text{for }(t,x)\in(\mathcal{D}_T)^\circ.
        \end{cases}
    \end{equation*}
    Specifically, it is not continuous across the boundary.
\end{corollary}
The specific representation of the temporal derivative will be given in \eqref{eq:time_derivative_finite_horizon_org} and \eqref{eq:time_deriv_inf} below for finite and infinite horizons, respectively.
The smooth-fit principle also enables us to upgrade the Lipschitz continuity of $b_T$ to $C^1([0,T))$ by direct application of \cite{de2024probabilistic}.
\begin{corollary}\label{cor:main_c1}
    The boundary satisfies $b_T\in C^1([0,T))$.
\end{corollary}

The regularity in Theorem \ref{thm:main_regular_smooth} leads to an equation for the unknown boundary itself. 
The integral equation is particularly useful for numerical schemes based on Picard iteration.

\begin{theorem}[Free-boundary integral equation]\label{thm:main_integral_equation}
The optimal boundary $b_T$ solves
\begin{equation}\label{eq:main_boundary_integral}
    \begin{aligned}
        g(x)=\mathbb{E}\Bigg[g(X^x_{T-t})&+\int_0^{T-t}c(t+u)\ind{X^x_{u}\in(1-b_T(t+u),b_T(t+u))}\,du\\
        &-\int_0^{T-t}\mathcal{L}g(X^x_{u})\ind{X^x_{u}\not\in(1-b_T(t+u),b_T(t+u))}\,du\Bigg],
    \end{aligned}
\end{equation}
for $x=b_T(t)$ and $x=1-b_T(t)$ with $t<T$. For $T=\infty$, the integral is interpreted as the monotone limit as the integration domain expands and $g(X^x_{T-t})=0$ since $X^x_{\infty}\in\{0,1\}$, almost surely.

Moreover, this equation characterizes the optimal boundary uniquely in the following sense. 
If $\beta:[0,T)\to(1/2,1)$ is a continuous function satisfying
\[
    \gamma(t)<\beta(t)<1
\]
and solves \eqref{eq:main_boundary_integral} with $b_T(t)$ replaced by $\beta(t)$ for all $t<T$, then $\beta(t)=b_T(t)$ for all $t<T$.
\end{theorem}

Before proving these statements under Assumption \ref{asmp:extended}, we will first take a detour into the homogeneous infinite-horizon problem, which admits constant optimal stopping boundaries, when they exist. The proof techniques are different and allow for less regularity in the assumptions.

\section{Homogeneous Infinite-Horizon Problems}\label{sec:homogeneous}
In this section we assume Assumption \ref{asmp:baseline} with $T=\infty$ and the time cost $c$ being constant.
Then the infinite-horizon optimal stopping problem, \eqref{eq:inf_horizon_osp}, becomes time-homogeneous. That is, we are to solve
\begin{equation}\label{eq:InfHorizonProblem_hom}
    V(x)=\inf_{\tau\in\mathcal{T}}\mathbb{E}\left[\int_0^\tau c\,du + g(X^x_\tau)\right],
\end{equation}
for a process $X^x$ given by \eqref{eq:general_X_SDE}. We note that $\tau\in\mathcal{T}$ is equivalent to $\mathbb{E}[\tau]<\infty$ for such a constant time cost. As usual in homogeneous one-dimensional optimal stopping problems, we formulate a free-boundary problem and verify that it indeed solves the optimal stopping problem. The natural free-boundary problem is as follows:
\begin{align}
    \mathcal{L}\widehat V(x)&=-c && x\in(A,B) \label{ode}\\
    \widehat V(x)&=g(x) && x\in\{A,B\}\label{cont}\\
    \widehat V'(x)&=g'(x) && x\in\{A,B\}\label{smooth}\\
    \widehat V(x)&< g(x)&&x\in(A,B)\label{eq:minorant}\\
    \widehat V(x)&= g(x)&&x\in[0,1]\setminus(A,B)\label{eq:value_mmatch}\\
    0<A &< B<1.\label{eq:A<B}
\end{align}
Note that since we only work under Assumption \ref{asmp:baseline}, we have formulated the problem to allow for asymmetric boundaries.
The following is an extension of \cite{campbell2025softclass} for general win-martingales.
\begin{theorem}\label{thm:FBP_homogeneous}
    Under Assumption \ref{asmp:baseline}, the free-boundary problem admits a solution $\widehat V\in C^2((0,1)\setminus\{A,B\})\cap C^1((0,1))$ which is unique. If Assumptions \ref{asmp:state_dynamics}-\ref{asmp:Lg:unimodal} and $\mathcal{L}g(x_0)\geq -c$ hold, then \eqref{ode}--\eqref{eq:value_mmatch} can only hold if $A=B$, that is, $\widehat V=g$.
\end{theorem}
\begin{proof}
    We rewrite the ODE in \eqref{ode} as 
    \begin{equation}\label{ode2}
        \widehat V''(x)=-\frac{2c}{\sigma^2(x)}.
    \end{equation}
    Let $\psi$ be any particular solution to the inhomogeneous ODE \eqref{ode2}. Then the general solution is $\psi(x)+c_1 x+c_2$ for constants $c_1$ and $c_2$. Using the boundary conditions for smooth and continuous fit, \eqref{cont} and \eqref{smooth}, at $A$, we get the representation for the candidate solution
    \begin{equation}\label{ValueGivenA}
        \widehat V_A(x) := (x-A)(g'(A)-\psi'(A))+g(A)-\psi(A)+\psi(x)
    \end{equation}
    for $x\in(A,B)$.

    Define $\widehat W_A(x)=\widehat V_A(x)-g(x)$ and $H(x)=g(x)-\psi(x)$. We can then write
    \begin{equation*}
        \widehat W_A(x)=(x-A)H'(A)+H(A)-H(x).
    \end{equation*}
    The smooth and continuous fit at $B$ can be written in terms of $H$ as 
    \begin{align}
        H'(A)&=H'(B)\label{hprime}\\
        H'(B)&=\frac{H(B)-H(A)}{B-A}\label{hsecant}
    \end{align}
    given $A\neq B$, as will be shown later. That is, we want to show that there exists a unique pair of points such that the tangents of $H$ at these points are equal and the slope of the secant line connecting the function values coincides with the tangents. To this end, we will need to inspect the $H$ function and thus the $g$ and $\psi$ functions.

    By Lemma \ref{lem:Lg0}, we have $\mathcal{L}g(0+)=\mathcal{L}g(1-)=0$ and the equation $\mathcal{L}g(x)=-c$ has two unique solutions $0<x_*<x^*<1$. Thus,
    \begin{equation*}
        H''(x)=g''(x)-\psi''(x)=\frac{2}{\sigma^2(x)}\left(\mathcal{L}g(x)+c\right),
    \end{equation*}
    which leads us to conclude $H'$ is strictly increasing on $(0,x_*)$ and $(x^*,1)$ while strictly decreasing on $(x_*,x^*)$. Furthermore, for any fixed $y\in(0,1)$, we have
    \begin{equation*}
        \psi'(x)=\int_y^x\psi''(z)dz+\psi'(y)=-\int_y^x\frac{2c}{\sigma^2(z)}dz+\psi'(y),
    \end{equation*}
    which diverges by \eqref{eq:feller_test} since for $x>y$ we have
    \begin{equation*}
        \int_y^x\frac{1}{\sigma^2(z)}\,dz\geq\int_y^x\frac{x-z}{\sigma^2(z)}\,dz=\int_y^x\int_z^x\,du\,\frac{1}{\sigma^2(z)}\,dz=\int_y^x\int_y^u\frac{1}{\sigma^2(z)}\,dz \,du\to\infty
    \end{equation*}
    for $x\to 1$ by Fubini's theorem. Hence, we have $\psi'(1-)=-\infty$ and a similar argument shows $\psi'(0+)=\infty$. As for
    $H'$ we can write
    \begin{equation*}
        H'(x)=g'(x)-\psi'(x) = \psi'(x)\left(\frac{g'(x)}{\psi'(x)}-1\right)
    \end{equation*}
    and we want to show $H'(0+)=-\infty$ and $H'(1-)=\infty$. We only consider the $x\to0$ case, the case $x\to1$ being similar. Since $g$ is concave with $g(0)=0$ and $g(x)>0$ for $x\in(0,1)$, $g'(x)>0$ for sufficiently small $x>0$, and $\lim_{x\searrow0}g'(x)$ exists. Suppose $\lim_{x\searrow0}g'(x)<\infty$. Then $H'(0+)=-\infty$ follows from $\psi'(x)\to\infty$. Conversely, if $\lim_{x\searrow0}g'(x)=\infty$, we consider
    \begin{equation*}
        \frac{g''(x)}{\psi''(x)} = \frac{\mathcal{L}g(x)}{-c}\to0
    \end{equation*}
    for $x\to0$ since $\mathcal{L}g(0+) = 0$, which eventually converges monotonically by Assumption \ref{asmp:Lg:unimodal}. Thus, by L'H\^opital's rule, $\left(\frac{g'(x)}{\psi'(x)}-1\right)\to-1$ eventually bounded away from 0 and thus, $H'(0+)=-\infty$ and similarly $H'(1-)=\infty$.

    By the preceding facts we get that there exists a unique solution $\underline{x}\in(0,x_*)$ to the equation $H'(\underline{x})=H'(x^*)$, and likewise a unique $\overline{x}\in(x^*,1)$ such that $H'(\overline{x})=H'(x_*)$. For $\tilde{x}\in(x_*,x^*)$ there exist two solutions to $H'(x)=H'(\tilde{x})$, one in $(\underline{x},x_*)$ and one in $(x^*,\overline{x})$. Finally, a solution to \eqref{hprime} cannot exist in $[0,\underline{x})\cup(\overline{x},1]$. We note that the solution to \eqref{hprime} must have $A\in[\underline{x},x_*]$ and $B\in[x^*,\overline{x}]$ since if, say, $A\in(x_*,x^*)$ then
    \begin{equation*}
        \widehat W_A(x)=(x-A)H'(A)+H(A)-H(x)=\int_A^x H'(A)-H'(y) \,dy >0
    \end{equation*}
    for $x\in(A,x^*)$, which contradicts $\widehat V_A(x)\leq g(x)$. A similar argument holds for $B$. In particular we must have $A<B$. 

    To summarize our findings, we have that $A\in[\underline{x},x_*]$ and $B\in [x^*,\overline{x}]$ which are disjoint sets. 
    Defining the restrictions $\underline{H}:=H\vert_{[\underline{x},x_*]}$ and $\overline{H}:=H\vert_{[x^*,\overline{x}]}$ we have two continuous convex functions on disjoint compact domains.
    Thus, by Bisztriczky's theorem, see \cite{campbell2025softclass}, there exists exactly one common supporting tangent line $\ell(x):=sx+r$ where
    \begin{equation*}
        s = \frac{H(B)-H(A)}{B-A}.
    \end{equation*}
    A sufficient condition for this tangent line of the restricted functions to be a tangent of the function $H$, and hence be equal to the slope of $H$ at $A$ and $B$, is that $A\in(\underline{x},x_*)$ and $B\in(x^*,\overline{x})$.
    Assume for contradiction that $\ell$ supports $\underline{H}$ at $\underline{x}$. Then
    \begin{equation}\label{eq:support_underline_x}
        s\leq H'(\underline{x}),
    \end{equation}
    by convexity.
    If $\ell$ supports $\overline{H}$ at $B\in(x^*,\overline{x})$, then
    \begin{equation*}
        s=H'(B)>H'(x^*)= H'(\underline{x}),
    \end{equation*}
    and if $\ell$ supports $\overline{H}$ at $B=\overline{x}$, then
    \begin{equation*}
        s\geq H'(\overline{x})>H'(\underline{x}),
    \end{equation*}
    and both cases lead to a contradiction. Hence $\ell$ must support $\overline{H}$ at $x^*$. However, this yields
    \begin{equation*}
        s=\frac{H(x^*)-H(\underline{x})}{x^*-\underline{x}}=\frac{1}{x^*-\underline{x}}\int_{\underline{x}}^{x^*}H'(y)\,dy>\frac{1}{x^*-\underline{x}}\int_{\underline{x}}^{x^*}H'(\underline{x})\,dy=H'(\underline{x}),
    \end{equation*}
    which also contradicts \eqref{eq:support_underline_x}. Analogous arguments rule out support at the endpoints $x_*,x^*$, and $\overline{x}$. Hence there exists a unique solution to \eqref{hprime} and \eqref{hsecant} with $A\in(\underline{x},x_*)$ and $B\in(x^*,\overline{x})$.

    Finally, we can also check the inequality constraint \eqref{eq:minorant}. Recall $H'(A)-H'(x)$ has roots at $A$ and $B$, and some point $\tilde{x}\in(x_*,x^*)$. Furthermore, it is strictly negative for $x\in(A,\tilde{x})$ and strictly positive for $x\in(\tilde{x},B)$. Hence,
    \begin{equation*}
        \widehat W_A(x)=\int_A^xH'(A)-H'(y)\,dy <0
    \end{equation*}
    for $x\in(A,\tilde{x}]$, and similarly, for $x\in[\tilde{x},B)$
    \begin{equation*}
        \widehat W_A(x)=\widehat W_A(x)-\widehat W_A(B)=-\int_x^B H'(A)-H'(y)\,dy <0.
    \end{equation*}
    By pasting to obey \eqref{eq:value_mmatch}, we conclude $\widehat V_A(\cdot)$ solves the free-boundary problem uniquely, and hence we are done with the first claim.

    For the latter claim, assume $\mathcal{L}g(x_0)\geq-c$ and \eqref{ode}--\eqref{eq:A<B}. Then, by Assumption \ref{asmp:Lg:unimodal} and the ODE \eqref{ode}, we have $\mathcal{L}g(x)\geq-c=\mathcal{L}\widehat V_A(x)$ for $x\in(A,B)$. Specifically, Assumption \ref{asmp:Lg:unimodal} yields that the inequality is strict for all $x\in(A,B)\setminus\{x_0\}$, that is, we have $g''(x)>\widehat V_A''(x)$ for a set with positive measure. Thus,
    \begin{align*}
        \widehat V_A(B)&=\widehat V_A(A)+\int_A^B\left(\widehat V_A'(A)+\int_A^y\widehat V_A''(z)\,dz\right)\,dy\\
        &<g(A)+\int_A^B \left(g'(A) +\int_A^y g''(z) \,dz\right)\,dy=g(B),
    \end{align*}
    and hence \eqref{cont} cannot hold at $B$, so we reach a contradiction.
\end{proof}
We can now verify that the solution to the free-boundary problem solves the optimal stopping problem. The verification follows from the generalized It\^o formula and a standard martingale argument, see e.g. \cite{peskir2006optimal}. We provide the main steps below.

\begin{theorem}\label{VerificationInfConstant}
    The following holds for the optimal stopping problem in \eqref{eq:InfHorizonProblem_hom}: 
    \begin{enumerate}
        \item[(i)] If Assumptions \ref{asmp:state_dynamics}-\ref{asmp:Lg:unimodal} and $\mathcal{L}g(x_0)\geq -c$ hold, then $V\equiv g$.
        \item[(ii)] Under Assumption \ref{asmp:baseline}, the optimal solution is given by
        \begin{equation*}
            V(x)=\begin{cases}
                \widehat V_{A^*}(x),   & x\in(A^*,B^*)\\
                g(x),       & x\in[0,A^*]\cup[B^*,1],
            \end{cases}
        \end{equation*}
        where $\widehat V_{A^*}(x)$ is given by \eqref{ValueGivenA}, and $A^*$ and $B^*$ form the unique solution pair to the equations \eqref{hprime} and \eqref{hsecant}.
    \end{enumerate}
\end{theorem}
\begin{proof}
    For the case $(i)$, we get $\mathcal{L}g(x)\geq -c$ for all $x$. Since $c+\mathcal{L}g\geq0$, the Lagrange formulation gives $W\geq0$. Immediate stopping gives $W\leq0$, hence $W=0$ and $V=g$.
    For the case $(ii)$, instead of working with $V$, we will verify using the Lagrange formulation of the problem. Let
    \begin{equation*}
        W(x):=V(x)-g(x)=\inf_\tau \mathbb{E}\left[\int_0^\tau \left(c+\mathcal{L}g(X^x_u)\right)\,du\right],
    \end{equation*}
    where the equality follows from \eqref{eq:lagrange_formulation}. The candidate solution is $W_{A^*}(x)=\widehat V_{A^*}(x)-g(x)$ for $x\in(A^*,B^*)$ and $W_{A^*}(x)=0$ otherwise.
    By the change-of-variable formula with local time on curves \cite{peskir2005change}, and by the smooth fit, we have
    \begin{equation}\label{eq:cov_constant}
        W_{A^*}(X^x_t)=W_{A^*}(x)+\int_0^t\mathcal{L}W_{A^*}(X^x_s)\ind{X^x_s\not\in\{A^*,B^*\}} \,ds +\int_0^t\sigma(X^x_s)W_{A^*}'(X^x_s)dB_s.
    \end{equation}
    Note that $\sigma(x)W_{A^*}'(x)$ is globally bounded since $A^*$ and $B^*$ are bounded away from $\{0,1\}$, so the local martingale term, $M_t=\int_0^t\sigma(X^x_s)W_{A^*}'(X^x_s)dB_s$, is a square-integrable martingale. Since we consider stopping times with $\mathbb{E}[\tau]<\infty$, we get $\mathbb{E}[M_\tau]=0$. Thus, by using \eqref{ode} in $(A^*,B^*)$ and $W_{A^*}(x)=0$ together with $c+\mathcal{L}g(x)\geq0$ for $x\not\in(A^*,B^*)$ in \eqref{eq:cov_constant}, we get 
    \begin{equation*}
        \mathbb{E}\left[\int_0^\tau \left(c+\mathcal{L}g(X^x_u)\right)\,du\right]\geq W_{A^*}(x)
    \end{equation*}
    for any $\tau\in\mathcal{T}$ by taking expectation. Taking infimum, we arrive at $W(x)\geq W_{A^*}(x)$. Conversely, plugging $\tau^*(x)=\inf\{t\geq0 : X^x_t\not\in(A^*,B^*)\}$, which satisfies $\mathbb{E}[\tau^*(x)]<\infty$, into \eqref{eq:cov_constant}, we have 
    \begin{equation*}
        W_{A^*}(x)=\mathbb{E}\left[\int_0^{\tau^*(x)} \left( c+\mathcal{L}g(X^x_u)\right)\,du\right]\geq W(x),
    \end{equation*}
    by taking expectation.
\end{proof}
Thus, one-dimensional asymmetric versions of Theorems \ref{thm:main_structure} and \ref{thm:main_regular_smooth} also hold under the baseline Assumption \ref{asmp:baseline} given a constant running cost $c$. Similarly, an asymmetric version of Corollary \ref{cor:main_c1} holds trivially while asymmetric versions of Corollary \ref{cor:main_discont_dxx} and Theorem \ref{thm:main_integral_equation} can be shown by similar arguments as for the inhomogeneous case below. The nonconstant running-cost case requires more work, but the homogeneous infinite-horizon problem will be useful for comparison in multiple results below.

\section{Inhomogeneous Problems}\label{sec:inhomogeneous}
In this section we will work under the extended Assumption \ref{asmp:extended}.
We first tackle initial regularity results for the value function to obtain estimates on the derivatives. Then we turn to regularity of the optimal stopping boundary and finally we move on to the smooth fit and upgrade the regularity of the optimal stopping boundary. 

It should first be noted that many regularity results have a local nature which in the temporal coordinate is located in the open set $(0,T)$. To extend the result to hold on $[0,T)$ one may consider an extension of the running cost $c$. Namely, it is not hard to see that for every $\varepsilon>0$ there exists an extension $\bar{c}:[0,\infty)\to(0,\infty)$ such that $\bar{c}(t+\varepsilon)=c(t)$ and the problem with the $\bar{c}$ running cost satisfies Assumption \ref{asmp:extended}. If $V_T$ denotes the original value function and $\bar{V}_{T+\varepsilon}$ denotes the value function with the running cost $\bar{c}$ on an extended horizon, then
\begin{equation}\label{eq:extend_0}
    V_T(t,x)=\bar{V}_{T+\varepsilon}(t+\varepsilon,x).
\end{equation}
This means that if a regularity property can be shown on $(0,T+\varepsilon)$ for general problems such as $\bar{V}_{T+\varepsilon}$ then the property specifically holds on $[0,T)$ for $V_T$.

\subsection{Properties of the value function}
We first show that the infinite-horizon problem can be realized as the limit of the finite-horizon problems, which in turn yields existence of an optimal stopping time in the infinite-horizon problem.
\begin{proposition}\label{prop:finite_to_infinite_1}
    For every $(t,x)\in\mathcal{S}$ it holds that $V_T(t,x)\searrow V_\infty(t,x)$ and $\tau^*_T(t,x)\nearrow\tau^*_\infty(t,x)$ as $T\nearrow\infty$ with $T>t$. Furthermore, $\tau^*_\infty$ is an optimal stopping time for \eqref{eq:inf_horizon_osp}.
\end{proposition}
\begin{proof}
    It holds that $T\mapsto V_T(t,x)$ is decreasing since the set of admissible stopping times is increasing and similarly $V_T(t,x)\geq V_\infty(t,x)$ for all $T>t$. Furthermore, $V_T(t,x)\geq0$ so a limit exists, say $V^\infty(t,x)$, thus $V^\infty(t,x)\geq V_\infty(t,x)$. Now fix an $\varepsilon>0$ and take $\tau\in\mathcal{T}$ such that
    \begin{equation*}
        \mathbb{E}\left[\int_0^\tau c(t+u)\,du+g(X^x_\tau)\right]\leq V_\infty(t,x)+\varepsilon.
    \end{equation*}
    For each $T$, $\tau\wedge (T-t)$ is admissible for the $V_T$ problem. Hence, 
    \begin{equation*}
        V_T(t,x)\leq \mathbb{E}\left[\int_0^{\tau\wedge (T-t)} c(t+u)\,du+g(X^x_{\tau\wedge (T-t)})\right].
    \end{equation*}
    Since $\tau$ and $\tau\wedge (T-t)$ are in $\mathcal{T}$, monotone convergence yields the limit of the integral term, while dominated convergence yields the limit of the $g$ term by the path continuity of $X^x$ and $g(X^x_{\tau})=0$ on the event $\{\tau=\infty\}$. Hence, taking limsup, we get
    \begin{equation*}
        V^\infty(t,x)=\limsup_{T\to\infty}V_T(t,x)\leq V_\infty(t,x)+\varepsilon,
    \end{equation*}
    and letting $\varepsilon\to0$, we arrive at $V_T\searrow V_\infty$ pointwise. Define the extended stopping and continuation sets
    \begin{equation*}
        \hat{\mathcal{D}}_T:=\mathcal{D}_T\cup([T,\infty)\times(0,1)),\qquad \hat{\mathcal{C}}_T:=\mathcal{S}\setminus\hat{\mathcal{D}}_T.
    \end{equation*}
    From the monotonicity it is clear that a limit exists, $\hat{\mathcal{D}}_T\searrow\mathcal{D}^\infty$, and $\hat{\mathcal{C}}_T\nearrow\mathcal{C}^\infty$. Furthermore, the limits are closed and open, respectively. The monotonicity also yields a limit $\tau_T^*(t,x)\nearrow \tau^\infty(t,x)$ since
    \begin{align*}
        T\mapsto\tau_T^*(t,x)&= \inf\{s\geq0:(t+s,X^x_s)\in \mathcal{D}_T\}\wedge(T-t)\\
        &=\inf\{s\geq0:(t+s,X^x_s)\in \widehat{\mathcal{D}}_T\}
    \end{align*}
    is increasing pathwise. It holds that 
    \begin{equation*}
        \tau^\infty(t,x)=\lim_{T\to\infty}\tau_T^*(t,x)\leq \tau_{\mathcal{D}^\infty}(t,x):=\inf\{s\geq0: (t+s,X^x_s)\in\mathcal{D}^\infty\}.
    \end{equation*}
    For the converse, we make the following pathwise argument. Take any $r<\tau_{\mathcal{D}^\infty}$. Then
    \begin{equation*}
        K_r:=\{(t+s,X^x_s):0\leq s\leq r\}
    \end{equation*}
    is a compact subset of $\mathcal{C}^\infty$ since $X^x_s\in(0,1)$ for finite times and the paths are continuous. Since $\mathcal{C}^\infty$ is a union of increasing open sets,
    \begin{equation*}
        \mathcal{C}^\infty=\bigcup_{T>t}\hat{\mathcal{C}}_T,
    \end{equation*}
    there exists some $T_0>t+r$ such that for all $T>T_0$ it holds that $K_r\subseteq \hat{\mathcal{C}}_T$. From this, we conclude that $\tau^*_T(t,x)>r$ for all $T>T_0$. Since $r$ was arbitrary, we can arrive at 
    \begin{equation*}
        \tau^\infty(t,x)=\lim_{T\nearrow\infty}\tau^*_T(t,x)\geq \tau_{\mathcal{D}^\infty}(t,x).
    \end{equation*}
    In total, we have $\tau^\infty=\tau_{\mathcal{D}^\infty}$.
    
    We need to identify the limit of the stopping times with the proposed optimal stopping time for the infinite-horizon problem given in \eqref{eq:definition_optimal_stopping_time}. To this end, note that since $V_\infty\leq V_T$, $\mathcal{D}\subseteq \widehat{\mathcal{D}}_T$ for all $T$, and so $\mathcal{D}\subseteq \mathcal{D}^\infty$. For the converse, if $(t,x)\in\mathcal{D}^\infty$, then $(t,x)\in\widehat{\mathcal{D}}_T$ for all $T>t$, i.e. $(t,x)\in\mathcal{D}_T$, thus $V_T(t,x)=g(x)$. But then, by the convergence of the value function, $g(x)=V_T(t,x)\to V_\infty(t,x)$. Thus $(t,x)\in \mathcal{D}$ and so in total $\mathcal{D}=\mathcal{D}^\infty$ which yields $\tau^\infty(t,x)=\tau^*_\infty(t,x)$. 

    Lastly, to show that $\tau^*_\infty(t,x)$ is an optimal stopping time, we first show that $\tau^*_\infty(t,x)\in\mathcal{T}$.
    For finite $T>t$, we have 
    \begin{equation*}
        \mathbb{E}\left[\int_0
        ^{\tau^*_T(t,x)}c(t+u)\,du\right]\leq V_T(t,x)\leq\|g\|_\infty.
    \end{equation*}
    Using monotone convergence, we have 
    \begin{equation*}
    \mathbb{E}\left[\int_0
        ^{\tau^*_\infty(t,x)}c(t+u)\,du\right]\leq\|g\|_\infty,
    \end{equation*}
    and so $\tau^*_\infty(t,x)\in\mathcal{T}$ by Proposition \ref{prop:Stopping_set}.
    To show that the stopping time achieves optimality, we utilize monotone convergence as above and dominated convergence on the bounded $g$ to get
    \begin{align*}
        V_\infty(t,x)=\lim_{T\to\infty}V_T(t,x)=\mathbb{E}\left[\int_0^{\tau^*_\infty(t,x)}c(t+u)\,du+g(X^x_{\tau^*_\infty(t,x)})\right].
    \end{align*}
    Thus $\tau^*_\infty(t,x)$ is admissible and attains the infinite-horizon value as claimed.
\end{proof}

Next, we show a technical lemma for the spatial concavity of the value function in $(t,x)$ coordinates.

\begin{lemma}\label{lem:concave_V}
    For every $t\in[0,T)$ the map $x\mapsto V_T(t,x)$ is concave on $[0,1]$ for $T\leq\infty$.
\end{lemma}
\begin{proof}
    Consider first $T<\infty$. If $t=T$, then $V_T(t,x)=g(x)$, hence concave. Fix a $t\in[0,T)$. We first give a probabilistic proof that for deterministic times the transition semigroup operator on $X^x$ preserves concavity. Let $(P_s)_{s\geq0}$ denote the transition semigroup of the Markov process $X^x$, defined by 
    \begin{equation*}
        P_sf(x)=\mathbb{E}[f(X^x_{s})]
    \end{equation*}
    for appropriate test functions $f$. 

    Let $Z^x_s:=\partial_xX^x_s$. By Assumption \ref{asmp:martingale}, $(Z^x_s)_{s\in[0,T]}$ is a true exponential martingale and $\mathbb{E}[Z^x_s]=1$ for all $x\in(0,1)$ and $s\in[0,T]$.
    Assume $f\in C^1_b([0,1])$.
    Let $x\in(0,1)$ and take $h\in[0,1-x)$. The integral remainder formula yields
    \begin{equation*}
        f(X^{x+h}_t)-f(X^x_t)=h\int_0^1 f'(X^{x+r h}_t)Z^{x+r h}_t dr
    \end{equation*}
    and so the difference quotient of the semigroup operator is given by
    \begin{equation}\label{eq:integralremainder_semi}
        \frac{P_tf(x+h)-P_tf(x)}{h}= \int_0^1\mathbb{E}[f'(X^{x+r h}_t)Z^{x+r h}_t]dr,
    \end{equation}
    by Fubini since $f'$ is bounded and $Z^y_t$ is positive and has expectation 1. Fix an $r\in[0,1]$. Then 
    \begin{equation}\label{eq:L1conv_semi}
        \mathbb{E}|f'(X^{x+r h}_t) Z^{x+r h}_t-f'(X^{x}_t) Z^{x}_t|\leq \|f'\|_\infty \mathbb{E}|Z^{x+r h}_t- Z^{x}_t|+\mathbb{E}[|f'(X^{x+r h}_t)-f'(X^x_t)|Z^x_t].
    \end{equation}
    The first term goes to $0$ by Scheffé's lemma since $y\mapsto Z^y_t$ is almost surely continuous, $Z^y$ is positive, and $\mathbb{E}[Z^{x+r h}_t]=\mathbb{E}[Z^{x}_t]=1$. The second term goes to $0$ by dominated convergence since $f'$ is bounded and $y\mapsto f'(X^y_t)$ is almost surely continuous. Hence
    \begin{equation*}
        \mathbb{E}[f'(X^{x+r h}_t) Z^{x+r h}_t]\to\mathbb{E}[f'(X^{x}_t) Z^{x}_t],
    \end{equation*}
    for any $r\in[0,1]$. Then, by another application of dominated convergence, and a symmetric argument for the left difference quotient, we get
    \begin{equation}\label{eq:derivative_before_mc}
        \partial_x(P_tf)(x)=\mathbb{E}[f'(X^x_t)Z^x_t].
    \end{equation}
    
    As argued above, $Z^x$ is a likelihood process and thus Girsanov's theorem induces a new measure $\mathbb{Q}^x$ on $\mathcal{F}_t$ by 
    \begin{equation*}
        \frac{d\mathbb{Q}^x}{d\mathbb{P}}\Big\vert_{\mathcal{F}_t}=Z^x_t,
    \end{equation*}
    where 
    \begin{equation*}
        W^x_s=B_s-\int_0^s\sigma'(X^x_u)\,du
    \end{equation*}
    for $s\in[0,t]$ is a $\mathbb{Q}^x$--Brownian motion. The process $X^x$ satisfies $X^x_0=x$ and 
    \begin{equation*}
        dX^x_s=\sigma'(X^x_s)\sigma(X^x_s)\,ds+\sigma(X^x_s)\,dW^x_s,
    \end{equation*}
    and \eqref{eq:derivative_before_mc} becomes
    \begin{equation}\label{eq:derivative_Qx}
        \partial_x P_tf(x) = \mathbb{E}^{\mathbb{Q}^x}[f'(X^x_t)].
    \end{equation}
    Let $Y^x=(Y^x_t)_{t\in[0,T]}$ be the unique strong solution, on some Wiener space $(\Omega,\mathcal{G},\mathbb{G},\mathbb{Q})$ carrying a Brownian motion $W=(W_t)_{t\in[0,T]}$, of
    \begin{equation*}
        dY^x_t=\sigma'(Y^x_t)\sigma(Y^x_t)dt+\sigma(Y^x_t)dW_t,
    \end{equation*}
    with $Y^x_0=x$. Then, since the law of $X^x$ under $\mathbb{Q}^x$ is equal to the law of $Y^x$ under $\mathbb{Q}$, we can write \eqref{eq:derivative_Qx} as
    \begin{equation*}
        \partial_x P_tf(x) = \mathbb{E}^{\mathbb{Q}}[f'(Y^x_t)].
    \end{equation*}
    By the comparison principle, see e.g. \cite{ikeda1977comparison}, we have that $x\mapsto Y^x_t$ is increasing almost surely. If $f$ is concave, then $x\mapsto f'(Y^x_t)$ is almost surely decreasing, and hence $x\mapsto\partial_x P_tf(x)$ is decreasing. Therefore, we conclude $x\mapsto P_t f(x)$ is concave. 
    
    Next, assume $f$ is concave with $0\leq f(x)\leq g(x)$ for $x\in[0,1]$. There exists a sequence $(f_n)_{n\in\mathbb{N}}\subset C^1_b([0,1])$ of concave functions such that $f_n\to f$ uniformly. By the triangle inequality, we have
    \begin{equation*}
        \vert\vert P_tf_n-P_tf\vert\vert_\infty\leq \vert\vert f_n - f\vert\vert_\infty,
    \end{equation*}
    so $P_tf_n\to P_tf$ uniformly. Noting that uniform convergence respects concavity, we have $x\mapsto P_tf(x)$ is concave on $(0,1)$. Since $f$ is bounded between $0$ and $g$, and $g(x)=0$ if and only if $x\in\{0,1\}$, it holds that $x\mapsto P_tf(x)$ is concave on $[0,1]$.

    Next, we will upgrade this to hold for stopping times. Specifically, we first prove this for discrete stopping times in the Bermudan approximation. Define for $n\in\mathbb{N}$ the dyadic Bermudan version of \eqref{eq:OSP_T}, namely, $V_T^n(t,x):=v^n_T(0,x)$ where
    \begin{equation*}
        v^n_T(k,x):=\inf_{\tau\in\mathcal{T}^{n,k}_{T-t}}\mathbb{E}\left[\int_0^\tau c(t+k\Delta+u)\,du+g(X^x_\tau)\right]
    \end{equation*}
    for $k=0,\ldots,2^n$, where $\Delta=(T-t)2^{-n}$ and $\mathcal{T}^{n,k}_{T-t}$ is the set of stopping times taking values in $\{0,\Delta,\dots,(2^n-k)\Delta\}$.
    We will denote $\mathcal{T}^{n}_{T-t}:=\mathcal{T}^{n,0}_{T-t}$ for simplicity.
    By the strong Markov property and the fact that the running cost is deterministic, the dynamic programming principle implies
    \begin{equation*}
        v^n_T(k,x)=\min\left\{g(x),\int_0^\Delta c(t+k\Delta+u)\,du+ P_{\Delta}v^n_T(k+1,\cdot)(x)\right\}
    \end{equation*}
    for $k=0,\dots,2^n-1$ with $v^n_T(2^n,x)=g(x)$. We proceed to prove concavity of $x\mapsto v^n_T(k,x)$ by backward induction. At $k=2^n$ concavity holds by concavity of $g$. For $k<2^n$, assume $x\mapsto v^n_T(k+1,x)$ is concave. Then $P_{\Delta}v^n_T(k+1,\cdot)(x)$ is concave and adding the running cost preserves the concavity. The pointwise minimum of two concave functions is concave. Hence $x\mapsto v^n_T(k,x)$ is concave. By backward induction, it then holds that $x\mapsto v^n_T(k,x)$ is concave for any $k\in\{0,1,\dots,2^n\}$, and so $x\mapsto V^n_T(t,x)$ is concave.

    Next, we pass from the Bermudan approximation to the continuous-time problem. Since $\mathcal{T}^{n}_{T-t}\subseteq\mathcal{T}^{n+1}_{T-t}\subseteq\mathcal{T}_{T-t}$, we have 
    \begin{equation*}
        V^n_T(t,x)\geq V^{n+1}_T(t,x)\geq V_T(t,x)
    \end{equation*}
    for $x\in[0,1]$. Hence, for each $x$, the limit 
    \begin{equation*}
        \overline{V}_T(t,x):=\lim_{n\to\infty}V^n_T(t,x)
    \end{equation*}
    exists and satisfies $\overline{V}_T(t,x)\geq V_T(t,x)$. To prove the converse, fix $\tau\in\mathcal{T}_{T-t}$ and define the dyadic approximation $\tau_n:=\inf\{k\Delta_n:k\Delta_n\geq \tau\}\wedge(T-t)$ where $\Delta_n:=(T-t)2^{-n}$. Then $\tau_n\in\mathcal{T}^n_{T-t}$ gives $\tau_n\searrow \tau$ almost surely. Since $X^x$ has continuous paths and the integrated cost, $u\mapsto F_t(u):=\int_0^uc(t+r)dr$, is continuous, we have 
    \begin{equation*}
        F_t(\tau_n)+g(X^x_{\tau_n})\to F_t(\tau) + g(X^x_\tau),
    \end{equation*}
    and dominated convergence yields
    \begin{equation*}
        \overline{V}_T(t,x)=\lim_{n\to\infty}V^n_T(t,x)\leq\lim_{n\to\infty}\mathbb{E}\left[F_t(\tau_n)+g(X^x_{\tau_n})\right]=\mathbb{E}\left[F_t(\tau)+g(X^x_{\tau})\right].
    \end{equation*}
    Since $\tau\in\mathcal{T}_{T-t}$ was arbitrary, taking the infimum over $\tau$ yields
    \begin{equation*}
        \overline{V}_T(t,x)\leq V_T(t,x).
    \end{equation*}
    Thus $V^n_T(t,x)\searrow V_T(t,x)$ pointwise as $n\to\infty$. Pointwise limits of concave functions preserve concavity, and hence $x\mapsto V_T(t,x)$ is concave.

    For $T=\infty$, we use the pointwise convergence from Proposition \ref{prop:finite_to_infinite_1} and conclude $x\mapsto V_\infty(t,x)$ is concave.
\end{proof}
From concavity, we get local Lipschitz regularity. However, we will need explicit Lipschitz bounds and a probabilistic representation of the derivative in finite horizon, when it exists. When $g'$ is bounded, under Assumption \ref{asmp:bounded:g'}, the argument in the proof of Lemma \ref{lem:concave_V} yielding \eqref{eq:derivative_before_mc} gives the representation by exchanging $t$ with a bounded stopping time. In the case of unbounded $g'$ the argument is less clear. To this end, we prove the following lemma.
\begin{lemma}\label{lem:compare}
    Let $[A,B]\subset(0,1)$ and $T_1<\infty$ be given. For any $\varepsilon>0$ such that $[A-\varepsilon,B+\varepsilon]\subset(0,1)$, there exists a $\delta>0$ such that for all $x,y\in[A,B]$ with $|x-y|<\delta$ it holds that
    \begin{equation*}
        |X^x_t-X^y_t|\leq \varepsilon
    \end{equation*}
    for $t\leq\inf\{s\geq0:X^x_s\not\in(A,B)\}\wedge T_1$.
\end{lemma}
\begin{proof}
    We write 
    \[ 
    \tilde x:=\Psi(x),\qquad \tilde y:=\Psi(y),\qquad \tilde A:=\Psi(A),\qquad \tilde B:=\Psi(B). 
    \]
    Define 
    \begin{equation*}
        D_t := L^{\tilde{x}}_t-L^{\tilde{y}}_t={\tilde{x}}-{\tilde{y}}+\int_0^t(\mu(L^{\tilde{x}}_s)-\mu(L^{\tilde{y}}_s))\, ds.
    \end{equation*}
    Let $\tilde\varepsilon>0$ be such that $[\tilde{A}-\tilde{\varepsilon},\tilde{B}+\tilde{\varepsilon}]\subset \Psi((0,1))$ and consider $\tilde{\rho}=\inf\{s\geq0:L^{\tilde{x}}_s\not\in[\tilde{A},\tilde{B}]\text{ or }L^{\tilde{y}}_s\not\in[\tilde{A}-\tilde{\varepsilon},\tilde{B}+\tilde{\varepsilon}]\}\wedge T_1$. Then, for $t\leq \tilde\rho$, we have
    \begin{equation*}
        |D_t|\leq |\tilde{x}-{\tilde{y}}|+\int_0^t|\mu(L^{\tilde{x}}_s)-\mu(L^{\tilde{y}}_s)|\,ds\leq|{\tilde{x}}-{\tilde{y}}|+K_{\tilde \varepsilon}\int_0^t |D_s|\,ds,
    \end{equation*}
    where $K_{\tilde \varepsilon}$ is a uniform Lipschitz constant for $\mu$ on $[\tilde{A}-\tilde{\varepsilon},\tilde{B}+\tilde{\varepsilon}]$. By Grönwall's inequality, we have
    \begin{equation*}
        |D_t|\leq |{\tilde{x}}-{\tilde{y}}|e^{tK_{\tilde\varepsilon}}\leq |{\tilde{x}}-{\tilde{y}}|e^{T_1K_{\tilde \varepsilon}}.
    \end{equation*}
    So in total, we get 
    \begin{equation*}
        |X^x_t-X^y_t|\leq K_{\Psi^{-1}}|L_t^{\tilde x}-L_t^{\tilde y}|\leq K_{\Psi^{-1}} |{\tilde{x}}-{\tilde{y}}|e^{T_1K_{\tilde \varepsilon}}\leq K_{\Psi^{-1}}K_{\Psi}e^{T_1K_{\tilde \varepsilon}}|x-y|
    \end{equation*}
    for $t\leq \tilde\rho$, where $K_\Psi$ and $K_{\Psi^{-1}}$ are Lipschitz constants for $\Psi$ and $\Psi^{-1}$. Taking
    \begin{equation*}
        \delta\leq \min \left\{\tilde\varepsilon,\frac{\varepsilon}{K_{\Psi^{-1}}}\right\}\frac{1}{K_\Psi}e^{-T_1K_{\tilde\varepsilon}}
    \end{equation*}
    proves the claim.
\end{proof}

Importantly, for a given
$\varepsilon>0$, the constant $\delta>0$ appearing in Lemma \ref{lem:compare} can be chosen \emph{independently} of the sample path.

\begin{proposition}\label{prop:space_lip}
    Let $T<\infty$. For every $t\in[0,T]$ the map $x\mapsto V_T(t,x)$ is $\llip{(0,1)}$. Specifically, for the optimal stopping time $\tau^* :=\tau^*_T(t,x)$, 
    \begin{equation}\label{eq:space_derivative}
        \partial_x V_T(t,x)=\mathbb{E}\left[g'(X^x_{\tau^*})\partial_x X^x_{\tau^*}\right],
    \end{equation}
    for almost every $x\in(0,1)$. In particular, for $x,y\in[a,b]\subset(0,1)$,
    \begin{equation}\label{eq:space_derivative_lipbound}
        |V_T(t,x)-V_T(t,y)|\leq K_\text{space}|x-y|,
    \end{equation}
    where $K_\text{space}:=\|g'\|_\infty$ under Assumption \ref{asmp:bounded:g'}, and there exists an interval $[A,B]\subset(0,1)$ independent of $t$ and $T$ such that $K_\text{space}:=\|g'\|_{L^\infty([a,b]\cup[A,B])}<\infty$ if Assumption \ref{asmp:bounded:c} holds.
\end{proposition}

\begin{proof}
    By Lemma \ref{lem:concave_V}, we get that $x\mapsto V_T(t,x)$ is $\llip{(0,1)}$ since it is concave. We would like to repeat the arguments above from the proof of \eqref{eq:derivative_before_mc}.
    If Assumption \ref{asmp:bounded:g'} holds, then the argument leading to
    \eqref{eq:derivative_before_mc} extends to every bounded stopping time
    \(\tau\in\mathcal{T}_{T-t}\). Indeed, by joint continuity of the stochastic flow,
    the integral remainder formula gives the stopping time equivalent of \eqref{eq:integralremainder_semi},
    \begin{equation*}
        \frac{1}{h}\mathbb E\left[g(X^{x+h}_\tau)-g(X^x_\tau)\right]
        =
        \int_0^1
        \mathbb E\left[
            g'(X^{x+rh}_\tau)Z^{x+rh}_\tau
        \right]dr.
    \end{equation*}
    Arguing as in \eqref{eq:L1conv_semi}, where $\mathbb{E}[Z^{x+r h}_\tau]=\mathbb{E}[Z^{x}_\tau]=1$ by optional sampling on bounded stopping times, we arrive at
    \begin{equation}\label{eq:Eg_deriv_stopping}
        \partial_x \mathbb E[g(X^x_\tau)]
        =
        \mathbb E\left[g'(X^x_\tau)Z^x_\tau\right],        
    \end{equation}
    given that $g'$ is bounded as assumed in Assumption \ref{asmp:bounded:g'}.
    
    Thus, if $(t,x)\in\mathcal{S}_T$ and $x+h\in[0,1]$,
    \begin{equation*}
        V_T(t,x+h)-V_T(t,x)\leq \mathbb{E}\left[g(X^{x+h}_{\tau^*})-g(X^x_{\tau^*})\right]
    \end{equation*}
    where $\tau^*:=\tau^*_T(t,x)$ since $\tau^*$ is admissible for the problem with initial data $(t,x+h)$. Similarly, if $x-h\in[0,1]$, then
    \begin{equation*}
        V_T(t,x)-V_T(t,x-h)\geq \mathbb{E}\left[g(X^{x}_{\tau^*})-g(X^{x-h}_{\tau^*})\right]
    \end{equation*}
    since $\tau^*$ is admissible for the problem with initial data $(t,x-h)$. Dividing by $h$ and taking limits, utilizing \eqref{eq:Eg_deriv_stopping}, we get
    \begin{equation}\label{eq:admiss}
        \partial^+_xV_T(t,x)\leq \mathbb{E}[g'(X^x_{\tau^*})\partial_xX^x_{\tau^*}]\leq \partial^-_xV_T(t,x)
    \end{equation}
    since the left and right derivatives exist by concavity. This in turn yields \eqref{eq:space_derivative} when the derivative exists.    
    
    Now assume Assumption \ref{asmp:bounded:c} and $g'$ is unbounded. Then, letting $\underline{c}:=\inf_{t\geq0}c(t)>0$ and considering the homogeneous infinite-horizon problem
    \begin{equation*}
        \underline{V}(x):=\inf_\tau\mathbb{E}\left[\int_0^\tau \underline{c}\, du+g(X^x_\tau)\right]
    \end{equation*}
    which has the continuation set $(A^*,B^*)$ given by Theorem \ref{VerificationInfConstant}, it holds that $\underline{V}(x)\leq V_T(t,x)$ for all $(t,x)\in\mathcal{S}_T$ so $V_T(t,x)=g(x)$ for all $t$ and $T$ if $x\not\in(A^*,B^*)$. Specifically, in \eqref{eq:OSP_T} it is enough to consider stopping times where $\tau\in\mathcal{T}_{T-t}$ with
    \begin{equation*}
        \tau\leq \inf\{u\geq 0:X^x_u\not\in(A^*,B^*)\}\wedge(T-t),
    \end{equation*}
    and hence, by Lemma \ref{lem:compare}, for small enough $h>0$ and all $u\leq\tau$, we have $X^{x+r h}_u\in[A^*-\varepsilon,B^*+\varepsilon]\subset(0,1)$ for some $\varepsilon>0$ and so $|g'(X^{x+rh}_\tau)|<\|g'\|_{L^{\infty}([A^*-\varepsilon,B^*+\varepsilon])}$
    for all $r\in[0,1]$ if $x\in(A^*,B^*)$.
    Thus, by boundedness, and since we consider bounded stopping times, we can also apply the admissibility arguments from \eqref{eq:admiss} and from the proof of \eqref{eq:derivative_before_mc} to this case and get \eqref{eq:space_derivative} for points of differentiability.

    For the Lipschitz bound, we first take $t\leq T$, $x$ and $x+h$ in $(A^*,B^*)$ for some $h>0$, and let $\tau^*:=\tau^*_T(t,x)$. Since $\tau^*$ is admissible for initial data $(t,x+h)$, we have 
    \begin{equation*}
        \frac{V_T(t,x+h)-V_T(t,x)}{h}\leq \frac{1}{h}\mathbb{E}\left[g(X^{x+h}_{\tau^*})-g(X^x_{\tau^*})\right]=\int_0^1\mathbb{E}\left[g'(X^{x+rh}_{\tau^*})\partial_xX^{x+rh}_{\tau^*}\right]dr\leq K_\text{space},
    \end{equation*}
    by Fubini and since $|g'(X^{x+rh}_{\tau^*})|<K_\text{space}$, $\partial_xX^{x+rh}_{\tau^*}>0$, and $\mathbb{E}[\partial_xX^{x+rh}_{\tau^*}]=1$ by Assumption \ref{asmp:martingale} since $\tau^*\leq T-t$. Similarly $\tau^*_h:=\tau^*_T(t,x+h)$ is admissible for initial data $(t,x)$ so
    \begin{equation*}
        \frac{V_T(t,x+h)-V_T(t,x)}{h}\geq \frac{1}{h}\mathbb{E}\left[g(X^{x+h}_{\tau^*_h})-g(X^x_{\tau^*_h})\right]=\int_0^1\mathbb{E}\left[g'(X^{x+rh}_{\tau^*_h})\partial_xX^{x+rh}_{\tau^*_h}\right]dr\geq -K_\text{space}.
    \end{equation*}
    Taking limsup and liminf, respectively, we get 
    \begin{equation*}
        |\partial_x^+V_T(t,x)|\leq K_\text{space},
    \end{equation*}
    where $\partial_x^+$ denotes the right derivative.
    This bound indeed holds for all $x\in[a,b]$ by the definition of $K_\text{space}$ and $V_T(t,x)=g(x)$ if $x\not\in(A^*,B^*)$ given Assumption \ref{asmp:bounded:c}. A similar argument for the left derivative gives \eqref{eq:space_derivative_lipbound} as claimed. By the same method, we can achieve the bound $K_\text{space}=\|g'\|_\infty$ under Assumption \ref{asmp:bounded:g'}.
\end{proof}
In the proof, we had to split the cases of Assumption \ref{asmp:bounded:g'} and \ref{asmp:bounded:c} since we wanted the Lipschitz constant to be independent of the time horizon. But we may use the same trick to bound by a homogeneous infinite-horizon problem under either of the two assumptions to get a horizon-dependent estimate of the problem. That is, for finite $T$, define $\underline{c}_T:=\inf_{t\leq T}c(t)$ and consider the associated homogeneous infinite-horizon problem denoted by $\underline{V}(\cdot;T)$. It holds that 
\begin{equation}\label{eq:hom_bound_T_value}
    \underline{V}(x;T)\leq V_T(t,x),
\end{equation}
and so the stopping set of $\underline{V}(\cdot;T)$ is contained in the stopping set of $V_T$. This in turn implies that 
\begin{equation}\label{eq:hom_bound_T}
    X^x_s\in[A^*(T),B^*(T)]\subset(0,1),
\end{equation}
    provided $x\in[A^*(T),B^*(T)]$, for all $s\leq\tau^*_T(t,x)$, when $A^*(T)$ and $B^*(T)$ denote the optimal stopping boundaries for $\underline{V}(x;T)$ given in Theorem \ref{VerificationInfConstant}, and $\tau^*_T(t,x)$ denotes the optimal stopping time of $V_T$.

Using this, we can translate Proposition \ref{prop:space_lip} into the Lagrange-formulated problem. 
\begin{corollary}\label{cor:space_lip_W}
    Let $T<\infty$. For every $t\in[0,T]$ the map $x\mapsto W_T(t,x)$ is $\llip{(0,1)}$. Specifically, for the optimal stopping time $\tau^* :=\tau^*_T(t,x)$, 
    \begin{equation*}%
        \partial_x W_T(t,x)=\mathbb{E}\left[\int_0^{\tau^*}\partial_x X^x_{u}\partial_x(\mathcal{L}g)(X^x_u)\,du\right],
    \end{equation*}
    for almost every $x\in(0,1)$. In particular, for $x,y\in[a,b]\subset(0,1)$,
    \begin{equation}\label{eq:space_derivative_lipbound_W}
        |W_T(t,x)-W_T(t,y)|\leq 2K_\text{space}|x-y|,
    \end{equation}
    where $K_\text{space}$ is defined in Proposition \ref{prop:space_lip}.
\end{corollary}
\begin{proof}
    Since $W_T(t,x)=V_T(t,x)-g(x)$, \eqref{eq:space_derivative_lipbound_W} follows from \eqref{eq:space_derivative_lipbound} by the triangle inequality in a neighborhood as in Proposition \ref{prop:space_lip}. Assume $(t,x)$ is a point of differentiability for $V_T$. Then 
    \begin{equation*}
        \partial_xW_T(t,x)=\mathbb{E}\left[ g'(X^x_{\tau^*})\partial_x X^x_{\tau^*}\right]-g'(x).
    \end{equation*}
    Let $f(x,z):=g'(x)z$, and define
    \begin{equation*}
        \mathcal{A}:=\sigma'(x)\sigma(x)z\partial_{xz}+\frac{\sigma^2(x)}{2}\partial_{xx}+\frac{(\sigma'(x)z)^2}{2}\partial_{zz}.
    \end{equation*}
    Using  \eqref{eq:hom_bound_T}, we see that $f$ restricted to $[A^*(T),B^*(T)]$ is in the domain of $\mathcal{A}$. Since $\tau^*$ is finite, we may apply Dynkin's formula to obtain
    \begin{equation*}
        \partial_xW_T(t,x)=\mathbb{E}\left[\int_0^{\tau^*}\mathcal{A}\big(g'(X^x_u)\partial_xX^x_u\big) \,du\right]=\mathbb{E}\left[\int_0^{\tau^*}\partial_x X^x_{u}\partial_x(\mathcal{L}g)(X^x_u)\,du\right],
    \end{equation*}
    where the second equality follows by direct computation.
\end{proof}

We now turn to the time derivative, where we also want a probabilistic expression for the derivative in the infinite-horizon case, when it exists. Recall the exponential growth condition on the running cost from Assumption \ref{asmp:base_penalty:c}.

\begin{proposition}\label{prop:time_lip}
    For any $T\leq\infty$ and $x\in[0,1]$, it holds that $t\mapsto V_T(t,x)$ is $\llip{[0,T)}$. Specifically, at points of differentiability it holds that
    \begin{equation}\label{eq:time_deriv_bound}
        \mathbb{E}[c(t+\tau^*)]-c(t)\leq \partial_tV_T(t,x)\leq \mathbb{E}[c(t+\tau^*)]-c(t) + K''\mathbb{P}(\tau^*=T-t),
    \end{equation}
     for $T<\infty$, where $\tau^*:=\tau^*_T(t,x)$ is the optimal stopping time and $K'':=\|\mathcal{L}g\|_\infty$, and 
     \begin{equation}\label{eq:time_deriv_inf}
        \partial_tV_\infty(t,x)=\mathbb{E}[c(t+\tau^*_\infty(t,x))]-c(t),
    \end{equation}
    when $T=\infty$, with the convention $c(t+\infty):=0$, which is well-defined. In particular, for $T\leq\infty$ and $|s-t|\leq 1$,
     \begin{equation}\label{eq:time_derivative_lipbound}
        |V_T(t,x)-V_T(s,x)|\leq K_{\text{time}}|s-t|,
    \end{equation}
    where $K_{\text{time}}:=K'\|g\|_\infty+K''$.
\end{proposition}
\begin{proof}
    We first take $T<\infty$. If $x\in\{0,1\}$, then the statement holds since $\tau^*_T(t,x)=0$ almost surely, so fix any $x\in(0,1)$. We will show the result through the Lagrange-formulated problem $W_T$ defined in \eqref{eq:lagrange_formulation}.
    We will first prove \eqref{eq:time_deriv_bound} by bounding the left and right derivatives from below and above, respectively. 

    Let $t\in(0,T)$ and $h>0$ such that $t-h\geq0$, and assume, for simplicity, that $h\leq1$. Denote the optimal stopping time by $\tau^*:=\tau^*_T(t,x)$. Then $\tau^*$ is admissible for the problem with initial data $(t-h,x)$. We have 
    \begin{equation*}
        \frac{W_T(t,x)-W_T(t-h,x)}{h}\geq \mathbb{E}\left[\int_0^{\tau^*}\frac{c(t+u)-c(t-h+u)}{h} \,du\right]=\mathbb{E}\left[\int_0^{\tau^*}\int_0^1 c'(t+u-rh)\,dr\, du\right],
    \end{equation*}
    by the fundamental theorem of calculus. By \eqref{eq:cprime_helper}, we have 
    \begin{equation*}
        |c'(t+u-rh)|\leq K'c(t+u),
    \end{equation*}
    and since $\tau^*\in\mathcal{T}_{T-t}\subset\mathcal{T}$, we have $\mathbb{E}[\int_0^{\tau^*}c(t+u)\,du]<\infty$, and
    \begin{equation*}
        \int_0^1c'(t+u-rh)\,dr\to c'(t+u)
    \end{equation*}
    as $h\searrow0$,
    we can apply dominated convergence theorem to get 
    \begin{equation}\label{eq:dt_left_lower}
        \liminf_{h\searrow0}\frac{W_T(t,x)-W_T(t-h,x)}{h}\geq 
        \mathbb{E}\left[\int_0^{\tau^*} c'(t+u) \,du\right]=\mathbb{E}[c(t+\tau^*)]-c(t).
    \end{equation}
    For the upper bound on the right derivative we let $t\in(0,T)$ and $h\in(0,1]$ such that $t+h\leq T$. We note that $\tau^*$ may be inadmissible for the problem with initial data $(t+h,x)$. Thus let $u_h=T-(t+h)$ and consider the stopping time $\tau^*\wedge u_h$. First, by the strong Markov property, the dynamic programming principle for the Lagrange-formulated problem yields
    \begin{equation*}
        W_T(t,x)=\mathbb{E}\left[\int_0^{\tau^*\wedge u_h}\bigl(c(t+u)+\mathcal{L}g(X^x_u)\bigr)\,du+\ind{\tau^*>u_h}W_T(T-h,X^x_{u_h})\right],
    \end{equation*}
    and hence
    \begin{equation*}
        W_T(t+h,x)-W_T(t,x)\leq \mathbb{E}\left[\int_0^{\tau^*\wedge u_h}\big(c(t+h+u)-c(t+u) \big)\,du-\ind{\tau^*>u_h}W_T(T-h,X^x_{u_h})\right].
    \end{equation*}
    For any $y\in(0,1)$, it holds that 
    \begin{equation*}
        0\leq -\frac{1}{h}W_T(T-h,y)\leq \frac{1}{h}\mathbb{E}\left[\int_0^h \|\mathcal{L}g\|_\infty \,du\right]=K''
    \end{equation*}
    since $-c(u)\leq 0$ for all $u\geq0$ and thus
    \begin{equation}\label{eq:W_tderiv_truncated}
        \frac{W_T(t+h,x)-W_T(t,x)}{h}\leq \mathbb{E}\left[\int_0^{\tau^*\wedge u_h} \int_0^1c'(t+u+rh) \,dr\,du\right]+K''\mathbb{P}(\tau^*>u_h).
    \end{equation}
    Since $u_h\nearrow T-t$ and
    \begin{equation*}
        \int_0^1\ind{u\leq u_h}c'(t+u+rh)dr\to \ind{u\leq T-t}c'(t+u),
    \end{equation*}
    by dominated convergence, we get
    \begin{equation}\label{eq:dt_right_upper}
        \limsup_{h\searrow0}\frac{W_T(t+h,x)-W_T(t,x)}{h}\leq \mathbb{E}\left[\int_0^{\tau^*} c'(t+u)\,du\right]+K''\mathbb{P}(\tau^*=T-t).
    \end{equation}
    Thus, if $(t,x)$ is a point of differentiability, we have \eqref{eq:time_deriv_bound} by \eqref{eq:lagrange_formulation} since $g(x)$ is independent of $t$.

    To prove the local Lipschitz property we need to bound the right and left derivatives from above and below, respectively. For the left derivative let again $t\in(0,T)$ and $h\in(0,1]$ such that $t-h\geq0$. Let $\tau^*_h:=\tau^*_T(t-h,x)$, then $\tau^*_h\wedge (T-t)$ is admissible for the problem with initial data $(t,x)$. Hence, similar to \eqref{eq:W_tderiv_truncated}, we get 
    \begin{equation*}
        \frac{W_T(t,x)-W_T(t-h,x)}{h}\leq \mathbb{E}\left[\int_0^{\tau^*_h\wedge (T-t)} \int_0^1c'(t+u-rh) \,dr\,du\right]+K''\mathbb{P}(\tau^*_h>T-t).
    \end{equation*}
    By \eqref{eq:cprime_helper}, we get
    \begin{align*}
        &\left|\mathbb{E}\left[\int_0^{\tau^*_h\wedge (T-t)} \int_0^1c'(t+u-rh) \,dr\,du\right]\right|\leq\mathbb{E}\left[\int_0^{\tau^*_h\wedge (T-t)} \int_0^1\left|c'((t-h+u)+(1-r)h)\right| \,dr\,du\right]\\
        \leq&K'\mathbb{E}\left[\int_0^{\tau^*_h\wedge (T-t)} \int_0^1c(t-h+u) \,dr\,du\right]\leq K'\mathbb{E}\left[\int_0^{\tau^*_h}c(t-h+u)\,du\right]\leq K'\|g\|_\infty
    \end{align*}
    since $(1-r)h\leq 1$ and 
    \begin{equation*}
        0\leq \mathbb{E}\left[\int_0^{\tau^*_T(s,y)}c(s+u)\,du\right]\leq \sup_{(s,y)\in\mathcal{S}_T}V_T(s,y)\leq \|g\|_\infty
    \end{equation*}
    for any $s\in[0,T)$.
    Thus, since $\mathbb{P}(\tau^*_h>T-t)\leq 1$, and recalling $K_\text{time}=K'\|g\|_\infty+K''$, we have
    \begin{equation}\label{eq:dt_left_upper}
        \left|\mathbb{E}\left[\int_0^{\tau^*_h\wedge (T-t)} \int_0^1c'(t+u-rh) \,dr\,du\right]+K''\mathbb{P}(\tau^*_h>T-t)\right|\leq K_\text{time}.
    \end{equation}
    Similarly, for the right derivative, letting $\tau^*_h:=\tau^*_T(t+h,x)$, we have 
    \begin{equation}\label{eq:dt_right_lower}
        \frac{W_T(t+h,x)-W_T(t,x)}{h}\geq \mathbb{E}\left[\int_0^{\tau^*_h}\int_0^1 c'(t+u+rh)\,dr\, du\right]\geq -K'\|g\|_\infty,
    \end{equation}
    by \eqref{eq:cprime_helper}.
    Combining \eqref{eq:dt_left_lower},\eqref{eq:dt_right_upper}, \eqref{eq:dt_left_upper}, and \eqref{eq:dt_right_lower}, we get 
    \begin{equation*}
        |W_T(t,x)-W_T(s,x)|\leq K_\text{time}|s-t|,
    \end{equation*}
    for $|s-t|\leq 1$, which again is equivalent to \eqref{eq:time_derivative_lipbound} by \eqref{eq:lagrange_formulation} since $g(x)$ is independent of $t$. Using the argument around \eqref{eq:extend_0}, we can extend the conclusion to $t=0$.

    If $T=\infty$, then the inadmissibility corrections can be ignored in the proof above. Hence for the infinite-horizon problem, we get \eqref{eq:time_deriv_inf} almost everywhere and \eqref{eq:time_derivative_lipbound}. Note that this derivative is well-defined since if $\tau_\infty^*(t,x)=\infty$ with positive probability, then $c(s)\to0$ as $s\to\infty$ because $\tau_\infty^*(t,x)\in\mathcal{T}$. 
\end{proof}

Combining Propositions \ref{prop:space_lip} and \ref{prop:time_lip}, we get that $(t,x)\mapsto V_T(t,x)$ is locally Lipschitz.
\begin{theorem}\label{thm:lip_value}
    The function $(t,x)\mapsto V_T(t,x)$ is $\llip{\mathcal{S}_T}$ for $T\leq \infty$.
\end{theorem}
\begin{proof}
    First, let $T<\infty$ and consider $(t_n,x_n)\to(t,x)\in[0,T)\times(0,1)$ such that $\lvert t_n-t\rvert<1$ and $x_n\in[a,b]\subset(0,1)$ for all $n\in\mathbb{N}$. By \eqref{eq:time_derivative_lipbound} and \eqref{eq:space_derivative_lipbound}, we have
    \begin{align*}
        \lvert V_T(t_n,x_n)-V_T(t,x)\rvert&\leq \lvert V_T(t_n,x_n)-V_T(t,x_n)\rvert+\lvert V_T(t,x_n)-V_T(t,x)\rvert\\
        &\leq K_\text{time}|t_n-t|+K_\text{space}|x_n-x|\leq 2(K_\text{time}+K_\text{space})\|(t_n,x_n)-(t,x)\|_2,
    \end{align*}
    which, we note, is uniform in $T$. Thus, for the infinite-horizon problem we have
    \begin{equation*}
        |V_\infty (t,x)-V_\infty(s,y)|=\lim_{T\to\infty}|V_T (t,x)-V_T(s,y)|\leq 2(K_{\text{time}}+K_{\text{space}})\|(t,x)-(s,y)\|_2,
    \end{equation*}
    by Proposition \ref{prop:finite_to_infinite_1} and so we are done.
\end{proof}

With the Lipschitz regularity of the value function obtained in this section we are able to strengthen the pointwise convergence result of Proposition \ref{prop:finite_to_infinite_1}.
\begin{corollary}\label{cor:unif_conv_T}
   For every $\mathcal{I}=[t_1,t_2]\subset[0,\infty)$ and $\mathcal{J}=[\underline{a},\overline{a}]\subset(0,1)$ it holds that
    \begin{equation*}
        V_T\to V_\infty
    \end{equation*}
    uniformly on $\mathcal{I}\times\mathcal{J}$ for $T\to\infty$.
\end{corollary}
\begin{proof}
    Proposition \ref{prop:finite_to_infinite_1} yields pointwise monotone convergence, so Dini's theorem yields the result since $(t,x)\mapsto V_\infty (t,x)$ is continuous by Theorem \ref{thm:lip_value}.
\end{proof}

The conclusions of Proposition \ref{prop:space_lip} and \ref{prop:time_lip}, and Theorem \ref{thm:lip_value} can be directly translated into the Lamperti-transformed problem since $\Psi^{-1}$ is locally regular. We present the result needed for the Lagrange formulation of the Lamperti-transformed problem on finite horizon.

\begin{corollary}\label{cor:W_T_lips}
   Let $T<\infty$. The function $(t,l)\mapsto \tilde W_T(t,l)$ is $\llip{\tilde{\mathcal{S}}_T}$. Specifically, at points of differentiability it holds that
    \begin{equation}\label{eq:space_derivative_lamperti_w}
        \partial_l \tilde{W}_T(t,l)=\mathbb{E}\left[\int_{0}^{\tilde\tau^*}\partial_l L^l_{u}\partial_l(\tilde{\mathcal{L}}\tilde{g})(L^l_{u})\,du\right]
    \end{equation}
    where $\tilde\tau^*:=\tilde\tau^*_T(t,l)$ is the optimal stopping time, and 
    \begin{equation}\label{eq:time_lamp_bound}
        \mathbb{E}[c(t+\tilde\tau^*)]-c(t)\leq \partial_t\tilde{W}_T(t,l)\leq  \mathbb{E}[c(t+\tilde\tau^*)]-c(t) + K''\mathbb{P}(\tilde\tau^*=T-t),
    \end{equation}
     where $K''$ is defined in Proposition \ref{prop:time_lip}.
\end{corollary}
\begin{proof}
    The first two claims hold by applying the chain rule to Theorem \ref{thm:lip_value} and Corollary \ref{cor:space_lip_W}. The latter holds by Proposition \ref{prop:time_lip}.
\end{proof}

The bounds on the time derivative in \eqref{eq:time_deriv_bound} and \eqref{eq:time_lamp_bound} are useful for the analysis below. However, we can in fact get an exact stochastic representation for the time derivative in the finite-horizon case as well, by a technique similar to \cite{jaillet1990variational}.
\begin{theorem}\label{thm:time_deriv_explicit}
    Let $T<\infty$. It holds that
    \begin{align}\label{eq:time_derivative_finite_horizon_org}
        \partial_t V_T(t,x)=&\mathbb{E}[c(t+\tau^*)]-c(t)\\&-\frac{1}{T-t}\mathbb{E}\left[c(t+\tau^*)\tau^*-\frac{1}{2}g'(X^x_{\tau^*})\partial_xX^x_{\tau^*}\left\{\int_0^{\tau^*} \tfrac{\sigma'(X^x_s)\sigma(X^x_s)}{\partial_xX^x_s}\,ds-\int_0^{\tau^*}\tfrac{\sigma(X^x_s)}{\partial_xX^x_s}\,dB_s\right\}\right],\nonumber
    \end{align}
    for almost every $(t,x)\in\mathcal{S}_T$ where $\tau^*:=\tau^*_T(t,x)$.
\end{theorem}
\begin{proof}
    Take any $(t,l)\in\tilde{\mathcal{S}}_T$. Let $\mathcal{T}^{(T-t)}_{[0,1]}$ denote the set of stopping times with respect to the filtration $(\mathcal{F}_{s(T-t)})_{s\in[0,1]}$. Then
    \begin{align}
        \tilde{V}_T(t,l)&=\inf_{\tau\leq T-t}\mathbb{E}\left[\int_0^\tau c(t+u)\,du + \tilde{g}(L_\tau^{l})\right]\label{eq:first_OSP_I}\\
        &=\inf_{\theta\in\mathcal{T}^{(T-t)}_{[0,1]}}\mathbb{E}\left[\int_0^{\theta(T-t)} c(t+u)\,du + \tilde{g}(L_{\theta(T-t)}^{l})\right]\\
        &=\inf_{\theta\in\mathcal{T}^{(T-t)}_{[0,1]}}\mathbb{E}\left[\int_0^\theta (T-t)c(t+(T-t)u)\,du + \tilde{g}(\sqrt{T-t}Y_{\theta}^{t,l/\sqrt{T-t}})\right]\label{eq:second_OSP_I}
    \end{align}
    and $Y_{s}^{t,i}:=L_{s(T-t)}^{i\sqrt{T-t}}/\sqrt{T-t}$, which solves
    \begin{equation*}
        Y_{s}^{t,i}=i+\int_0^s\sqrt{T-t}\mu(\sqrt{T-t}Y_{u}^{t,i})\,du+W^t_s
    \end{equation*}
    where $W^t_s:=B_{s(T-t)}/\sqrt{T-t}$, so $(W^t_s)_{s\in[0,1]}$ is a Brownian motion.
    There is a one-to-one correspondence between the optimal stopping times of \eqref{eq:first_OSP_I} and \eqref{eq:second_OSP_I}.

    Consider a new Wiener space $(\Omega,\mathcal{G},\mathbb{G},\mathbb{Q})$ carrying a Brownian motion $W=(W_t)_{t\geq0}$, and define the process $I^{t,i}$ as the unique strong solution to 
    \begin{equation*}
        I_{s}^{t,i}=i+\int_0^s\sqrt{T-t}\mu(\sqrt{T-t}I_{u}^{t,i})\,du+W_s.
    \end{equation*}
    Letting $\mathcal{T}_{[0,1]}$ be the stopping times with respect to $\mathbb{G}$ restricted to $[0,1]$, we can consider the auxiliary optimal stopping problem
    \begin{equation}\label{eq:osp_I}
        V^I(t,i):=\inf_{\theta\in\mathcal{T}_{[0,1]}}\mathbb{E}^{\mathbb{Q}}\left[\int_0^\theta (T-t)c(t+(T-t)u)\,du + \tilde{g}(\sqrt{T-t}I_{\theta}^{t,i})\right].
    \end{equation}
    By a similar argument to Proposition \ref{prop:semicontinuity} there exists a smallest optimal stopping time $\theta^*$ that achieves the infimum in \eqref{eq:osp_I}, which is a first hitting time. Note that distributional equivalence yields $V^I(t,i)=\tilde{V}_T(t,i\sqrt{T-t})$, and so the optimal stopping time is given by
    \begin{equation*}
        \theta^*(t,i)=\inf\{s\in[0,1):(t+(T-t)s,\sqrt{T-t}I_s^{t,i})\in\tilde{\mathcal{D}}_T\}\wedge1.
    \end{equation*}
    Thus again by the distributional equivalence of paths, we may conclude that $(I^{t,i},\theta^*)$ is equal in law to $(Y^{t,i},\tilde{\tau}^*/(T-t))$ for $\tilde{\tau}^*:=\tilde{\tau}^*_T(t,i\sqrt{T-t})$. Furthermore, by the argument around \eqref{eq:hom_bound_T}, we have 
    \begin{equation*}
        I_s^{t,i}\in\frac{\Psi([A^*(T),B^*(T)])}{\sqrt{T-t}}\subset\frac{\Psi((0,1))}{\sqrt{T-t}}
    \end{equation*}
    for all $s\leq \theta^*(t,i)$ when $i$ is within the interval and similarly in the Lamperti and original coordinates. In the following analysis we restrict to such an interval, which ensures boundedness by Assumption \ref{asmp:sigma_g}.
    
    Differentiating the flow $t\mapsto I_s^{t,i}$ as in \eqref{eq:partiall_explicit} away from $T$, we get
    \begin{align*}
        \partial_tI^{t,i}_s=\int_0^s&\quad\overbrace{-\frac{1}{2\sqrt{T-t}}\mu(\sqrt{T-t}I^{t,i}_u)-\frac{1}{2}\mu'(\sqrt{T-t}I^{t,i}_u)I^{t,i}_u}^{\eta(t,I^{t,i}_u)}\\
        &+(T-t)\mu'(\sqrt{T-t}I^{t,i}_u)\partial_tI^{t,i}_u \,du.
    \end{align*}
    Defining $Z(s)=\partial_tI^{t,i}_s$, we get the equivalent differential equation $Z(0)=0$ and
    \begin{equation*}
        Z'(s)=\eta(t,I^{t,i}_s)+(T-t)\mu'(\sqrt{T-t}I^{t,i}_s)Z(s)
    \end{equation*}
    i.e.
    \begin{equation*}
        \partial_tI^{t,i}_s=\int_0^s\eta(t,I^{t,i}_u)e^{\int_u^s(T-t)\mu'(\sqrt{T-t}I^{t,i}_v)dv}\,du.
    \end{equation*}
    Similarly, for the flow $i\mapsto I^{t,i}_s$ we get
    \begin{equation*}
        \partial_i I^{t,i}_s=\exp\left\{\int_0^s(T-t)\mu'(\sqrt{T-t}I^{t,i}_v)dv\right\}.
    \end{equation*}
    The function $(t,i)\mapsto V^I(t,i)=\tilde{V}_T(t,i\sqrt{T-t})$ is locally Lipschitz by an argument similar to Theorem \ref{thm:lip_value}. Let $(t,i)$ be a point of differentiability. For fixed $\theta$, let $J(t,i;\theta)$ denote the objective in \eqref{eq:osp_I}. If $\theta^*$ is optimal at $(t,i)$, then localization to a compact interval strictly containing the homogeneous bound and dominated convergence show that $J(\cdot,\cdot;\theta^*)$ is differentiable at $(t,i)$. For $h>0$,
    \[
        \frac{V^I(t+h,i)-V^I(t,i)}{h}
        \leq
        \frac{J(t+h,i;\theta^*)-J(t,i;\theta^*)}{h},
    \]
    while division by $h<0$ reverses the inequality. At a differentiability point the two one-sided limits agree, so $\partial_tV^I(t,i)=\partial_tJ(t,i;\theta^*)$. The same argument applies to the spatial derivative. Applying these identities and differentiating the objective in \eqref{eq:osp_I}, we get
    \begin{equation*}
        \begin{aligned}
            \partial_tV^I(t,i)=\mathbb{E}^{\mathbb{Q}}\Big[&\int_0^{\theta^*}(T-t)c'(t+(T-t)u)(1-u)-c(t+(T-t)u) \,du\\
            &+\tilde{g}'(\sqrt{T-t}I^{t,i}_{\theta^*})\left(\sqrt{T-t}\partial_tI^{t,i}_{\theta^*}-\frac{I^{t,i}_{\theta^*}}{2\sqrt{T-t}}\right)\Big]
        \end{aligned}    
    \end{equation*}
    where $\theta^*$ is the optimal stopping time for initial data $(t,i)$, and we note that the integral term can be simplified using integration by parts,
    \begin{equation*}
        \int_0^s(T-t)c'(t+(T-t)u)(1-u)\,du=\int_0^{s(T-t)}c'(t+u)\,du-c(t+s(T-t))s+\int_0^sc(t+(T-t)u)\,du.
    \end{equation*}
    Similarly, for the spatial derivative of $V^I$ we get 
    \begin{equation*}
        \partial_iV^I(t,i)=\mathbb{E}^{\mathbb{Q}}\left[\tilde{g}'(\sqrt{T-t}I^{t,i}_{\theta^*})\left(\sqrt{T-t}\partial_iI^{t,i}_{\theta^*}\right)\right].
    \end{equation*}
    To revert the transformation,
    we first note that by the chain rule away from $T$, we have
    \begin{align*}
        &\partial_t\tilde V_T(t,l)=\partial_tV^I\left(t,\frac{l}{\sqrt{T-t}}\right)+\frac{l}{2(T-t)^{3/2}}\partial_iV^I\left(t,\frac{l}{\sqrt{T-t}}\right)\\
        =&\mathbb{E}^{\mathbb{Q}}\Bigg[\tilde g'(\sqrt{T-t}I^{t,i}_{\theta^*})\left(\sqrt{T-t}\partial_tI^{t,i}_{\theta^*}-\frac{I^{t,i}_{\theta^*}}{2\sqrt{T-t}}+\frac{l}{2(T-t)}\partial_i I^{t,i}_{\theta^*}\right)-c(t+\theta^*(T-t))\theta^*\\&\qquad+\int_0^{\theta^*(T-t)}c'(t+u)\,du\Bigg],
    \end{align*}
    where $i=l/\sqrt{T-t}$.
    Furthermore, recalling \eqref{eq:partiall_explicit}, we note that for fixed $t$ the quantity $\partial_i I^{t,l/\sqrt{T-t}}_{s}$ has the same law under $\mathbb{Q}$ as $\partial_l L^{l}_{s(T-t)}$ has under $\mathbb{P}$. Similarly, $\partial_t I^{t,l/\sqrt{T-t}}_s$ has the same law under $\mathbb{Q}$ as
    \begin{equation*}
        Z^l_{s(T-t)}:=-\frac{1}{2(T-t)^{3/2}}\int_0^{s(T-t)}\frac{\partial_lL^{l}_{s(T-t)}}{\partial_lL^{l}_{u}}\left(\mu(L^{l}_{u})+L^l_{u}\mu'(L^{l}_{u})\right)\,du
    \end{equation*}
    has under $\mathbb{P}$. This can indeed be extended to path laws, jointly with the optimal stopping time. Namely, $(I^{t,i},\partial_iI^{t,i},\partial_tI^{t,i},\theta^*)$ under $\mathbb{Q}$ has the same law as 
    \begin{equation*}
        \left(\frac{L^l_{\cdot(T-t)}}{\sqrt{T-t}},\partial_l L^l_{\cdot(T-t)},Z^l_{\cdot(T-t)},\frac{\tilde{\tau}^*}{T-t}\right)
    \end{equation*}
    has under $\mathbb{P}$.
    Thus, we can express $\partial_t\tilde V_T(t,l)$ by
    \begin{equation*}
        \partial_t \tilde V_T(t,l)=\mathbb{E}[c(t+\tau^*)]-c(t)-\frac{1}{T-t}\mathbb{E}\left[c(t+\tau^*)\tau^*+\frac{1}{2}\tilde g'(L^l_{\tau^*})\partial_lL^l_{\tau^*}\Xi_{\tau^*}\right],
    \end{equation*}
    where $\Xi$ is given by 
    \begin{equation*}
        \Xi_t = \frac{L^l_t}{\partial_lL^l_t}-l+\int_0^t\frac{\mu(L^l_u)+L^l_u\mu'(L^l_u)}{\partial_l L^l_u}\,du=2\int_0^t\frac{\mu(L^l_u)}{\partial_lL^l_u}\,du+\int_0^t \frac{1}{\partial_lL^l_u}dB_u,
    \end{equation*}
    and we again note that the preceding inclusion implies bounded coefficients for $\Xi$ on $s\leq \tilde{\tau}^*$.
    Rewriting in terms of $\partial_t V_T(t,x)=\partial_t \tilde{V}_T(t,\Psi(x))$, we get the claim.
\end{proof}
The probabilistic representation of the temporal derivative is nice but does not play a role in the following analysis of the optimal stopping problem, since the bound in Corollary \ref{cor:W_T_lips} is easier to control. However, it does enter into the probabilistic representation of the second spatial derivative as seen in Corollary \ref{cor:main_discont_dxx}.

Before going on to the regularity of the optimal stopping boundary, we conclude this section with a classical result on the regularity of the value function.
\begin{proposition}\label{prop:PDE}
The value function $V_T$ satisfies
\begin{align}
    (\partial_t+\mathcal{L})V_T(t,x)=-c(t),\qquad &(t,x)\in\mathcal{C}_T,\label{eq:PDE}\\
    V_T(t,x)=g(x),\qquad &(t,x)\in\mathcal{D}_T,\label{eq:inst_stop}
\end{align}
for any $T\leq\infty$.
Specifically, the functions $\partial_tV_T, \partial_xV_T$, and $\partial_{xx}V_T$ are jointly continuous in $(t,x)\in\mathcal{C}_T$.
\end{proposition}
\begin{proof}
Since $\mathcal{C}_T$ is open by Proposition \ref{prop:semicontinuity} and $(t,x)\mapsto V_T(t,x)$ is continuous by Theorem \ref{thm:lip_value}, the PDE \eqref{eq:PDE} follows from standard Markovian optimal stopping arguments applied locally; see for example the methodology in \cite[Theorem 2.7.7]{karatzas1998methods}. The condition \eqref{eq:inst_stop} follows from the definition of $\mathcal{D}_T$.
\end{proof}

Equipped with Proposition \ref{prop:PDE} we can specify the points of differentiability of $V_T$, namely, $V_T$ is differentiable away from the boundary. This indeed also yields the first part of Corollary \ref{cor:main_discont_dxx}.
\begin{corollary}\label{cor:main_discont_dxx_pt1}
    The second spatial derivative of the value function admits the representation
    \begin{equation*}
        \partial_{xx} V_T(t,x)=\begin{cases}
            -2\frac{c(t)+\partial_tV_T(t,x)}{\sigma^2(x)} & \text{for }(t,x)\in\mathcal{C}_T\\
            g''(x)&\text{for }(t,x)\in(\mathcal{D}_T)^\circ.
        \end{cases}
    \end{equation*}
\end{corollary}
\begin{proof}
    The expression for $\partial_{xx} V_T(t,x)$ in $(\mathcal{D}_T)^\circ$ is obvious.
    In the continuation region the value function $V_T$ obeys the partial differential equation \eqref{eq:PDE}, and thus the expression for $\partial_{xx} V_T$ in $\mathcal{C}_T$ can easily be obtained.
\end{proof}

\subsection{Properties of the optimal stopping boundary}

Next, we investigate the structure of the continuation set. We first show that the continuation set, and thus the optimal stopping boundary, can be described by at most two functions, partially showing Theorem \ref{thm:main_structure}.
\begin{proposition}\label{prop:two_boundaries}
    Let $T\leq\infty$. There exists a lower semicontinuous function $b_T:[0,T)\to[1/2,1]$ such that the continuation region is given by
    \begin{equation}\label{eq:continuation_set_b}
        \mathcal{C}_T=\{(t,x)\in\mathcal{S}_T : 1-b_T(t)<x<b_T(t)\}.
    \end{equation}
    Furthermore, $b_T(t)\geq\gamma(t)$ for any $t\in[0,T)$.
\end{proposition}
\begin{proof}
    First, let $T<\infty$. We prove the equivalent statement in the Lamperti coordinates.
    Note first that $\tilde{\mathcal{L}}\tilde{g}(l)=\mathcal{L}g(\Psi^{-1}(l))$, and hence, by Assumption \ref{asmp:Lg:unimodal}, Assumption \ref{asmp:dxLg}, and Assumption \ref{asmp:sym}, it holds that $\partial_l\tilde{\mathcal{L}}\tilde{g}(l)\geq0$ for $l\geq0$ and $\partial_l\tilde{\mathcal{L}}\tilde{g}(l)\leq0$ for $l\leq0$. Furthermore, by symmetry we must have $\partial_l\tilde{W}_T(t,0)=0$ and the derivative exists since $(t,0)\in\tilde{\mathcal{C}}_T$. Consider initial data $(t,l)$ with $l>0$ and any fixed $t\in[0,T)$. Let 
    \begin{equation*}
        \alpha(l):=\inf\{u\geq0 : L^l_u\leq 0\},
    \end{equation*}
    and $\tilde{\tau}^*:=\tilde{\tau}^*_T(t,l)$.
    Rewriting \eqref{eq:space_derivative_lamperti_w} by the tower property, we get
    \begin{align*}
        &\partial_l \tilde{W}_T(t,l)\\
        =&\mathbb{E}\left[\int_{0}^{\tilde\tau^*\wedge\alpha(l)}\exp\left\{\int_0^u\mu'(L^l_v)\,dv\right\}\partial_l(\tilde{\mathcal{L}}\tilde{g})(L^l_{u})\,du\right]\\
        &+\mathbb{E}\left[\ind{\tilde\tau^*>\alpha(l)}\exp\left\{\int_0^{\alpha(l)}\mu'(L^l_v)\,dv\right\}\mathbb{E}\left[\int_{\alpha(l)}^{\tilde\tau^*}\exp\left\{\int_{\alpha(l)}^u\mu'(L^l_v)\,dv\right\}\partial_l(\tilde{\mathcal{L}}\tilde{g})(L^l_{u})\,du\Bigg \vert\mathcal{F}^{L^l}_{\alpha(l)} \right]\right]
    \end{align*}
    and by the strong Markov property, we get 
    \begin{align*}
        \partial_l \tilde{W}_T(t,l)=&\mathbb{E}\left[\int_{0}^{\tilde\tau^*\wedge\alpha(l)}\partial_lL^l_u\partial_l(\tilde{\mathcal{L}}\tilde{g})(L^l_{u})\,du\right]+\mathbb{E}\left[\ind{\tilde\tau^*>\alpha(l)}\partial_lL^l_{\alpha(l)}\partial_l\tilde{W}_T(t+\alpha(l),L^l_{\alpha(l)})\right],
    \end{align*}
    at points of differentiability.
    But now $L^l_{\alpha(l)}=0$ on $\{\tilde\tau^*>\alpha(l)\}$, thus, by the initial observations that $\partial_l\tilde{W}_T(t,0)=0$ and $\partial_l(\tilde{\mathcal{L}}\tilde{g})(l)\geq0$, we have 
    \begin{equation*}
        \partial_l \tilde{W}_T(t,l)=\mathbb{E}\left[\int_{0}^{\tilde\tau^*\wedge\alpha(l)}\partial_lL^l_u\partial_l(\tilde{\mathcal{L}}\tilde{g})(L^l_{u})\,du\right]\geq0
    \end{equation*}
    for every $t\in[0,T)$ and almost all $l\geq0$ such that $(t,l)\in\tilde{\mathcal{S}}_T$.
    Since $l\mapsto\tilde{W}_T(t,l)$ is locally Lipschitz, and hence absolutely continuous, we get
    \begin{equation*}
        \tilde{W}_T(t,l_2)- \tilde{W}_T(t,l_1)=\int_{l_1}^{l_2}\partial_l \tilde{W}_T(t,l)\, dl\geq0
    \end{equation*}
    for all $0\leq l_1\leq l_2<\Psi(1)$.
    
    If $(t,l)\in \tilde{\mathcal{D}}_T$ for $l>0$, then for any $\varepsilon>0$ such that $(t,l+\varepsilon)\in\tilde{\mathcal{S}}_T$ it holds that 
    \begin{equation*}
        0\geq\tilde{W}_T(t,l+\varepsilon)\geq \tilde{W}_T(t,l)=0.
    \end{equation*}
    Hence $(t,l+\varepsilon)\in\tilde{\mathcal{D}}_T$. By a completely analogous argument on $l\leq0$ we can conclude that there exists a function $\tilde{b}_T:[0,T)\to [0,\Psi(1)]$ such that 
    \begin{equation*}
        \tilde{\mathcal{C}}_T=\{(t,l)\in\tilde{\mathcal{S}}_T : -\tilde{b}_T(t)< l< \tilde{b}_T(t)\}.
    \end{equation*}
    Since $\tilde{\mathcal{C}}_T$ is open, it holds that $\tilde{b}_T(t)$ is lower semicontinuous. Translating back to $\mathcal{S}_T$ space, we get the first claim where $b_T(t):=\Psi^{-1}(\tilde{b}_T(t))$. The second claim follows from $\mathcal{U}\cap\mathcal{S}_T\subseteq\mathcal{C}_T$ by Proposition \ref{prop:U_subset_C}.

    For infinite horizon we use the pointwise monotone convergence of Proposition \ref{prop:finite_to_infinite_1}. Since $T\mapsto V_T$ is decreasing, the $t$--slices of $\mathcal{D}_T$ decrease, i.e. $T\mapsto b_T$ increases. Furthermore, $b_T(t)\in[\gamma(t),1]$, so there exists a limit 
    \begin{equation}\label{eq:limit_b_T}
        b_T(t)\nearrow b_\infty(t):=\sup_{T>t}b_T(t)   
    \end{equation}
    where $b_\infty(t)\in[\gamma(t),1]$. Recall that $b_T$ is lower semicontinuous for any $T<\infty$. Since $b_\infty$ is the pointwise supremum of lower semicontinuous functions, it is lower semicontinuous itself. Recall the definition of the extended continuation set $\hat{\mathcal{C}}_T$. The monotonicity of $T\mapsto \hat{\mathcal{C}}_T$ combined with \eqref{eq:continuation_set_b} yields
    \begin{equation*}
        \hat{\mathcal{C}}_T\nearrow \{(t,x)\in\mathcal{S}:x\in(1-b_\infty(t),b_\infty(t))\}
    \end{equation*}
    while from Proposition \ref{prop:finite_to_infinite_1} we had $\hat{\mathcal{C}}_T\nearrow \mathcal{C}_\infty$, whence the desired conclusion.
\end{proof}
\begin{remark}
    We note that this argument is specialized to smooth symmetric costs $g$. For the classical hard-classification cost $g(x)=x\wedge(1-x)$, the existence of a boundary function instead follows directly from the concavity of $V_T$, since $g$ is affine on each side of $1/2$ and $V_T(t,\cdot)\leq g$ with equality at the endpoints..
\end{remark}

Proposition \ref{prop:two_boundaries} also showed $\partial_l \tilde{W}_T(t,l)\geq0$ for $l\geq0$ and $T<\infty$. We need a stronger result for the subsequent analysis.
\begin{lemma}\label{lem:pos_deriv}
    Let $T<\infty$. If $(t,l)\in\tilde{\mathcal{C}}_T$ with $l>0$, then $\partial_l \tilde{W}_T(t,l)>0$. Furthermore, if $\mathcal{I}=(t_1,t_2)$ with $0< t_1<t_2$ and $\underline{a}\in(0,\inf_{t\in\mathcal{I}}\tilde\gamma(t))$, then $\inf_{t\in\mathcal{I}}\partial_l\tilde{W}_T(t,\underline{a})\geq k>0$ uniformly in $T\in(T_1,\infty)$ for some $T_1>t_2$.
\end{lemma}
\begin{proof}
    Take $(t,l)\in\tilde{\mathcal{C}}_T$ such that $l>0$. Since $\tilde{\mathcal{C}}_T$ is open, we can find an open rectangle $R$ with $(t,l)\in R$ and $\overline{R}\subset\tilde{\mathcal{C}}_T$ such that for every $(t',l')\in R$ it holds that $l'>0$. Let $\rho:=\inf\{s\geq0:(t+s,L^l_s)\not\in R\}$, and note that $L^l_\rho\geq0$ and $\tilde\tau^*_T(t,l)\geq \rho$ almost surely. Thus using the tower property and strong Markov property as above, we get
    \begin{equation}
        \begin{aligned}\label{eq:positive_derivative}
            \partial_l\tilde{W}_T(t,l)&=\mathbb{E}\left[\int_0^\rho \partial_lL^l_u\partial_l(\tilde{\mathcal{L}}\tilde{g})(L^l_u)\,du\right]+\mathbb{E}\left[\partial_l L^l_\rho\partial_l\tilde{W}_T(t+\rho,L^l_\rho)\right]\\
            &\geq\mathbb{E}\left[\int_0^\rho \partial_lL^l_u\partial_l(\tilde{\mathcal{L}}\tilde{g})(L^l_u)\,du\right]>0,
        \end{aligned}
    \end{equation}
    where we used $\partial_l\tilde{W}_T(t+\rho,L^l_\rho)\geq0$ by Proposition \ref{prop:two_boundaries}, $\partial_l(\tilde{\mathcal{L}}\tilde{g})(l)>0$ on $l>0$ by Assumption \ref{asmp:dxLg}, and $\rho>0$ almost surely.

    For the second part take any $T\geq T_1$ and note that we can take $\varepsilon,\delta>0$ with $\delta<T_1-t_2$ and define the open rectangle $R:=\mathcal{I}_\delta\times\mathcal{J}_\varepsilon$ where $\mathcal{I}_\delta:=(t_1-\delta,t_2+\delta)$ and $\mathcal{J}_\varepsilon:=(\underline{a}-\varepsilon,\underline{a}+\varepsilon)$ such that, for any $(t,l)\in \overline{R}$, it holds that $0<l<\inf_{t\in\mathcal{I}} \tilde\gamma(t)$ since $\tilde\gamma\in C^1([0,\infty))$. In particular, $\varepsilon$ and $\delta$ can be chosen uniformly in $T$ for $T>t_2$ since $\tilde\gamma$ is independent of $T$. The functions $l\mapsto \mu'(l)$ and $l\mapsto \partial_l(\tilde{\mathcal{L}}\tilde{g})(l)$ are continuous with $\partial_l(\tilde{\mathcal{L}}\tilde{g})(l)>0$ on $\overline{\mathcal{J}_\varepsilon}$. Hence, there exist constants $k_1,k_2>0$ such that 
    \begin{equation}\label{eq:k1_and_k2}
        \partial_l(\tilde{\mathcal{L}}\tilde{g})(l)\geq k_1, \qquad |\mu'(l)|\leq k_2,
    \end{equation}
    on $\overline{\mathcal{J}_\varepsilon}$. Define the stopping time $\rho(t):=\inf\{s\geq0:(t+s,L^{\underline{a}}_s)\not\in R\}\leq \tilde\tau^*_T(t,\underline{a})$, and note in particular that $\mathbb{E}[\rho(t)]\geq k_3$ for all $t\in\mathcal{I}$ for some $k_3>0$. Thus, by arguing as in \eqref{eq:positive_derivative}, we get
    \begin{align*}
        \partial_l\tilde{W}_T(t,\underline{a})\geq\mathbb{E}\left[\int_0^{\rho(t)} \partial_lL^{\underline{a}}_u\partial_l(\tilde{\mathcal{L}}\tilde{g})(L^{\underline{a}}_u)\,du\right]\geq k_1e^{-k_2 t_2}\mathbb{E}[\rho(t)]\geq k_1k_3e^{-k_2 t_2}=:k,
    \end{align*}
    and $k>0$. Since $T>t_2$ was arbitrary, we arrive at the result. 
\end{proof}
\begin{remark}
For each fixed $T$, a positive lower bound on $\partial_l\tilde W_T(t,\underline a)$ over a compact time interval would follow directly from pointwise positivity and continuity of $(t,l)\mapsto\partial_l\tilde W_T(t,l)$. The argument above is needed
to obtain a lower bound that is uniform in the horizon $T$.
\end{remark}

In Proposition \ref{prop:two_boundaries} we obtained $\gamma(t)\leq b_T(t)\leq 1$. To make the first inequality strict we will use the smooth-fit principle. To this end, we need sufficient regularity of the optimal stopping boundary. In the case of $t\mapsto c(t)$ increasing it is easy to show that $t\mapsto b_T(t)$ is decreasing, and so starting the process $X^x$ at the boundary will result in the process hitting the interior of the stopping region immediately. However, since $t\mapsto c(t)$ is not assumed to be increasing, $t\mapsto b_T(t)$ may be increasing. We need to show that $t\mapsto b_T(t)$ does not increase too fast. To this end, we will show that $b_T$ is locally Lipschitz continuous by a method similar to \cite{de2019lipschitz}.

On the other hand, we can easily get $b_T(t)<1$ on finite horizon by arguing as in \eqref{eq:hom_bound_T_value}--\eqref{eq:hom_bound_T}. However, this argument does not work for infinite horizon if Assumption \ref{asmp:bounded:c} is not satisfied. In addition, $b_T$ is lower semicontinuous, and thus we may get arbitrarily close to $1$.
To find a uniform upper bound away from $1$ on any compact interval we will utilize a bound obtained from a particularly nice infinite-horizon problem in Lemma \ref{lem:b_uniform_upper_bound}.

\begin{lemma}\label{lem:b_uniform_upper_bound}
    For any $\mathcal{I}=(t_1,t_2)$ with $0\leq t_1<t_2$, there exists a constant $a<1$ such that $\sup _{t\in \mathcal{I}}b_T(t)\leq a$ uniformly in $T\in(t_2,\infty]$.
\end{lemma}
\begin{proof}
    If Assumption \ref{asmp:bounded:c} holds, then $b_T(t)\leq B^*<1$ where $B^*$ is given in Theorem \ref{VerificationInfConstant} using the constant time cost $\underline{c}=\inf_{t\geq0}c(t)>0$. If Assumption \ref{asmp:bounded:g'} holds, we introduce a new running cost defined by $\underline{c}(t)=e^{-\lambda t}c(t)$ for $\lambda>K$, where $K$ is the constant in Assumption \ref{asmp:base_penalty:c}. Then $\underline{c}\in C^1([0,\infty)), \underline{c}(t)>0$, and
    \begin{equation}\label{eq:underline_c}
        \underline{c}'(t)=\underline{c}(t)\left(\frac{c'(t)}{c(t)}-\lambda\right),
    \end{equation}
    hence $|\underline{c}'(t)|/\underline{c}(t)\leq \lambda+K<\infty$, and so Assumption \ref{asmp:base_penalty:c} holds for $\underline{c}$. From \eqref{eq:underline_c} we get 
    \begin{equation*}
        \underline{c}'(t)\leq \underline{c}(t)(K-\lambda)<0,
    \end{equation*}
    so $\underline{c}$ is decreasing. Hence, we can consider the infinite-horizon optimal stopping problem with running cost $\underline{c}$ and value function $\underline{V}$.
    
    Take $t\in(0,\infty)$ and $\varepsilon>0$ such that $t\geq\varepsilon$. Then
    \begin{equation}\label{eq:monotone_W_lower}
        \underline{W}(t-\varepsilon,x)-\underline{W}(t,x)\geq \mathbb{E}\left[\int_0^{\underline{\tau}^*(t-\varepsilon,x)}\big(\underline{c}(t-\varepsilon+u)-\underline{c}(t+u)\big)\,du\right]\geq0,
    \end{equation}
    where $\underline{\tau}^*$ denotes the optimal stopping time for $\underline{V}$.
    Thus, if $(t,x)$ is in the stopping set for $\underline{W}$, denoted by $\underline{\mathcal{D}}$, then $0\geq\underline{W}(t-\varepsilon,x)\geq \underline{W}(t,x)=0$ and so $(t-\varepsilon,x)\in\underline{\mathcal{D}}$. 
    Since $\underline{V}$ admits an optimal stopping boundary, say $\underline{b}$, it holds that $t\mapsto\underline{b}(t)$ is increasing. Furthermore, since $c\geq \underline{c}$ for all $t\geq0$ it holds that $\underline{V}\leq V_T$, and hence $b_T\leq \underline{b}$ for any $T\leq\infty$.
    
    Assume for contradiction that there exists a point $t'\geq0$ such that $\underline{b}(t')=1$. Since $t\mapsto\underline{b}(t)$ is increasing, we have $\underline{b}(t'+s)=1$ for $s\geq0$. Thus, the first hitting time to the stopping set for the $\underline{V}$ problem with initial data $(t',x)$ is almost surely infinite, for any $x\in(0,1)$, since $\{0,1\}$ are inaccessible in finite time. By the choice of $\underline{c}$, we also have that stopping at infinity is admissible. Indeed, by applying \eqref{eq:expGrowth} to $c$, we have 
    \begin{equation*}
        \underline{c}(t'+u)\leq e^{-\lambda t'}c(t')e^{(K-\lambda)u},
    \end{equation*}
    so $\int_0^\infty \underline{c}(t'+u)\,du<\infty$.
    Hence
    \begin{equation*}
         \underline{V}(t',x)=\mathbb{E}\left[\int_0^{\infty}\underline{c}(t'+u)\,du+g(X^x_\infty)\right]=\int_0^\infty \underline{c}(t'+u)\,du>0,
    \end{equation*}
    using $g(X^x_\infty)=0$ almost surely.
    However, $0\leq \underline{V}(t',x)\leq g(x)$ and $g$ is continuous with $g(1)=0$, so taking the limit as $x\to 1$ we get $\lim_{x\to1}\underline{V}(t',x) = 0$ and thus, we reach a contradiction.

    Since $t\mapsto \underline{b}(t)$ is increasing and pointwise $\underline{b}(t)<1$, there exists a constant $a<1$ such that for any bounded interval $\mathcal{I}:=(t_1,t_2)$ it holds that $\sup_{t\in\mathcal{I}}\underline{b}(t)\leq a<1$. Since $b_T\leq \underline{b}$, we have the desired result when $T\in(t_2,\infty]$.
\end{proof}

We are now ready to show that the optimal stopping boundary is locally Lipschitz continuous. We will do so through the Lamperti-transformed problem. Take any $0< t_1<t_2<\infty$ and define $\mathcal{I}:=(t_1,t_2)$.
Since $t\mapsto\tilde\gamma(t)$ is continuous and $\tilde\gamma(t)>0$ for all $t\geq0$, there exists a $\underline{a}\in(0,\inf_{t\in\mathcal{I}}\tilde\gamma(t))$ as in Lemma \ref{lem:pos_deriv}. Furthermore, for such $\underline{a}$ it holds that $\sup_{t\in\mathcal{I}}\tilde{W}_T(t,\underline{a})<0$ for any $T\leq\infty$. Similarly, take $\overline{a}\in(\sup_{t\in\mathcal{I}}\tilde b_T(t),\Psi(1))$ which is possible by Lemma \ref{lem:b_uniform_upper_bound}. Defining $\mathcal{J}:=(\underline{a},\overline{a})$ it holds that
\begin{equation*}
    0<\underline{a}<\inf_{t\in \mathcal{I}}\tilde\gamma(t)\leq \tilde b_T(t)\leq \sup_{t\in\mathcal{I}}\tilde b_T(t)<\overline{a}<\Psi(1),
\end{equation*}
uniformly in $T\in(t_2,\infty]$.

Take any $T<\infty$. On $\mathcal{I}\times\mathcal{J}$ we have $\partial_l\tilde{W}_T\geq 0$, $\tilde{W}_T$ is continuous, $\tilde{W}_T(t,\overline{a})=0$, and $\sup_{t\in\mathcal{I}}\tilde{W}_T(t,\underline{a})<0$, so the intermediate value theorem yields existence of a solution to $\tilde{W}_T(t,l)=-\delta$ in $\mathcal{I}\times\mathcal{J}$ for small $\delta>0$, say $\delta<-\sup_{t\in\mathcal{I}}\tilde{W}_T(t,\underline{a})/2$. Since $\tilde{b}_T(t)$ is the $0$ level set of $\tilde{W}_T(t,l)$, the $-\delta$ level set is in the interior of $(\mathcal{I}\times\mathcal{J})\cap\tilde{\mathcal{C}}_T$ and so $\partial_l\tilde{W}_T>0$ by Lemma \ref{lem:pos_deriv}. Thus, the $-\delta$ level set is given by a function $t\mapsto a^\delta_T(t)$. Furthermore, by Proposition \ref{prop:PDE}, $\tilde{W}_T$ is $C^1$ here, and hence $t\mapsto a^\delta_T(t)$ is $C^1$ by the implicit function theorem. In fact, it holds that
\begin{equation*}
    \partial_t a_T^\delta(t)=-\frac{\partial_t \tilde{W}_T(t,a^\delta_T(t))}{\partial_l \tilde{W}_T(t,a^\delta_T(t))},
\end{equation*}
for $t\in\mathcal{I}$.

\begin{lemma}\label{lem:derivative_bound}
    Let $T<\infty$. For any $\varepsilon>0$ such that $\mathcal{I}_\varepsilon:=[t_1+\varepsilon,t_2-\varepsilon]$ has nonempty interior, there exists a $K_\varepsilon>0$ uniformly for all sufficiently large $T$ such that
    \begin{equation*}
        \lvert \partial_t a_T^\delta(t)\rvert\leq K_\varepsilon
    \end{equation*}
    for any $t\in\mathcal{I}_\varepsilon$ and $\delta>0$ sufficiently small.
\end{lemma}
\begin{proof}
Fix any $\varepsilon>0$ such that $\mathcal{I}_\varepsilon$ has nonempty interior. Consider $(t,l)\in \mathcal{I}\times\mathcal{J}$
on which we have $m_1,m_2,m_3>0$ such that
\begin{align}
    \partial_l(\tilde{\mathcal{L}}\tilde{g})(l)&\geq m_1,\label{eq:m1}\\
    |\mu'(l)|&\leq m_2\label{eq:m2},
\end{align}
and 
\begin{equation}\label{eq:m3}
    \inf_{t\in\mathcal{I}}\partial_l\tilde{W}_T(t,\underline{a})>m_3,
\end{equation}
uniformly in $T$, where \eqref{eq:m3} is given by $k$ in the second claim of Lemma \ref{lem:pos_deriv} while \eqref{eq:m1} and \eqref{eq:m2} can be found on $\mathcal{I}\times\mathcal{J}$ by an argument similar to \eqref{eq:k1_and_k2}.

Fix any $t\in \mathcal{I}_\varepsilon$, let $l=a^\delta_T(t)$, $\tilde\tau^*:=\tilde\tau^*_T(t,l)$
and define the stopping time
\begin{equation*}
    \tau_r(t,l)=\inf\{u\geq0 : L^l_u\not\in \mathcal{J}\}\wedge(t_2-t)
\end{equation*}
for $t\in\mathcal{I}_\varepsilon$.
From Corollary \ref{cor:W_T_lips}, we have 
\begin{align*}
    \lvert\partial_t\tilde{W}_T(t,l)\rvert&\leq \mathbb{E}\left[\lvert c(t+\tilde\tau^*)-c(t)\rvert +\ind{\tilde\tau^*=T-t}K''\right]\\
    &\leq\mathbb{E}\left[\int_0^{\tilde\tau^*\wedge\tau_r}\lvert c'(t+u)\rvert \,du+\ind{\tilde\tau^*>\tau_r}\left(\lvert c(t+\tilde\tau^*)-c(t+\tau_r)\rvert+K''\right)\right]\\
    &\leq K\|\mathcal{L}g\|_\infty\mathbb{E}[\tilde\tau^*\wedge\tau_r]+(2\|\mathcal{L}g\|_\infty+K'')\mathbb{E}\left[\ind{\tilde\tau^*>\tau_r}\right]\\
    &\leq K_1(\mathbb{E}[\tilde\tau^*\wedge\tau_r]+\mathbb{E}\left[\ind{\tilde\tau^*>\tau_r}\right])
\end{align*}
for $K_1=(K+4)\|\mathcal{L}g\|_\infty$ by Assumption \ref{asmp:base_penalty:c}, $\|c\|_\infty\leq\|\mathcal{L}g\|_\infty$, and since $K''\leq 2\|\mathcal{L}g\|_\infty$.
Recall $\partial_l\tilde{W}_T(t,l')\geq 0$ and $\partial_l(\tilde{\mathcal{L}}\tilde{g})(l')\geq0$ for $l'\geq0$. By the strong Markov property on \eqref{eq:space_derivative_lamperti_w}, we have 
\begin{align*}
    \partial_l\tilde{W}_T(t,l)=&\mathbb{E}\left[\int_0^{\tilde\tau^*\wedge\tau_r}\partial_lL^l_u\partial_l(\tilde{\mathcal{L}}\tilde{g})(L^{l}_u)\,du+\ind{\tilde\tau^*>\tau_r}\partial_lL^l_{\tau_r}\partial_l\tilde{W}_T(t+\tau_r,L^{l}_{\tau_r})\right]\\
    \geq&\mathbb{E}\left[\int_0^{\tilde\tau^*\wedge\tau_r}\partial_lL^l_u\partial_l(\tilde{\mathcal{L}}\tilde{g})(L^{l}_u)\,du+\ind{\tilde\tau^*>\tau_r,\tau_r<t_2-t}\partial_lL^l_{\tau_r}\partial_l\tilde{W}_T(t+\tau_r,\underline{a})\right],
\end{align*}
since the process will hit the $\underline{a}$--level before it hits the $\overline{a}$--level given that $\tau_r$ is less than $t_2-t$ and $\tilde\tau^*$.
Thus, by the previous considerations and recalling \eqref{eq:partiall_explicit}, we have
\begin{equation*}
    \partial_l \tilde{W}_T(t,l)\geq e^{-m_2(t_2-t)}\mathbb{E}\left[m_1(\tilde\tau^*\wedge\tau_r)+m_3\ind{\tilde\tau^*>\tau_r,\tau_r<t_2-t}\right]\geq K_2\mathbb{E}\left[\tilde\tau^*\wedge\tau_r+\ind{\tilde\tau^*>\tau_r,\tau_r<t_2-t}\right]>0
\end{equation*}
for $K_2=e^{-m_2(t_2-t_1)}(m_1\wedge m_3)$.
Thus 
\begin{align*}
    \lvert \partial_ta_T^\delta(t)\rvert&\leq\frac{K_1}{K_2}\frac{\mathbb{E}[\tilde\tau^*\wedge\tau_r+\ind{\tilde\tau^*>\tau_r}]}{\mathbb{E}[\tilde\tau^*\wedge\tau_r+\ind{\tilde\tau^*>\tau_r,\tau_r<t_2-t}]}\leq\frac{K_1}{K_2}\left(1+\frac{\mathbb{P}(\tilde\tau^*>\tau_r)}{\mathbb{E}[\tilde\tau^*\wedge\tau_r+\ind{\tilde\tau^*>\tau_r,\tau_r<t_2-t}]}\right).
\end{align*}
Splitting the numerator into the space and time component of $\tau_r$, we note
\begin{equation*}
    \frac{\mathbb{P}(\tilde\tau^*>\tau_r,\tau_r<t_2-t)}{\mathbb{E}[\tilde\tau^*\wedge\tau_r+\ind{\tilde\tau^*>\tau_r,\tau_r<t_2-t}]}\leq1
\end{equation*}
and 
\begin{equation*}
    \frac{\mathbb{P}(\tilde\tau^*>\tau_r,\tau_r=t_2-t)}{\mathbb{E}[\tilde\tau^*\wedge\tau_r+\ind{\tilde\tau^*>\tau_r,\tau_r<t_2-t}]}\leq\frac{\mathbb{P}(\tilde\tau^*>\tau_r,\tau_r=t_2-t)}{\mathbb{E}[\ind{\tilde\tau^*>\tau_r,\tau_r=t_2-t}(t_2-t)]} \leq \varepsilon^{-1}
\end{equation*}
since $t\in\mathcal{I}_\varepsilon$.
Hence, we have that $\partial_ta_T^\delta(t)$ is bounded for $t\in\mathcal{I}_\varepsilon$ by some $K_\varepsilon$ that is uniform for $(t,l)\in\mathcal{I}_\varepsilon\times \mathcal{J}$ and eventually uniformly in $\delta$.
\end{proof}
\begin{proposition}\label{prop:bound_lip}
    The optimal stopping boundary $t\mapsto b_T(t)$ is $\llip{[0,T)}$ for $T\leq \infty$.
\end{proposition}
\begin{proof}
    First, let $T<\infty$. As noted, we have that $a_T^\delta(t)<\tilde{b}_T(t)$ for all $t\in\mathcal{I}$. Furthermore, since $l\mapsto \tilde{W}_T(t,l)$ is increasing, $(a_T^\delta)_{\delta>0}$ is increasing on $t\in\mathcal{I}$ as $\delta\searrow0$. Hence, by monotone convergence, a limit $a_T^0$ exists and satisfies $a_T^0(t)\leq \tilde{b}_T(t)$. Since $a_T^\delta$ are the level sets and $\tilde{W}_T$ is continuous, we get $\tilde{W}_T(t,a_T^0(t))=0$, and hence $a_T^0(t)\geq \tilde{b}_T(t)$. Thus, we have $a_T^\delta(t)\to \tilde{b}_T(t)$ for $t\in\mathcal{I}$. To conclude local Lipschitz continuity we use the pointwise convergence and Lemma \ref{lem:derivative_bound}. That is,
    for $t,s\in\mathcal{I}_\varepsilon$ we can write
    \begin{align*}
        \lvert \tilde{b}_T(t)-\tilde{b}_T(s)\rvert\leq& \lvert \tilde{b}_T(t)-a_T^\delta(t)\rvert+\lvert a_T^\delta(t)-a_T^\delta(s)\rvert+\lvert a_T^\delta(s)-\tilde{b}_T(s)\rvert \\
        \leq& \lvert \tilde{b}_T(t)-a_T^\delta(t)\rvert+K_\varepsilon\lvert t-s\rvert+\lvert a_T^\delta(s)-\tilde{b}_T(s)\rvert\\
        &\to K_\varepsilon\lvert t-s\rvert
    \end{align*}
    as $\delta\to 0$, and hence $\tilde{b}_T\in\llip{(0,T)}$.
    For infinite horizon, note that
    \begin{equation*}
        \lvert \tilde{b}_\infty(t)-\tilde{b}_\infty(s)\rvert\leq 
        \lvert \tilde{b}_\infty(t)-\tilde{b}_T(t)\rvert
        +K_\varepsilon\lvert t-s\rvert
        +\lvert \tilde{b}_\infty(s)-\tilde{b}_T(s)\rvert
        \to K_\varepsilon\lvert t-s\rvert,
    \end{equation*}
    as $T\to \infty$ by \eqref{eq:limit_b_T}.
    Finally, we can also conclude that $b_T$ is locally Lipschitz on $[0,T)$ since $\Psi^{-1}$ is Lipschitz on $\mathcal{J}$ and the extension to $0$ is due to \eqref{eq:extend_0}.
\end{proof}

\subsection{Regularity and the smooth fit}

Next, we use the Lipschitz boundary to show that the stopping set is probabilistically regular, which in turn gives the smooth-fit property of the value function in Theorem \ref{thm:main_regular_smooth} for finite horizon due to \cite{de2020global}. 
\begin{lemma}\label{lem:regular_boundary}
    Let $((t_n,x_n))_{n\in\mathbb{N}}\subset\mathcal{C}_T$ with $(t_n,x_n)\to(t,b_T(t))$. Then $\tau^*_T(t_n,x_n)\to 0$ almost surely for any $T\leq \infty$.
\end{lemma}
\begin{proof}
    Let $t\in(0,T)$, take $\varepsilon>0$ and $t_1,t_2\in(0,T)$ such that $t$ is in the interior of $\mathcal{I}_\varepsilon=[t_1+\varepsilon,t_2-\varepsilon]$ as above, and let $K_\varepsilon$ be the upper Lipschitz constant for $\tilde{b}_T$ on $\mathcal{I}_\varepsilon$. Set $l=\tilde{b}_T(t)$ and take a compact $\mathcal{J}:=[l-r,l+r]\subset\Psi((0,1))$ for some $r>0$. Define the stopping time 
    \begin{equation*}
        \tau^\mathcal{J}:=\inf\{s\geq0: L^l_s\not\in \mathcal{J}\},
    \end{equation*}
    for which we note $\tau^\mathcal{J}>0$ almost surely since $L^l$ has continuous paths. Now consider the stopped exponential martingale
    \begin{equation*}
        Z^\mathcal{J}_s:=\mathcal{E}\left(\int_0^{s\wedge\tau^\mathcal{J}} (K_\varepsilon-\mu(L^l_u))\, dB_u\right).
    \end{equation*}
    The process $Z^\mathcal{J}$ is indeed a true martingale on, say, $s\in[0,\delta]$ for $\delta:=t_2-\varepsilon-t>0$, by Novikov's condition due to the localization on $\mathcal{J}$. By Girsanov's theorem, we can define the new measure $\mathbb{Q}^\mathcal{J}$ on $\mathcal{F}_\delta$ by
    \begin{equation*}
        \frac{d\mathbb{Q}^\mathcal{J}}{d\mathbb{P}}\Bigg\vert_{\mathcal{F}_\delta}=Z^\mathcal{J}_\delta,
    \end{equation*}
    and, up to $\tau^\mathcal{J}$, it holds that 
    \begin{equation*}
        L^l_s-K_\varepsilon s=l+B^\mathcal{J}_s,
    \end{equation*}
    where $B^\mathcal{J}$ is a Brownian motion under $\mathbb{Q}^\mathcal{J}$.

    Consider the first hitting time to the upper boundary
    \begin{equation*}
        \sigma^l := \inf\{s>0 : L^l_s> \tilde{b}_T(t+s)\}.
    \end{equation*}
    Since $\tilde{b}_T$ is Lipschitz, we have 
    \begin{equation}\label{eq:exit_time_ineq}
        \overline{\sigma}_l^\mathcal{J}:=\inf\{0<s<\delta\wedge\tau^\mathcal{J} : L^l_s-K_\varepsilon s> l\}\geq \sigma^l
    \end{equation}
    where $\inf$ of the empty set is $\infty$ by convention. Thus, by the previous considerations, we have 
    \begin{equation*}
        \overline{\sigma}_l^\mathcal{J}=\inf\{0<s<\delta\wedge\tau^\mathcal{J} : B^\mathcal{J}_s> 0\}
    \end{equation*}
    and, specifically, $\overline{\sigma}_l^\mathcal{J}=0$ holds $\mathbb{Q}^\mathcal{J}$--almost surely by standard results, see e.g. \cite[Problem 7.18]{karatzas2012brownian}.
    Since $\mathbb{P}$ is an equivalent measure, we have
    \begin{equation*}
        \mathbb{P}(\sigma^l=0)=1
    \end{equation*}
    by \eqref{eq:exit_time_ineq}.
    Thus $l=\tilde{b}_T(t)$ is probabilistically regular for $(\tilde{\mathcal{D}}_T)^\circ$, and thus also $\tilde{\mathcal{D}}_T$, with respect to $L$. Since $t$ was arbitrary, for any $t\in(0,T)$, the boundary point $(t,l)=(t,\tilde{b}_T(t))$ is probabilistically regular for $\tilde{\mathcal{D}}_T$ and $(\tilde{\mathcal{D}}_T)^\circ$. Furthermore, 
    since $l\mapsto L^l_s$ is continuous, it also holds that $(t,\tilde{b}_T(t))$ is Green regular for $\tilde{\mathcal{D}}_T$ and $(\tilde{\mathcal{D}}_T)^\circ$ by \cite[Corollary 5]{de2020global}, that is,
    \begin{equation}\label{eq:green_reg}
        \lim_{\tilde{\mathcal{C}}_T\ni (t_n,l_n)\to (t,\tilde{b}_T(t))}\mathbb{P}(\tau(t_n,l_n)\geq \delta)= 0
    \end{equation}
    for every $\delta>0$, where 
    \begin{equation*}
        \tau(t,l) = \inf\{s\geq0 : (t+s,L^l_s)\in A\}
    \end{equation*}
    for $A=\tilde{\mathcal{D}}_T$ and $A=(\tilde{\mathcal{D}}_T)^\circ$. In fact, \cite[Corollary 6]{de2020global} upgrades the convergence in probability to almost sure convergence. The same conclusions can then be translated directly to the $\mathcal{S}_T$ coordinates and extended to $t=0$ by \eqref{eq:extend_0}.

    Specifically, for $((t_n,x_n))_{n\in\mathbb{N}}\subset \mathcal{C}_T$ with $(t_n,x_n)\to(t, b_T(t))$ it holds that $\tau^*_T(t_n,x_n)\to0$ almost surely.
\end{proof}
We can now prove the temporal smooth fit. We note that since $V_T$ is $C^{1,2}$ on $\mathcal{C}_T$, the probabilistic derivative representations derived above hold at every point of $\mathcal{C}_T$, rather than merely almost everywhere.
\begin{proposition}\label{prop:temp_smooth}
    For any $T\leq\infty$, let $((t_n,x_n))_{n\in\mathbb{N}}$ be a sequence as in Lemma \ref{lem:regular_boundary}. Then 
    \begin{equation}\label{eq:time_smooth_fit}
        \lim_{n\to\infty}\partial_tV_T(t_n,x_n)=0.
    \end{equation}
\end{proposition}
\begin{proof}
    It holds that $\mathbb{E}[c(t_n+\tau^*_T(t_n,x_n))]-c(t_n)\to 0$ since $c$ is bounded and $\tau^*_T(t_n,x_n)\to0$ almost surely. Furthermore, if $T<\infty$, then, by \eqref{eq:green_reg}, we have $\mathbb{P}(\tau^*_T(t_n,x_n)=T-t_n)\to0$. Hence \eqref{eq:time_deriv_bound} and \eqref{eq:time_deriv_inf} imply the temporal smooth fit.
\end{proof}
The spatial smooth fit requires more work in the infinite-horizon case. Indeed, we will need the finite-horizon smooth fit for the approach below.
\begin{proposition}\label{prop:spac_smooth}
    For any $T<\infty$, let $((t_n,x_n))_{n\in\mathbb{N}}$ be a sequence as in Lemma \ref{lem:regular_boundary}. Then 
    \begin{equation*}
        \lim_{n\to\infty}\partial_xV_T(t_n,x_n)=g'(b_T(t)).
    \end{equation*}
\end{proposition}
\begin{proof}
    By an application of dominated convergence theorem and Scheffé's lemma, since $g'$ is locally bounded by $K_\text{space}$, and $\partial_xX^x$ is a true martingale on $[0,T]$, as in Lemma \ref{lem:concave_V} and Proposition \ref{prop:space_lip}, we have 
    \begin{equation*}
        \lim_{n\to\infty}\partial_xV_T(t_n,x_n)= \lim_{n\to\infty}\mathbb{E}\left[g'(X^{x_n}_{\tau^*_T(t_n,x_n)})\partial_xX^{x_n}_{\tau^*_T(t_n,x_n)}\right]=g'(b_T(t)),
    \end{equation*}
    using $\tau^*_T(t_n,x_n)\to0$ almost surely.
\end{proof}

In the proof of the spatial smooth fit we made use of the probabilistic representation of the spatial derivative on finite horizon. This derivative is not directly available in the infinite-horizon case. We thus must take a different approach. We first prove the following convergence result.
\begin{proposition}\label{prop:uniform_conv_dx}
    For every $\mathcal{I}=[t_1,t_2]\subset[0,\infty)$ and $\mathcal{J}=[\underline{a},\overline{a}]\subset(0,1)$ it holds that $x\mapsto V_\infty(t,x)$ is differentiable, and
    \begin{equation*}
        \partial_xV_T\to \partial_xV_\infty
    \end{equation*}
    uniformly on $\mathcal{I}\times\mathcal{J}$ for $T\to\infty$.
\end{proposition}
\begin{proof}
    Let $\delta>0$ such that $\mathcal{J}_\delta=[\underline{a}-\delta,\overline{a}+\delta]\subset(0,1)$.
    We will first show that $\partial_{xx}V_T$ is uniformly bounded on $\mathcal{I}\times\mathcal{J}_\delta$ to show that $x\mapsto \partial_x V_T(t,x)$ is Lipschitz continuous. 

    Recall from Corollary \ref{cor:main_discont_dxx_pt1} that we have an expression for $\partial_{xx} V_T(t,x)$, which can be bounded by
    \begin{equation*}
        |\partial_{xx}V_T(t,x)|\leq |g''(x)|+2\frac{c(t)+|\partial_tV_T(t,x)|}{\sigma^2(x)}.
    \end{equation*}
    It holds that $\sup_{x\in\mathcal{J}_\delta}|g''(x)|<\infty$, $\inf_{x\in\mathcal{J}_\delta}\sigma^2(x)>0$, and $\sup_{t\in\mathcal{I}}c(t)<\infty$. Furthermore, $|\partial_t V_T(t,x)|\leq K_{\text{time}}$ from Proposition \ref{prop:time_lip} which is independent of $T$. Thus, there exists a uniform constant $K_\delta$ such that
    \begin{equation*}
        |\partial_{xx}V_T(t,x)|\leq K_\delta,
    \end{equation*}
    whenever it exists on $\mathcal{I}\times\mathcal{J}_\delta$.
    It then follows that 
    \begin{equation}\label{eq:lip_partialx}
        |\partial_x V_T(t,x)-\partial_x V_T(t,y)|\leq K_\delta|x-y|,
    \end{equation}
    for all $x,y\in\mathcal{J}_\delta$. Namely, we may, without loss of generality, assume $x<b_T(t)<y$. Then, by smooth fit, we have
    \begin{align*}
        |\partial_x V_T(t,x)-\partial_x V_T(t,y)|=&|\partial_x V_T(t,x)-\partial_x V_T(t,b_T(t)-)+\partial_x V_T(t,b_T(t)+)-\partial_x V_T(t,y)|\\
        =&\left|\int_x^{b_T(t)}\partial_{xx}V_T(t,z)\,dz+\int_{b_T(t)}^y\partial_{xx}V_T(t,z)\,dz\right|\leq K_\delta|x-y|.
    \end{align*}

    Next, we show that $\partial_x V_T$ is uniformly Cauchy in $T$ on $\mathcal{I}\times\mathcal{J}$. Take $h\in(0,\delta)$ and $x\in\mathcal{J}$ such that $x+h\in\mathcal{J}_\delta$. By the fundamental theorem of calculus
    \begin{equation*}
        V_T(t,x+h)-V_T(t,x)=\int_0^h\partial_xV_T(t,x+r)dr,
    \end{equation*}
    and so, by the triangle inequality,
    \begin{equation}\label{eq:dx_mixed}
        \left|\frac{V_T(t,x+h)-V_T(t,x)}{h}-\partial_xV_T(t,x)\right|\leq \frac{1}{h}\int_0^h|\partial_xV_T(t,x+r)-\partial_xV_T(t,x)|\,dr\leq \frac{hK_\delta}{2}
    \end{equation}
    by \eqref{eq:lip_partialx}. For $T,S>t_2$, we then have
    \begin{align}
        &|\partial_xV_T(t,x)-\partial_xV_S(t,x)|\label{eq:dx_cauchy_first}\\
        \leq& \left|\partial_xV_T(t,x)-\frac{V_T(t,x+h)-V_T(t,x)}{h}\right|+\left|\frac{V_T(t,x+h)-V_S(t,x+h)}{h}\right|
        \\&+\left|\frac{V_S(t,x)-V_T(t,x)}{h}\right|+\left|\frac{V_S(t,x+h)-V_S(t,x)}{h}-\partial_xV_S(t,x)\right|\\
        \leq& hK_\delta+\frac{2}{h}\|V_T-V_S\|_{L^\infty(\mathcal{I}\times\mathcal{J}_\delta)}\label{eq:dx_cauchy_last},
    \end{align}
    by \eqref{eq:dx_mixed}. Since $V_T$ converges uniformly on $\mathcal{I}\times\mathcal{J}_\delta$ as $T\to\infty$, it specifically holds that $V_T$ is uniformly Cauchy, hence, for any $\varepsilon>0$, we can choose $T,S$ large enough such that 
    $\|V_T-V_S\|_{L^\infty(\mathcal{I}\times\mathcal{J}_\delta)}<h\varepsilon/4$. Similarly, we can take $h<\varepsilon/(2K_\delta)$. Thus, taking supremum over $\mathcal{I}\times\mathcal{J}$ in \eqref{eq:dx_cauchy_first}-\eqref{eq:dx_cauchy_last}, we get
    \begin{equation*}
        \|\partial_xV_T-\partial_xV_S\|_{L^\infty(\mathcal{I}\times\mathcal{J})}<\varepsilon,
    \end{equation*}
    i.e. $\partial_x V_T$ is uniformly Cauchy in $T$ on $\mathcal{I}\times\mathcal{J}$ and so it converges uniformly to some function $(t,x)\mapsto U(t,x)$. By the fundamental theorem of calculus, we have
    \begin{equation*}
        V_\infty(t,x)-V_\infty(t,y)=\lim_{T\to\infty}\int_y^x\partial_x V_T(t,z)\,dz=\int_y^x U(t,z)\,dz,
    \end{equation*}
    so we identify $U(t,x)=\partial_xV_\infty(t,x)$.
\end{proof}

Using this, we can now proceed to the proof of the main result, Theorem \ref{thm:main_regular_smooth}.
\begin{proof}[Proof of Theorem \ref{thm:main_regular_smooth}]
    Since $V_T$ is locally Lipschitz and equals $g$ on $\mathcal{D}_T$, once
    $\partial_tV_T$ and $\partial_xV_T$ on $\mathcal{C}_T$ extend continuously
    to the boundary with limits $0$ and $g'$, respectively, a standard pasting argument yields that $V_T\in C^1(\mathcal{S}_T)$.
    By this observation it suffices to restrict to sequences approaching the boundary from the continuation region. By Propositions \ref{prop:PDE}, \ref{prop:temp_smooth}, and
    \ref{prop:spac_smooth}, we are only missing the spatial smooth fit
    for the infinite-horizon case.

    Let $((t_n,x_n))_{n\in\mathbb{N}}$ be as in Lemma \ref{lem:regular_boundary} (for $T=\infty$). 
    Since $b_\infty$ is locally Lipschitz by Proposition
    \ref{prop:bound_lip}, $b_\infty(t_n)\to b_\infty(t)$ and $|x_n-b_\infty(t_n)|\to0$.
    By restricting to sufficiently large $n$, all the points
    $(t_n,x_n)$ and $(t_n,b_\infty(t_n))$ lie in a common compact subset of $[0,\infty)\times(0,1)$. 
    
    Using the inequality
    \begin{align*} \left| \partial_xV_\infty(t_n,x_n) - \partial_xV_\infty(t,b_\infty(t)) \right|&\leq \left| \partial_xV_\infty(t_n,x_n) - \partial_xV_\infty(t_n,b_\infty(t_n)) \right|\\
    &\quad + \left| \partial_xV_\infty(t_n,b_\infty(t_n)) - \partial_xV_\infty(t,b_\infty(t)) \right|,
    \end{align*} 
    we will establish that each term vanishes as $n\to\infty$.
    By Proposition \ref{prop:uniform_conv_dx} and the estimate
    \eqref{eq:lip_partialx} (which is uniform in the horizon $T$), there is a constant $K$ depending on the compact such that
    \begin{align*}
        \left|
            \partial_xV_\infty(t_n,x_n)
            -
            \partial_xV_\infty(t_n,b_\infty(t_n))
        \right|&=
        \lim_{T\to\infty}
        \left|
            \partial_xV_T(t_n,x_n)
            -
            \partial_xV_T(t_n,b_\infty(t_n))
        \right|\\
        &\quad\leq
        K |x_n-b_\infty(t_n)|
        \to 0.
    \end{align*}
    At the same time, Proposition \ref{prop:uniform_conv_dx} implies that
    $x\mapsto V_\infty(s,x)$ is differentiable on $(0,1)$ for every
    $s\geq0$. Since $V_\infty(s,y)=g(y)$, $y\geq b_\infty(s)$, differentiability at the boundary gives $\partial_xV_\infty(s,b_\infty(s)) = g'(b_\infty(s))$, $s\geq0$. Consequently, \begin{align*} \left| \partial_xV_\infty(t_n,b_\infty(t_n)) - \partial_xV_\infty(t,b_\infty(t)) \right|= \left| g'(b_\infty(t_n)) - g'(b_\infty(t)) \right| \to 0 \end{align*} by continuity of $b_\infty$ and $g'$.
    This completes the proof for the upper boundary, while the lower boundary holds similarly by symmetry. Hence we are done.
\end{proof}

With the smooth-fit property of the value function in finite horizon we are finally able to verify that the optimal stopping boundary is away from the set $\mathcal{U}$. This in turn will give us the final part of Theorem \ref{thm:main_structure}.

\begin{proof}[Proof of Theorem \ref{thm:main_structure}]
    The optimality of the first hitting time to the stopping set was proven in Proposition \ref{prop:semicontinuity} for finite horizon and Proposition \ref{prop:finite_to_infinite_1} for infinite horizon.
    The existence of a lower semicontinuous boundary $b_T$ such that $b_T(t)\geq\gamma(t)$ and the continuation set is given by \eqref{eq:continuation_set_b} was proven in Proposition \ref{prop:two_boundaries}.
    From Lemma \ref{lem:b_uniform_upper_bound}, we have $b_T(t)<1$, and so we are only missing $b_T(t)>\gamma(t)$ for all $t\in[0,T)$ with $T\leq \infty$.

    If the conclusion holds for finite $T$, then it can be extended to infinite horizon by noting that $T\mapsto b_T$ is increasing. Let $T<\infty$ and assume for contradiction that $b_T(t_0)=\gamma(t_0)$ for some $t_0\in(0,T)$. Choose $\delta,\eta>0$, with $t_0+\delta<T$, such that $(t,\gamma(t)+z)\in\mathcal{U}$ whenever $|t-t_0|<\delta$ and $z\in(-\eta,0)$. Set
    \[
        \overline W(t,z):=W_T(t,\gamma(t)+z).
    \]
    Then $\overline W<0$ in this cylinder and $\overline W(t_0,0)=0$. Moreover,
    \[
        \partial_t\overline W
        -\gamma'(t)\partial_z\overline W
        +\frac12\sigma^2(\gamma(t)+z)\partial_{zz}\overline W
        =-\bigl(c(t)+\mathcal Lg(\gamma(t)+z)\bigr)>0.
    \]
    The operator is uniformly parabolic on the closure of a smaller cylinder. After reversing time, the parabolic Hopf lemma, see e.g. \cite[Theorem 3.2.3]{protter2012maximum}, at the flat lateral boundary $z=0$ gives
    \[
        \partial_z\overline W(t_0,0)>0.
    \]
    On the other hand, spatial smooth fit gives
    \[
        \partial_z\overline W(t_0,0)
        =\partial_xW_T(t_0,\gamma(t_0))=0,
    \]
    which is a contradiction. Hence $b_T(t)>\gamma(t)$ for $t\in(0,T)$.
    The conclusion can be extended to $t\in[0,T)$ by arguing as around \eqref{eq:extend_0}.    
\end{proof}

The separation of the optimal stopping boundary from $\mathcal{U}$ is useful since it yields a strict inequality $c(t)+\mathcal{L}g(x)>0$ close to the boundary. First, this can be utilized to conclude that the smooth fit is in some sense the maximal regularity we can achieve as stated in Corollary \ref{cor:main_discont_dxx}.

\begin{proof}[Proof of Corollary \ref{cor:main_discont_dxx}]
    The first part is given by Corollary \ref{cor:main_discont_dxx_pt1}.
    For the second part, Proposition \ref{prop:PDE} yields that $\partial_{xx} V_T$ is continuous in $\mathcal{C}_T$. We can utilize \eqref{eq:time_smooth_fit} to get
    \begin{equation*}
        \lim_{n\to\infty}\partial_{xx}V_T(t_n,x_n)=-\frac{2c(t)}{\sigma^2(b_T(t))},
    \end{equation*}
    for a sequence $(t_n,x_n)\to (t,b_T(t))$ with $(t_n,x_n)\in \mathcal{C}_T$ for all $n\in\mathbb{N}$. Now,
    combining this with the fact that $\partial\mathcal{D}_T\subset \mathcal{S}_T\setminus \overline{\mathcal{U}}$, where $\overline{\mathcal{U}}$ denotes the closure of $\mathcal{U}$ in $\mathcal{S}$, we obtain
    \begin{equation*}
        g''(b_T(t))>-\frac{2c(t)}{\sigma^2(b_T(t))}.
    \end{equation*}
    Thus, $\partial_{xx}V_T$ has a discontinuity across the boundary.
\end{proof}

With the smooth-fit property of the value function, Lipschitz stopping boundaries, and the separation of the boundary and the $\mathcal{U}$ set, we can upgrade to continuously differentiable optimal stopping boundaries by \cite[Theorem 2.4]{de2024probabilistic}, thereby proving Corollary \ref{cor:main_c1}.
\begin{proof}[Proof of Corollary \ref{cor:main_c1}]
    The result follows if we can verify \cite[Assumption 2.3]{de2024probabilistic}. Their problem is formulated as a maximization problem, and we implicitly use the standard conversion from minimization with running cost to maximization with terminal gain.
    Take any $t_0\in(0,T)$. By Theorem \ref{thm:main_structure} giving $\gamma(t)<b_T(t)<1$ for all $t<T$ it holds that there exists an open rectangle $\mathcal{R}$ with closure $\overline{\mathcal{R}}\subset [0,T)\times (0,1)$ and $(t_0,b_T(t_0))\in\mathcal{R}$ such that: (i) $V_T$ is continuously differentiable in time and space on $\overline{\mathcal{R}}$, by Theorem \ref{thm:main_regular_smooth}, (ii) with some abuse of notation, $g\in C^2(\overline{\mathcal{R}})$, $c\in C^1(\overline{\mathcal{R}})$, and $h(t,x)=c(t)+\mathcal{L}g(x)>0$, by Assumption \ref{asmp:sigma_g} on $g$, Assumption \ref{asmp:base_penalty:c}, and Theorem \ref{thm:main_structure}, (iv) $\sigma\in C^1(\overline{\mathcal{R}})$ bounded away from 0 with $\sigma'$ being Lipschitz continuous on $\overline{\mathcal{R}}$, by Assumption \ref{asmp:sigma_g} on $\sigma$, and since $\sigma(x)=0$ if and only if $x\in\{0,1\}$, which the boundary is away from. Assumption (iii) in \cite[Assumption 2.3]{de2024probabilistic} is trivially satisfied since $X$ has no drift and the optimal stopping problem is without discounting. By arguing as in \eqref{eq:extend_0} we can extend to $[0,T)$.

    Thus, since $b_T$ is Lipschitz continuous on $[0,T_1]$ for any $T_1<T$ by Proposition \ref{prop:bound_lip}, we have the claim by applying \cite[Theorem 2.4]{de2024probabilistic}.
\end{proof}

\section{Free-Boundary Integral Equation}\label{sec:integral_eqs}
To get an explicit equation for the boundary in both the finite- and infinite-horizon cases, we can utilize the previous results to obtain a nonlinear integral equation under Assumption \ref{asmp:extended}, which we will assume throughout the section. We treat the finite- and infinite-horizon cases simultaneously, and we obtain uniqueness in both cases. For any function $f:[0,T)\to (1/2,1)$, we define the family of sets 
\begin{equation*}
    \mathcal{C}^f_t:=\{x:x\in(1-f(t),f(t))\} 
\end{equation*}
for $t\in[0,T)$. In the proof below we suppress the $T$ notation for $b_T$. Note in particular that $\mathcal{C}^{b_T}_t=\{x\in(0,1) : (t,x)\in\mathcal{C}_T\}$.

\begin{proof}[Proof of Theorem \ref{thm:main_integral_equation}]
       Take any $(t,x)\in\mathcal{S}_T$ and let $r\in(0,T-t)$. Define $r_m=r\wedge\eta_m$, where $\eta_m$ is as in Proposition \ref{prop:lagrange}. For simplicity, we write $b(t):=b_T(t)$. On every compact time interval, the curves $b$ and $1-b$ are separated and of bounded variation. Moreover, $V_T$ is $C^{1,2}$ off the curves
       and the spatial concavity gives the local variation condition in \cite[Remark 3.2]{peskir2005change}. The change-of-variables formula \cite[Theorem 3.1]{peskir2005change} may therefore be applied successively to the two curves, yielding
    \begin{align*}
        V_T(t+r_m,X^x_{r_m})-V_T(t,x)=\quad&\int_0^{r_m}(\partial_tV_T+\mathcal{L}V_T)(t+u,X^x_{u})\ind{X^x_{u}\not\in\partial\mathcal{C}^b_{t+u}}\,du\\
        +&\int_0^{r_m}\sigma(X^x_{u})\partial_xV_T(t+u,X^x_{u})\ind{X^x_{u}\not\in\partial\mathcal{C}^b_{t+u}}\,dB_{u}\\
        +&\int_0^{r_m}\frac{1}{2}\Delta_x\partial_xV_T(t+u,X^x_{u})\ind{X^x_{u}\in\partial\mathcal{C}^b_{t+u}}\,d\ell_{u},
    \end{align*}
    where $\ell_u$ is the sum of the local times of $X^x$ along $1-b$ and $b$, and we have defined the operator $\Delta_x$ by 
    \begin{equation*}
        \Delta_x\partial_xV_T(t,x)=\partial_xV_T(t,x+)-\partial_xV_T(t,x-).
    \end{equation*}
    The local time integral vanishes by the spatial smooth fit, Theorem \ref{thm:main_regular_smooth}. Taking expectation, we note that the local martingale term vanishes, since $\sigma(X^x_{u})\partial_xV_T(t+u,X^x_{u})$ is bounded on $[0,r_m]$. Thus, we have
    \begin{align*}
        \mathbb{E}[V_T(t+r_m,X^x_{r_m})]-V_T(t,x)&=\mathbb{E}\Big[\int_0^{r_m}(\partial_tV_T+\mathcal{L}V_T)(t+u,X^x_{u})\ind{X^x_{u}\not\in\partial\mathcal{C}^b_{t+u}}\,du\Big]\\
        &=\mathbb{E}\left[\int_0^{r_m}\left(-c(t+u)\ind{X^x_{u}\in\mathcal{C}^b_{t+u}}+\mathcal{L}g(X^x_{u})\ind{X^x_{u}\not\in\mathcal{C}^b_{u+t}}\right)\,du\right].
    \end{align*}
    The integrand in the last expression is nonpositive. As $m\to\infty$, bounded convergence applies to the value term, while monotone convergence applies to the integral term. Hence the identity holds with $r_m$ replaced by $r$.

    Now let $r\uparrow T-t$. If $T<\infty$, the Lagrange formulation gives, for $h>0$ and $y\in[0,1]$,
    \[
        -h\|\mathcal Lg\|_\infty
        \leq W_T(T-h,y)\leq0.
    \]
    Thus $V_T(T-h,\cdot)\to g(\cdot)$ uniformly as $h\downarrow0$. Since $X^x$ has continuous paths and $g$ is continuous,
    \[
        V_T(t+r,X_r^x)\to g(X_{T-t}^x)
    \]
    almost surely and in $L^1$. If $T=\infty$, then $0\leq V_\infty(t+r,X_r^x)\leq g(X_r^x)\to0$ almost surely and in $L^1$. Letting $r\uparrow T-t$ and rearranging therefore gives
    \begin{equation}\label{eq:cov_V}
        V_T(t,x)=\mathbb{E}\left[\int_0^{T-t}c(t+u)\ind{X^x_{u}\in\mathcal{C}^b_{t+u}}\,du-\int_0^{T-t}\mathcal{L}g(X^x_{u})\ind{X^x_{u}\not\in\mathcal{C}^b_{t+u}}\,du+g(X^x_{T-t})\right],
    \end{equation}
    Evaluating \eqref{eq:cov_V} at $x=b(t)$ and $x=1-b(t)$, we obtain \eqref{eq:main_boundary_integral}.

    Next, we show uniqueness. Assume that there exists a solution $\beta$ to \eqref{eq:main_boundary_integral} such that $\beta(t)\in(\gamma(t),1)$ for all $t< T$ and $\beta\in C([0,T))$. We define the $\beta$--analogue of $V_T$ in \eqref{eq:cov_V} by
    \begin{equation*}
        V_T^\beta(t,x):=\mathbb{E}\left[\int_0^{T-t}F^\beta(t+u,X^x_u)\,du+g(X^x_{T-t})\right],
    \end{equation*}
    where
    \begin{equation*}
        F^\beta(u,x):=c(u)\ind{x\in\mathcal{C}^\beta_{u}}-\mathcal{L}g(x)\ind{x\not\in\mathcal{C}^\beta_{u}}.
    \end{equation*}
    Note in particular that since $F^\beta\geq0$, the integral is well-defined but possibly infinite when $T=\infty$.
    Define the first hitting time 
    \begin{equation*}
        \tau^f(t,x)=\inf\{0\leq s<T-t:X^x_s\in\partial \mathcal{C}^f_{t+s}\}\wedge(T-t),
    \end{equation*}
    for $f\in\{b,\beta\}$.
    
    Take any $(t,x)\in\mathcal{S}_T$ and denote $\tau^\beta:=\tau^\beta(t,x)$. Since $F^\beta\geq0$, Tonelli's theorem and the strong Markov property justify the following decomposition also when $T=\infty$:
    \begin{equation}\label{eq:U_markov}
        V^\beta_T(t,x)=\mathbb{E}\left[\int_0^{\tau^\beta}F^\beta(t+u,X^x_u)\,du+\ind{\tau^\beta<T-t}V_T^\beta(t+\tau^\beta,X^x_{\tau^\beta})+\ind{\tau^\beta=T-t}g(X^x_{T-t})\right].
    \end{equation}
    On $\{\tau^\beta<T-t\}$, we have $X^x_{\tau^\beta}=\beta(t+\tau^\beta)$ or $X^x_{\tau^\beta}=1-\beta(t+\tau^\beta)$, and since $\beta$ and $1-\beta$ solve \eqref{eq:main_boundary_integral}, we have
    \begin{equation*}
        V^\beta_T(t+\tau^\beta,\beta(t+\tau^\beta))=g(\beta(t+\tau^\beta)),
    \end{equation*}
    by symmetry. Thus \eqref{eq:U_markov} becomes
    \begin{equation}\label{eq:U_markov_2}
        V^\beta_T(t,x)=\mathbb{E}\left[\int_0^{\tau^\beta}F^\beta(t+u,X^x_u)\,du+g(X^x_{\tau^\beta})\right].
    \end{equation}
    If $x\in(0,1)\setminus\mathcal{C}^\beta_t$, then $X^x_u\not\in\mathcal{C}^\beta_{t+u}$ for all $u<\tau^\beta$, so by the definition of $F^\beta$, we have
    \begin{equation}\label{eq:U_eq_g}
        V^\beta_T(t,x)=\mathbb{E}\left[-\int_0^{\tau^\beta}\mathcal{L}g(X^x_u)\,du+g(X^x_{\tau^\beta})\right]=g(x),
    \end{equation}
    where the last equality follows from Proposition \ref{prop:lagrange}. Similarly, if $x\in\mathcal{C}^\beta_t$, then $X^x_u\in\mathcal{C}^\beta_{t+u}$ for all $u<\tau^\beta$, so by the definition of $F^\beta$, we have
    \begin{equation}\label{eq:U_geq_V}
        V^\beta_T(t,x)=\mathbb{E}\left[\int_0^{\tau^\beta}c(t+u)\,du+g(X^x_{\tau^\beta})\right]\geq V_T(t,x),
    \end{equation}
    since $\tau^\beta$ is a suboptimal stopping time. In particular, if $\tau^\beta\not\in\mathcal{T}$, then the inequality trivially holds. Now consider the Lagrange-formulated version of $V^\beta_T$ denoted by $W_T^\beta(t,x):=V_T^\beta(t,x)-g(x)$. The integral equation \eqref{eq:main_boundary_integral} yields $W_T^\beta(t,\beta(t))=0$ and $W_T^\beta$ is given by
    \begin{equation*}
        W_T^\beta(t,x)=\mathbb{E}\left[\int_0^{T-t}H(t+u,X^x_u)\ind{X^x_u\in\mathcal{C}^\beta_{t+u}}\,du\right],
    \end{equation*}
    where, utilizing Proposition \ref{prop:lagrange}, we get
    \begin{equation*}
        H(u,x):=c(u)+\mathcal{L}g(x).
    \end{equation*}
    Note in particular that this is a well-defined improper integral. By Proposition \ref{prop:lagrange},
    \[
        \mathbb{E}\left[\int_0^{T-t}-\mathcal{L}g(X_u^x)\,du\right]\leq g(x),
    \]
    so only the nonnegative $c$ term can diverge.
    Equation \eqref{eq:U_eq_g} and \eqref{eq:U_geq_V} then yield $W_T^\beta(t,x)=0$ for $x\in(0,1)\setminus\mathcal{C}^\beta_t$ and $W_T^\beta(t,x)\geq W_T(t,x)$ for $x\in\mathcal{C}^\beta_t$.

    Assume for contradiction that there exists a $t<T$ such that $\beta(t)>b(t)$. Since $\beta$ and $b$ are continuous and differ at $t$, there is a $\delta\in(0,T-t)$ such that $\beta(t+u)>b(t+u)$ for $u\in[0,\delta]$. Furthermore, one can find a compact rectangle $R$ such that 
    \begin{equation*}
        R\subset\{(t+u,y)\in\mathcal{S}_T:u\in[0,\delta],y\in(b(t+u),\beta(t+u))\}.
    \end{equation*} 
    Put $x=\beta(t)$ and $\rho:=\tau^b(t,x)\wedge\delta$. Then $W_T^\beta(t,x)=0$ and 
    the strong Markov property gives
    \begin{equation}\label{eq:integral_eq_contradiction}
        0=W_T^\beta(t,x)
        ={}\mathbb{E}\!\left[\int_0^\rho H(t+u,X_u^x)\ind{X_u^x\in\mathcal{C}_{t+u}^\beta}\,du\right]
        +\mathbb{E}\!\left[W_T^\beta(t+\rho,X_\rho^x)\right].
    \end{equation} 
    The point $(t+\rho,X_\rho^x)$ lies in the true stopping set. Indeed, if $\rho=\tau^b(t,x)$, the process is on its boundary, while if $\rho=\delta<\tau^b(t,x)$, it remains above the upper boundary. Hence $W_T(t+\rho,X_\rho^x)=0$. Since $W_T^\beta\geq W_T$ in $\mathcal{C}^\beta$ and $W_T^\beta=0$ outside $\mathcal{C}^\beta$, the terminal term above is nonnegative. 
    Thus \eqref{eq:integral_eq_contradiction} gives
    \begin{equation}\label{eq:integral_eq_contradiction_2}
        0=W_T^\beta(t,x)
        \geq\mathbb{E}\!\left[\int_0^\rho H(t+u,X_u^x)\ind{X_u^x\in\mathcal{C}_{t+u}^\beta}\,du\right]
    \end{equation}
    Moreover, for $u<\rho$, continuity of the paths gives $X_u^x>b(t+u)$. Consequently, on $\{X_u^x\in\mathcal{C}_{t+u}^\beta\}$,
    \[
        b(t+u)<X_u^x<\beta(t+u).
    \]
    Since $b(t+u)>\gamma(t+u)$, we have $H(t+u,X_u^x)>0$ on this event. Thus, the integral in \eqref{eq:integral_eq_contradiction_2} above is nonnegative. Furthermore, the local nondegeneracy gives positive probability that the process enters $R$ and spends positive Lebesgue time there. The first expectation is therefore strictly positive, which is a contradiction. Hence $\beta(t)\leq b(t)$ for all $t<T$.
    
    Now assume for contradiction that there exists a $t<T$ such that $\beta(t)< b(t)$. Taking $x=b(t)$, we thus get 
    \begin{equation*}
        0=W_T(t,x)-W_T^\beta(t,x)=\mathbb{E}\left[\int_0^{T-t}H(t+u,X^x_u)\ind{X^x_{u}\in \mathcal{C}^b_{t+u}\setminus \mathcal{C}^\beta_{t+u}}\,du\right].
    \end{equation*}
    Again, $H>0$ on $\mathcal{C}^b_u\setminus \mathcal{C}^\beta_u$ for any $u\in(t,T)$ and there is nonzero probability of staying in that set for nonzero Lebesgue time. This gives another contradiction, and we conclude that $\beta(t)=b(t)$.
\end{proof}

\section{Examples}\label{sec:examples}
We now illustrate the preceding results for the win-martingales introduced in Section \ref{sec:motivation}. The main
examples are the Aldous, Bass, and sequential testing win-martingales, whose time-changed volatility coefficients are
\begin{equation*}
    \sigma_A(x)=\frac{\sin(\pi x)}{\pi},
    \qquad
    \sigma_B(x)=\varphi(\Phi^{-1}(x)),
    \qquad 
    \sigma_I(x)=x(1-x).
\end{equation*}
We consider the terminal costs induced by optimal actions as in Section \ref{sec:motivation}:
\begin{equation*}
    g_{L^2}(x)=x(1-x),\qquad
    g_{\mathrm{CE}}(x)
    =-x\log x-(1-x)\log(1-x),
\end{equation*}
where we note that $g_{\mathrm{CE}}$ fails to obey Assumption \ref{asmp:bounded:g'}, and thus, we have to make use of Assumption \ref{asmp:bounded:c} for cross-entropy loss.

\subsection{Constant optimal stopping boundaries}

\begin{example}
We first compare the Aldous and Bass win-martingales in the homogeneous infinite-horizon setting.
The function $\rho$ is given by $\rho(s)=1/\sqrt{S-s}$
so the time transformation in \eqref{eq:A_motivation} is
\begin{equation*}
    A(s)=\int_0^s\frac{1}{S-u}\,du
    =
    -\log\left(\frac{S-s}{S}\right),
    \qquad
    \Gamma(t)=S(1-e^{-t}),
\end{equation*}
and hence $T=\infty$. If
\begin{equation*}
    f(s)=\frac{c_0}{S-s},
\end{equation*}
then $c(t)\equiv c_0$ by \eqref{eq:c_motivation}.
Thus, the time-changed problem is precisely the homogeneous infinite-horizon problem studied in
Section \ref{sec:homogeneous}. The optimal rule is therefore the first exit time from a fixed interval given by Theorem \ref{VerificationInfConstant}.

In Figure \ref{fig:constant-boundaries} we take $c_0=0.075$ and use the $L^2$--loss. 
The dotted curve corresponds to the Aldous win-martingale and the dashed curve
to the Bass win-martingale. Numerically, the Aldous value function lies above the Bass value function, meaning
that, under the same terminal loss and running cost, the Aldous game is more costly to predict
optimally. In this sense it is the more difficult, or more exciting, game to call.
\end{example}

\begin{figure}[ht]
    \centering
    \includegraphics[width=0.7\textwidth]{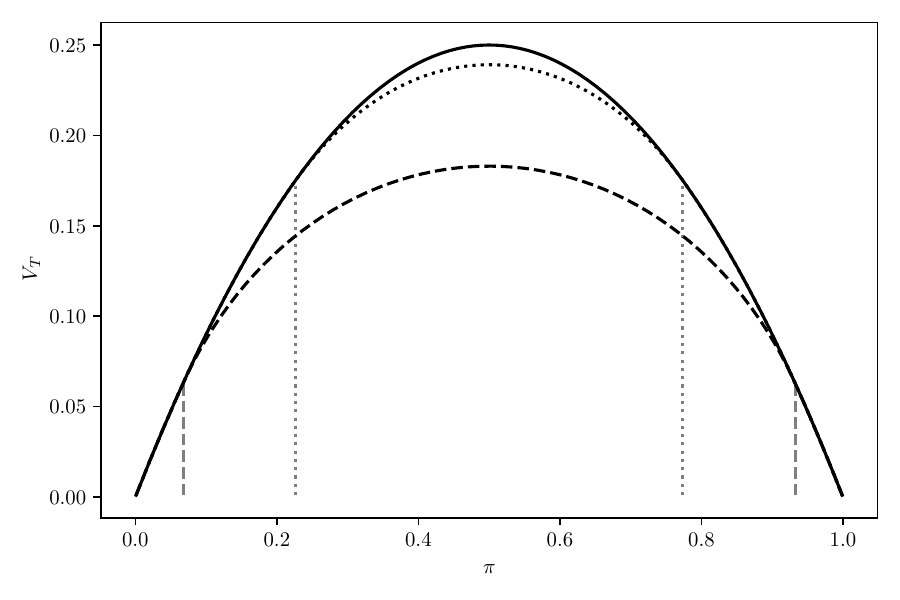}
    \caption{Numerical plot of the value function $x\mapsto V(x)$ for the Aldous (dotted) and Bass (dashed) win-martingales with
    $L^2$--loss and constant transformed running cost $c_0=0.075$. The vertical lines indicate the
    optimal stopping boundaries from Theorem \ref{VerificationInfConstant}, and the solid curve is $x\mapsto g_{L^2}(x)$.}
    \label{fig:constant-boundaries}
\end{figure}

\subsection{Nonconstant optimal stopping boundaries}

If the transformed running cost is monotone, then the optimal stopping boundary inherits a
corresponding monotonicity. This monotonicity is determined by the cost-to-noise ratio
$f/\rho^2$.

\begin{proposition}\label{prop:monotone}
Assume Assumption \ref{asmp:extended} for the time transformed problem. If $s\mapsto f(s)/\rho^2(s)$ is increasing, then the continuation region of the original problem is shrinking. Equivalently, $t\mapsto b_T(t)$ is decreasing.

If $A(s)\to\infty$ as $s\to S$ and $s\mapsto f(s)/\rho^2(s)$ is decreasing, then the continuation region of the original problem is expanding. Equivalently, $t\mapsto b(t)$ is increasing.
\end{proposition}

\begin{proof}
By the definition of the time change and \eqref{eq:c_motivation},
\begin{equation*}
    c(t)=f(\Gamma(t))\Gamma'(t)
    =
    \frac{f(\Gamma(t))}{\rho^2(\Gamma(t))}.
\end{equation*}
Since $\Gamma$ is increasing, $c$ has the same monotonicity as $f/\rho^2$.

Suppose first that $c$ is increasing and let $t_1\le t_2<T$. Using an optimal stopping time for
the problem started from $(t_2,x)$ as an admissible stopping time for the problem started from
$(t_1,x)$, we obtain $W_T(t_2,x)\geq W_T(t_1,x)$ by arguing as in \eqref{eq:monotone_W_lower}.
Thus, if $(t_1,x)\in \mathcal{D}_T$, then $0\geq W_T(t_2,x)\ge W_T(t_1,x)=0$, so it follows
that $(t_2,x)\in \mathcal{D}_T$. Therefore, the stopping set expands with time, and $t\mapsto b_T(t)$ is
decreasing.

If $T=\infty$ and $c$ is decreasing, the reverse comparison gives
\begin{equation*}
    W_\infty(t_2,x)-W_\infty(t_1,x)
    \le
    \mathbb{E}\left[
        \int_0^{\tau^*_\infty(t_1,x)}
        \big(c(t_2+u)-c(t_1+u)\big)\,du
    \right]
    \le 0.
\end{equation*}
Hence the continuation set expands, and $t\mapsto b(t)$ is increasing.
\end{proof}

\begin{remark}
    It should be noted that for the second case of Proposition \ref{prop:monotone} it
    is necessary that we assume $T=\infty$. If $T<\infty$, then $\tau^*_T(t_1,x)$ may not be admissible for the problem with initial data $(t_2,x)$.
    In fact, the set of stopping times over which the infimum is taken decreases, and so one cannot expect to get a lower value alone from the decrease in the temporal cost.
\end{remark}

Thus, the monotonicity of the continuation region is governed by $f/\rho^2$, not by $f$ alone.
In particular, an increasing running cost in the original time scale may still lead to an expanding
continuation region if the noise level increases sufficiently fast.

\begin{example}
We illustrate this point for the classical sequential testing win-martingale on the horizon $S=1$ with the same time scaling
\begin{equation*}
    d\Pi_s^I=\frac{\Pi_s^I(1-\Pi_s^I)}{\sqrt{S-s}}\,dW_s,
\end{equation*}
using the $L^2$--loss and three increasing original time costs.

The left panel of Figure \ref{fig:three-boundaries} uses
\begin{equation*}
    f(s)=\frac{k(s+\alpha)}{1-s},\qquad c(t)=k(1+\alpha-e^{-t})
\end{equation*}
with $k=0.03$ and $\alpha=1$. The transformed running cost is increasing, resulting in a shrinking continuation region. The middle panel uses
\begin{equation*}
    f(s)=k(s+\alpha),\qquad  c(t)=k(1+\alpha-e^{-t})e^{-t},
\end{equation*}
with $k=0.24$ and $\alpha=0.01$. The transformed running cost is nonmonotone, and so the resulting optimal stopping boundary is correspondingly nonmonotone. The right panel uses the constant running cost
\begin{equation*}
    f(s)=k,\qquad c(t)=ke^{-t},
\end{equation*}
with $k=0.06$. The transformed running cost is decreasing. Hence, the continuation region expands over time.

These three choices show that monotone running costs in the original time scale can produce
shrinking, nonmonotone, or expanding continuation regions after the time change.
It can be verified that all of the examples satisfy Assumption \ref{asmp:extended}. Namely, the first satisfies Assumption \ref{asmp:bounded:g'} and \ref{asmp:bounded:c} while the latter two satisfy only \ref{asmp:bounded:g'}.
\end{example}

\begin{figure}[ht]
    \centering
    \includegraphics[width=\textwidth]{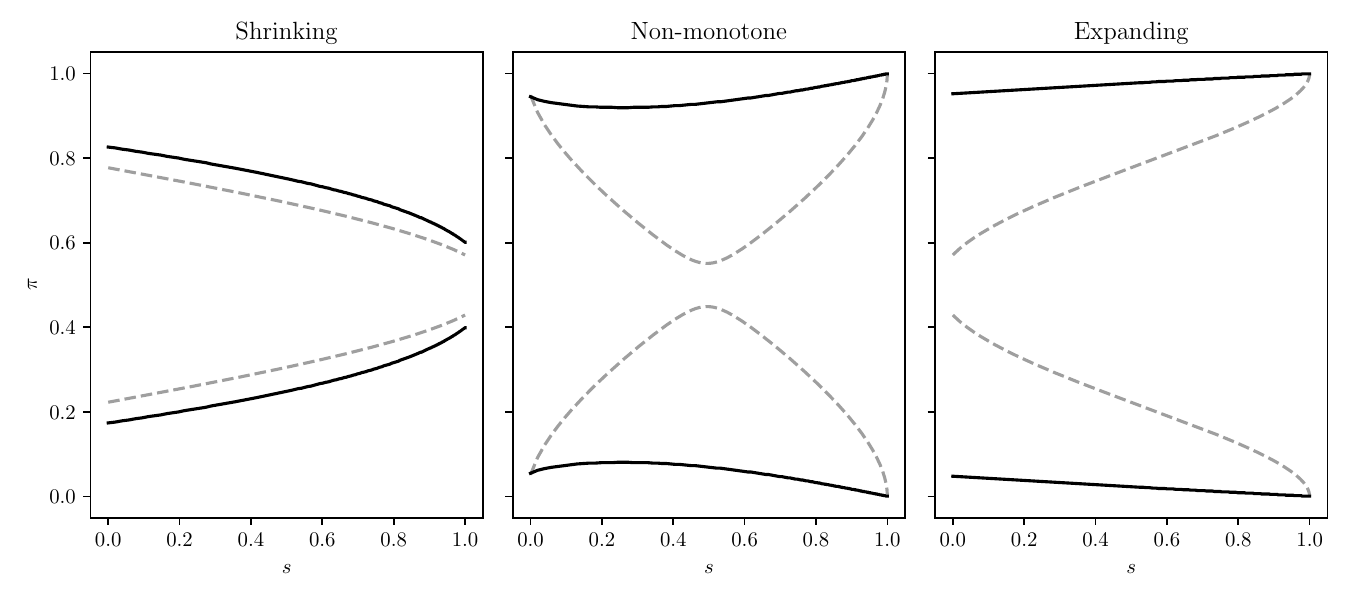}
    \caption{Numerical plot of the optimal stopping boundaries for the classical sequential testing win-martingale with $L^2$--loss and
    $S=1$. The three panels correspond to transformed running costs that are increasing,
    nonmonotone, and decreasing, respectively. Solid curves are the optimal stopping boundaries while the dashed lines are $\gamma$ and $1-\gamma$.}
    \label{fig:three-boundaries}
\end{figure}

\begin{example}\label{ex:oscillatory}
Finally, we consider an oscillatory transformed running cost for the Aldous win-martingale. We take
$S=1$, use the cross-entropy loss, and prescribe the transformed running cost directly by
\begin{equation}\label{eq:c_oscill}
    c(t)=k_1+k_2\sin(2\pi t),
\end{equation}
where $k_1=0.125$ and $k_2=0.075$. In the original time scale this
corresponds to
\begin{equation*}
    f(s)
    =
    \frac{k_1+k_2\sin(-2\pi\log(1-s))}{1-s}.
\end{equation*}
It can easily be verified that Assumption \ref{asmp:extended} holds, specifically with Assumption \ref{asmp:bounded:c} by the choice of $k_1>k_2$, while Assumption \ref{asmp:bounded:g'} does not hold since $g'_{\text{CE}}$ is unbounded.

The optimal stopping boundary is plotted in Figure \ref{fig:oscillatory-boundary}. Since the
transformed running cost is periodic in $t$, the boundary oscillates in the transformed time
variable. Consequently, when viewed in the original time variable $s$, namely, as $b(A(s))$, the
boundary need not have a limit as $s\nearrow1$. This illustrates that the local $C^1$--regularity
of the boundary in transformed time cannot be extended to the terminal time point.
We give a proof of this fact in the following proposition.
\end{example}

\begin{figure}[ht]
    \centering
    \includegraphics[width=0.7\textwidth]{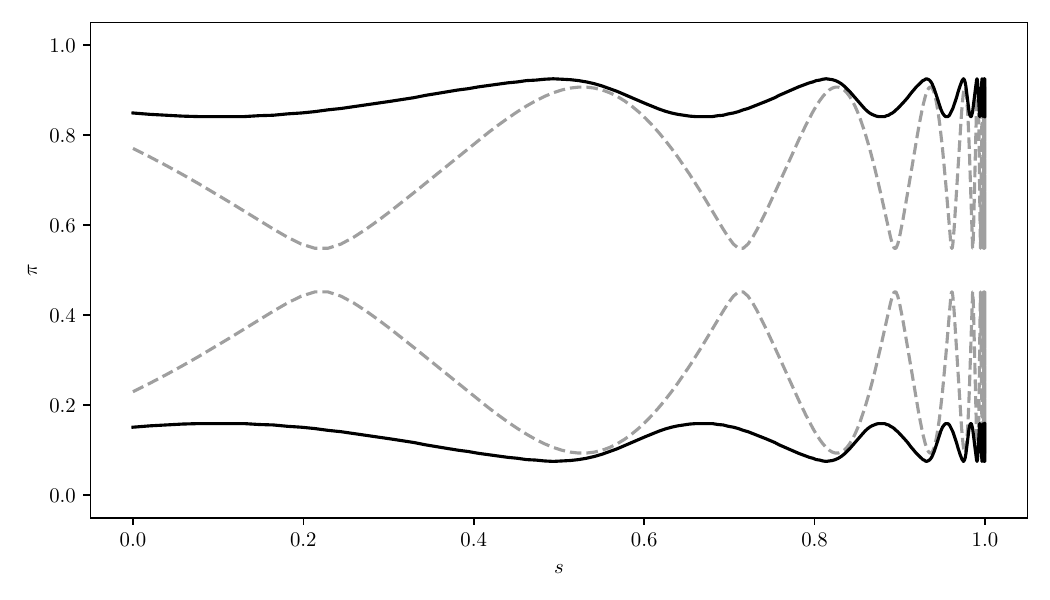}
    \caption{Numerical plot of the optimal stopping boundary for the Aldous win-martingale with cross-entropy loss and
    sinusoidal transformed running cost $c(t)=k_1+k_2\sin(2\pi t)$. The dashed curves show $\gamma$ and $1-\gamma$.}
    \label{fig:oscillatory-boundary}
\end{figure}

\begin{proposition}
    Consider an optimal stopping problem \eqref{eq:motivation_OSP_pi} such that the time transformation obeys $A(s)\to\infty$ as $s\to S$. Assume, furthermore, that the transformed optimal stopping problem satisfies Assumption \ref{asmp:extended} and the transformed running cost is given by \eqref{eq:c_oscill} for some constants $k_1> |k_2|>0$. Then the optimal stopping boundary in the original problem $s\mapsto b_\infty(A(s))$ does not converge as $s\to S$.
\end{proposition}
\begin{proof}
    Since $A(s)\to \infty$ for $s\to S$, we need to show $t\mapsto b_\infty(t)$ does not converge as $t\to\infty$. As noted in Example \ref{ex:oscillatory}, $c$ is periodic with period $1$, that is, $c(t+1)=c(t)$. Thus $t\mapsto V(t,x)$ is periodic, and so by the characterizations of $\mathcal{C}$ in \eqref{eq:stop_and_cont_sets} and \eqref{eq:main_continuation_region}, it holds that $t\mapsto b_\infty(t)$ is periodic. Thus it is enough to show that the optimal stopping boundary is not constant. 

    Assume for contradiction that $b_\infty(t)\equiv B$ and let 
    \begin{equation*}
        \tau(x):=\inf\{u\geq0:X^x_u\not\in(1-B,B)\}.
    \end{equation*}
    It holds that $\mathbb{E}[\tau(x)]<\infty$ since the diffusion is nondegenerate in $(1-B,B)$.
    Then, by \eqref{eq:c_oscill}, the value function is given by 
    \begin{align}
        V(t,x)&= \mathbb{E}\left[g(X^x_{\tau(x)})+k_1\tau(x)\right]+k_2\mathbb{E}\left[\int_0^{\tau(x)}\sin(2\pi(t+u))\,du\right]\\
        &=H(x)+\frac{k_2}{2\pi}(\cos(2\pi t)(1-C(x))+\sin(2\pi t)S(x)),\label{eq:V_osc_rhs}
    \end{align}
    where the second equality follows by integration and the trigonometric identity
    \begin{equation*}
        \cos(a+b)=\cos(a)\cos(b)-\sin(a)\sin(b),
    \end{equation*}
    and we have defined the functions
    \begin{equation*}
        H(x):=\mathbb{E}\left[g(X^x_{\tau(x)})+k_1\tau(x)\right],\qquad C(x):=\mathbb{E}[\cos(2\pi \tau(x))],\qquad S(x):=\mathbb{E}[\sin(2\pi \tau(x))].
    \end{equation*}
    By Theorem \ref{thm:main_regular_smooth}, smooth fit implies 
    \begin{equation*}
        g'(B)=H'(B-)+\frac{k_2}{2\pi}(-\cos(2\pi t)C'(B-)+\sin(2\pi t)S'(B-)),
    \end{equation*}
    for all $t\geq0$ where we have used that the terms in \eqref{eq:V_osc_rhs} are $C^2((1-B,B))$. Since this holds for all $t\geq0$, the last term must vanish. However, it is a linear combination of sine and cosine, so we must also have
    \begin{equation}\label{eq:C'_S'_vanish}
        C'(B-)=S'(B-)=0.
    \end{equation}
    Now, let $u(t,x):=\mathbb{E}[\cos(2\pi (t+\tau(x)))]$. Then $u$ solves the Dirichlet problem,
    \begin{align}
        \partial_t u(t,x)+\mathcal{L}u(t,x)&=0, &&x\in(1-B,B),\label{eq:oci_pde}\\
        u(t,x)&=\cos(2\pi t),&& x\in\{1-B,B\},
    \end{align}
    since the first exit time from $(1-B,B)$ is almost surely finite.
    Again
    \begin{equation*}
        u(t,x)=\cos(2\pi t)C(x)-\sin(2\pi t)S(x),
    \end{equation*}
    so \eqref{eq:oci_pde} yields
    \begin{equation*}
        \cos(2\pi t)\big(\mathcal{L}C(x)-2\pi S(x)\big)-\sin(2\pi t)\big(\mathcal{L}S(x)+2\pi C(x)\big)=0.
    \end{equation*}
    This implies 
    \begin{equation}\label{eq:oci_pdes_both}
        \mathcal{L}C(x)=2\pi S(x),\qquad \mathcal{L}S(x)=-2\pi C(x),
    \end{equation}
    for $x\in(1-B,B)$, with $C(B)=C(1-B)=1$ and $S(B)=S(1-B)=0$. 

    Note that $(C(x),S(x))$ is the mean of a unit vector. Defining
    \begin{equation*}
        q(x):=C^2(x)+S^2(x),
    \end{equation*}
    we have $q(x)\leq1$ and $q(B)=q(1-B)=1$.
    Furthermore, it holds that
    \begin{equation}\label{eq:Lq}
        \begin{aligned}
            \mathcal{L}q(x) &=\sigma^2(x)\left((C'(x))^2+(S'(x))^2+C(x)C''(x)+S(x)S''(x)\right)\\
            &=\sigma^2(x)\left((C'(x))^2+(S'(x))^2\right)\geq0.
        \end{aligned}
    \end{equation}
    By the one-dimensional maximum principle, see e.g. \cite[Theorem 1.1.2]{protter2012maximum}, $q$ is either constant or $q'(B-)>0$ since $q(x)\leq q(B)=q(1-B)=1$. If $q(x)\equiv 1$, then $q''(x)=0$, so $C$ and $S$ are constant by \eqref{eq:Lq}, and by the boundary conditions $C(x)\equiv1$ and $S(x)\equiv0$. But then, the second expression in \eqref{eq:oci_pdes_both} states $0=-2\pi$, hence $q$ cannot be constant. Calculating $q'(B-)$ we get
    \begin{equation*}
        0<q'(B-)=2C(B-)C'(B-)+2S(B-)S'(B-)=0
    \end{equation*}
    by \eqref{eq:C'_S'_vanish}.
    Hence, we reach a contradiction to the assumption of $b_\infty(t)\equiv B$, proving the claim.    
\end{proof}

The motivating games all exhibit the same temporal component of the diffusion $\rho$. However, it should be clear from the time transformation in Section \ref{sec:motivation} that the methodology can handle more general $\rho$. In fact, we did not use any monotonicity relations for $\rho$, and similarly, we did not make use of $S$ being finite. We will consider a specific example arising from sequential hypothesis testing, following \cite{ekstrom2015bayesian}.

\begin{example}\label{ex:ekstrom}
    Let $R$ be a Gaussian random variable with mean $m$ and variance $\xi^2$, and consider the observation process
    \begin{equation*}
        Y_s=Rs+\hat{W}_s
    \end{equation*}
    where $\hat{W}$ is a Brownian motion in the full model.
    The goal is to test the hypothesis $H_0:R<0$ against $H_1:R\geq0$ upon sequential observation of $Y$. The posterior probability process is 
    \begin{equation*}
        \Pi_s = \mathbb{P}(R\geq0\mid \mathcal{G}_s) 
    \end{equation*}
    where $(\mathcal{G}_s)_{s\geq0}$ is the filtration generated by $Y$, and it can be shown, see \cite{ekstrom2015bayesian}, that $\Pi_s$ solves the stochastic differential equation 
    \begin{equation}\label{eq:Pi_ekstrom}
        d\Pi_s=\frac{\xi \varphi(\Phi^{-1}(\Pi_s))}{\sqrt{1+s\xi^2}}d W_s,
    \end{equation}
    with some initial condition, where $W$ is a Brownian motion adapted to the observation filtration $(\mathcal{G}_s)_{s\geq0}$. Thus this process falls into the scope of win-martingales with separable diffusion coefficients and 
    \begin{equation*}
        \sigma(\pi)=\varphi(\Phi^{-1}(\pi)),\qquad\rho(s)= \frac{\xi}{\sqrt{1+s\xi^2}}
    \end{equation*}
    yielding the time transformation 
    \begin{equation*}
        A(s) = \log(1+s\xi^2),\qquad \Gamma(t)=\frac{\exp(t)-1}{\xi^2},
    \end{equation*}
    and specifically $T=\infty$ if $S=\infty$. A typical objective in sequential hypothesis testing is to consider 
    \begin{equation}\label{eq:ekstrom_osp_pre_tt}
        v(s,\pi):=\inf_{\nu\in\mathcal{T}^\Pi} \mathbb{E}\left[\int_0^\nu f(s+u)\,du+g(\Pi^\pi_\nu)\right],
    \end{equation}
    where the infimum is taken over stopping times with respect to the $\Pi^\pi$ filtration such that the expectation is finite. Here $\Pi^\pi$ is the flow enlargement of \eqref{eq:Pi_ekstrom}, and $g$ is the terminal penalty stemming from a loss function as in Section \ref{sec:motivation}. 

    In \cite{ekstrom2015bayesian} the authors consider problems such as \eqref{eq:ekstrom_osp_pre_tt} for more general posterior martingales, not necessarily being on separable form but having a volatility coefficient which is decreasing in time. Furthermore, they assume a constant time cost $f(s)\equiv c_0>0$ and the hard classification problem. The authors go on to study this problem in great detail to prove well-posedness and regularity of the value function. In \cite[Proposition 4.3]{ekstrom2015bayesian} the authors prove that $v$ is jointly continuous on $[0,\infty)\times[0,1]$ and remark that they expect the result to hold if $s\mapsto f(s)$ is increasing. Their proof of joint continuity relies on temporal monotonicity of the value function $v$, and so by similar arguments as in Proposition \ref{prop:monotone} one may indeed expect their proof to carry over.
    
    However, applying the time transformation to arrive at \eqref{eq:motivation_OSP_X_shifted} with 
    \begin{equation*}
        c(t)=f\left(\frac{\exp(t)-1}{\xi^2}\right)\frac{\exp(t)}{\xi^2}
    \end{equation*}
    we can conclude that the postulated result can be extended to general $f$ such that $c$ satisfies Assumption \ref{asmp:base_penalty:c}, at least in the case of the posterior probability process in \eqref{eq:Pi_ekstrom} stemming from the Gaussian prior, under a smooth terminal penalty, such as the one induced by $L^2$--loss. 
\end{example}

\begin{example}
    If $f(s)\equiv c_0>0$ in Example \ref{ex:ekstrom}, then $t\mapsto c(t)$ is unbounded, so Assumption \ref{asmp:Lg:goodset_nonempty} does not hold for large $t$, making instantaneous stopping optimal for such times. Consider the $L^2$--loss, which gives a terminal cost of $g(x)=x(1-x)$. To obey Assumption \ref{asmp:Lg:goodset_nonempty} we may consider the problem for an initial time horizon $S<S'$ where 
    \begin{equation*}
        S'=\left(\frac{1}{2\pi c_0}-\frac{1}{\xi^2}\right)^+.
    \end{equation*}
    If $s\geq S'$, instantaneous stopping is optimal in \eqref{eq:ekstrom_osp_pre_tt} by the converse argument of Proposition \ref{prop:U_subset_C}.
    If $S'>0$,
    then $0<T<T':=\log (\xi^2/(2\pi c_0))$. For every such horizon $T$, we note that Assumption \ref{asmp:extended} holds with a horizon-dependent extension of $c(t)$ for $t>T$, specifically, $t\mapsto c_T(t)$ can be made increasing on $\mathbb{R}_+$, e.g.
    \begin{equation*}
        c_T(t)=\begin{cases}
            c(t)&t\leq T\\
            M_T-(M_T-c(T))e^{\lambda_T(T-t)}&t>T,
        \end{cases}
    \end{equation*}
    where $M_T=(1/(2\pi)+c(T))/2$ and $\lambda_T=c(T)/(M_T-c(T))$.
    For every $T<T'$, it holds that $c_T(t)\geq c_0\xi^{-2}=: \underline{c}$ for all $t\geq0$. Hence the optimal stopping boundary $t\mapsto b_T(t)$ is decreasing by Proposition \ref{prop:monotone} and $b_T(t)\leq B^*<1$ where $B^*$ is given in Theorem \ref{VerificationInfConstant} with time cost $\underline{c}$. It is thus not hard to show by monotone limit arguments, for $T\to T'$, similar to Proposition \ref{prop:finite_to_infinite_1} and the end of Proposition \ref{prop:two_boundaries} that for the infinite-horizon optimal stopping problem \eqref{eq:ekstrom_osp_pre_tt} the optimal stopping time is given by 
    \begin{equation*}
        \nu^*=\inf\{s\in[0,S'): \Pi^\pi_s\not\in(1-b_{T'}(A(s)),b_{T'}(A(s)))\}\wedge S'
    \end{equation*}
    where $t\mapsto b_{T'}(t):=\lim_{T\to T'}b_T(t)$ is decreasing with $b_{T'}(A(s))\in[\gamma(A(s)),B^*]$ for $s<S'$ and 
    \begin{equation*}
        \gamma(t):=\Phi\left(\varphi^{-1}\left(\frac{\sqrt{c_0}}{\xi}e^{t/2}\right)\right),
    \end{equation*}
    where $\varphi^{-1}$ is the inverse of $\varphi$ restricted to $[0,\infty)$.
    
    To compare, consider the following sequential inference problem studied in \cite{ekstrom2022bayesian},
    \begin{equation*}
        \hat{v}:=\inf_{\nu}\mathbb{E}\left[\ell_{L^2}(R,\hat{R}_\nu)+c_0\nu\right],
    \end{equation*}
    where $\hat{R}_s=\mathbb{E}\left[R\mid\mathcal{G}_s\right]$. In \cite{ekstrom2022bayesian} the authors show that this is a particularly simple problem, yielding a deterministic optimal stopping time given by
    \begin{equation*}
        \hat{\nu}^*=\left(\frac{1}{\sqrt{c_0}}-\frac{1}{\xi^2}\right)^+.
    \end{equation*}
    That is, when seeking a point estimate of the realization of $R$ under an $L^2$--loss with a constant time cost, the optimal stopping rule is to use the point estimate at a specific deterministic time point. On the other hand, when estimating the sign of $R$ under an $L^2$--loss with a constant time cost, the optimal stopping rule is nontrivial but occurs before a different specific deterministic time point $S'$. Notably, there is no general order in the deterministic cutoffs of the decision times for the two problems, meaning that if, say, $c_0<1/(4\pi^2)$, then $S'>\hat{\nu}^*$. However, we may of course still get $\nu^*<\hat{\nu}^*$. 
\end{example}

\section{Conclusion}\label{sec:conclusion}

This paper studied when to stop observing a posterior win-martingale and make a terminal decision. Starting from a sequential ``stop-and-declare'' problem, we reduced the Bayes risk to an optimal stopping problem for the posterior belief. For win-martingales with separable volatility, a deterministic time change removes the deterministic time factor from the dynamics and transfers it to the running cost. This gives a common framework for a general class of win-martingales such as the Aldous, Bass, and binary sequential-inference martingales, while allowing general running costs and terminal penalties.

Under our standing assumptions, we characterized the optimal stopping rule by a free boundary. In the homogeneous infinite-horizon problem, the continuation region is an interval with two constant boundaries. Under the extended assumptions, the inhomogeneous finite- and infinite-horizon problems have continuation regions with time slices of the form $(1-b_T(t),b_T(t))$, and the first exit from this region is optimal. We established that the value function is $C^1$ globally and $C^{1,2}$ away from the free boundary, proved $C^1$ regularity of $b_T$ before the terminal horizon, and derived a nonlinear integral equation that characterizes the boundary uniquely. Importantly, these do not rely on monotone properties of the temporal cost.

Although the analysis captures many settings, it deliberately assumes that the time-changed win-martingale remains in $(0,1)$ at every finite time, and therefore excludes models that can reach $\{0,1\}$ in finite time after the time change. Thus, in the present setting, the game is revealed only in the limit, or through a singular deterministic time factor at the terminal time. This leaves open the early-termination formulation considered by Guo, Howison, Possama\"i, and Reisinger \cite{guo2025randomness}, where the win-probability martingale may be absorbed at $0$ or $1$ strictly before the prescribed horizon. This includes cases such as the Wright--Fisher martingale and a Brownian motion absorbed at $\{0,1\}$. In such models, there are two distinct stopping mechanisms. Namely, these are the decision maker's voluntary stopping time and the exogenous termination time
\begin{equation*}
    \zeta:=\inf\{s\ge0:\Pi_s\in\{0,1\}\}.
\end{equation*}
Extending the present theory to this setting would require incorporating absorbing endpoints that are reached in finite time and understanding how the interaction between voluntary stopping and exogenous game termination changes the geometry and regularity of the optimal boundary. We leave this case for future work.

\bibliographystyle{abbrv}
\bibliography{main.bib}

\end{document}